\documentclass[12pt,oneside]{amsart}

\usepackage{graphicx}
\usepackage{amsfonts}
\usepackage{epsf}
\usepackage{amssymb}
\usepackage{amsmath}
\usepackage{amscd}
\usepackage{tikz}
\usepackage{pdfpages}
\usepackage{fancyhdr}
\usepackage{setspace}
\usepackage{hyperref}
\usepackage[all]{xy}
\usetikzlibrary{matrix}
\usepackage{verbatim}

\theoremstyle{definition}
\newtheorem{theorem}{Theorem}[section]
\newtheorem{proposition}[theorem]{Proposition}
\newtheorem{lemma}[theorem]{Lemma}
\newtheorem{definition}[theorem]{Definition}
\newtheorem{question}[theorem]{Question}
\newtheorem{corollary}[theorem]{Corollary}
\newtheorem{conjecture}[theorem]{Conjecture}
\newtheorem{remark}[theorem]{Remark}

\newtheorem*{UC}{The Unknotting Conjecture}

\newcommand{\Z}{\mathbb{Z}}

\newcommand{\R}{\mathbb{R}}
\newcommand{\RP}{\mathbb{RP}}
\newcommand{\CP}{\mathbb{CP}}

\newcommand{\U}{\mathcal{U}}

\newcommand{\bi}{\begin{itemize}}
\newcommand{\ei}{\end{itemize}}
\newcommand{\be}{\begin{enumerate}}
\newcommand{\ee}{\end{enumerate}}

\newcommand{\n}{\beta}

\newcommand{\emp}{\emptyset}
\newcommand{\X}{\times}

\newcommand{\eps}{\epsilon}

\newcommand{\A}{\alpha}
\newcommand{\pd}{\partial}

\newcommand{\T}{\mathcal{T}}
\newcommand{\Ss}{\mathcal{S}}
\newcommand{\D}{\mathcal{D}}
\newcommand{\Pau}{\mathcal{P}}
\newcommand{\K}{\mathcal{K}}

\newcommand{\g}{\gamma}

\newcommand{\ups}{\upsilon}
\newcommand{\BB}{\mathcal{B}}
\newcommand{\mrk}{\text{mrk}}

\makeatletter
\newtheorem*{rep@theorem}{\rep@title}
\newcommand{\newreptheorem}[2]{%
\newenvironment{rep#1}[1]{%
 \def\rep@title{#2 \ref{##1}}%
 \begin{rep@theorem}}%
 {\end{rep@theorem}}}
\makeatother

\newreptheorem{theorem}{Theorem}
\newreptheorem{lemma}{Lemma}
\newreptheorem{question}{Question}
\newreptheorem{corollary}{Corollary}

\begin{document}

\rhead{\thepage}
\lhead{\author}
\thispagestyle{empty}


\raggedbottom
\pagenumbering{arabic}
\setcounter{section}{0}


\title{Bridge trisections of knotted surfaces in $S^4$}
\date{\today}

\author{Jeffrey Meier}
\address{Department of Mathematics, Indiana University, 
Bloomington, IN 47408}
\email{jlmeier@indiana.edu}
\urladdr{http://pages.iu.edu/~jlmeier/} 

\author{Alexander Zupan}
\address{Department of Mathematics, The University of Texas at Austin, Austin, TX 78712}
\email{zupan@math.utexas.edu}
\urladdr{http://math.utexas.edu/users/zupan}

\begin{abstract}

We introduce bridge trisections of knotted surfaces in the four-sphere.  This description is inspired by the work of Gay and Kirby on trisections of four-manifolds and extends the classical concept of bridge splittings of links in the three-sphere to four dimensions.  We prove that every knotted surface in the four-sphere admits a bridge trisection (a decomposition into three simple pieces) and that any two bridge trisections for a fixed surface are related by a sequence of stabilizations and destabilizations.  We also introduce a corresponding diagrammatic representation of knotted surfaces and describe a set of moves that suffice to pass between two diagrams for the same surface.  Using these decompositions, we define a new complexity measure: the bridge number of a knotted surface.  In addition, we classify bridge trisections with low complexity, we relate bridge trisections to the fundamental groups of knotted surface complements, and we prove that there exist knotted surfaces with arbitrarily large bridge number.

\end{abstract}

\maketitle
\section{Introduction}\label{sec:intro}

The bridge number of a link in $S^3$ was first introduced by Horst Schubert \cite{schubert} in 1954, and in the past sixty years, it has become clear that this invariant is an effective measure of topological complexity.

Moreover, in the last several decades a significant body of work has revealed deep connections between bridge decompositions of links and Heegaard splittings of closed 3--manifolds.  Superficially, the analogy between the two theories is clear: A bridge decomposition of a link $L$ in $S^3$ is a splitting of $(S^3, L)$ into the union of two trivial tangles, while a Heegaard splitting of a closed 3--manifold $Y$ is a description of $Y$ as the union of two handlebodies. A good motto in dimension three is that any technique for studying Heegaard splittings of 3--manifolds can be adapted to produce an analogous technique for studying bridge decompositions of knots and links in $S^3$ (or in other manifolds). In the present article, we push this analogy up one dimension.

While a number of foundational questions in three-manifold topology have now been resolved, it remains a topic of great interest to develop means of applying 3--dimensional techniques to 4--dimensional objects.  Recently, Gay and Kirby \cite{gay-kirby:trisections} introduced trisections of smooth, closed 4--manifolds as the appropriate 4--dimensional analogue of Heegaard splittings.  A trisection of a closed 4--manifold $X$ is a description of $X$ as the union of \emph{three} 4--dimensional 1--handlebodies with restrictions placed on how the three pieces intersect in $X$.  Gay and Kirby prove that every $X$ admits a trisection, and any two trisections of $X$ are related by natural stabilization and destabilization operations in a 4--dimensional version of the Reidemeister-Singer Theorem for Heegaard splittings of 3--manifolds.  As evidence of the utility of trisections to bridge the gap between 3-- and 4--manifolds, the authors have shown in previous work \cite{mz:genus2} that 3--dimensional tools suffice to classify those 4--manifolds which admit a trisection of genus two.

For the precise details pertaining to trisections of 4--manifolds, we refer the reader to \cite{gay-kirby:trisections} (see also Subsection \ref{subsec:branched} below). Here, we simly recall the example of the genus zero trisection of $S^4$, which decomposes the 4--sphere into three 4--balls. Note that all objects are assumed throughout to be PL or smooth unless otherwise stated.

\begin{definition}
	The \emph{0--trisection} of $S^4$ is a decomposition $S^4=X_1\cup X_2\cup X_3$, such that
\be
\item $X_i$ is a 4--ball,
\item $B_{ij}=X_i\cap X_j= \partial X_i\cap \partial X_j$ is a 3--ball, and
\item $\Sigma = X_1\cap X_2\cap X_3 = B_{12}\cap B_{23}\cap B_{31}$ is a 2--sphere.
\ee
\end{definition}

The goal of the present paper is to adapt the theory of trisections to the study of knotted surfaces in $S^4$ in order to describe a 4--dimensional analogue to bridge decompositions of knots in $S^3$.  A \emph{knotted surface} $\K$ in $S^4$ is a smoothly embedded, closed surface, which may be disconnected or non-orientable.  When $\K$ is homeomorphic to $S^2$, we call $\K$ a 2--knot.  
	A \emph{trivial $c$--disk system} is a pair $(X,\D)$ where $X$ is a 4--ball and $\D\subset X$ is a collection of $c$ properly embedded disks $\D$ which are simultaneously isotopic into the boundary of the 4--ball $X$.
	
\begin{definition}
	A \emph{$(b;c_1,c_2,c_3)$--bridge trisection} $\T$ of a knotted surface $\K \subset S^4$ is a decomposition of the form $(S^4,\K)=(X_1, \D_1)\cup(X_2,\D_2)\cup(X_3,\D_3)$ such that
\be
\item $S^4=X_1\cup X_2\cup X_3$ is the standard genus zero trisection of $S^4$,
\item $(X_i,\D_i)$ is a trivial $c_i$--disk system, and
\item $(B_{ij},\A_{ij}) = (X_i,\D_i) \cap (X_j,\D_j)$ is a $b$--strand trivial tangle.
\ee
When appropriate, we simply refer to $\T$ as a $b$--bridge trisection.  If $c_i=c$ for all $i$, then we call $\T$ \emph{balanced}, and we say that $\K$ admits a $(b,c)$--bridge trisection.
\end{definition}

Several properties follow immediately from this definition.  First, $L_i=\partial\D_i$ is a $c_i$--component unlink in $Y_i = \pd X_i\cong S^3$, and $(Y_i, L_i) = (B_{ij},\A_{ij}) \cup_{\Sigma} (B_{ki},\A_{ki})$ is a $b$--bridge decomposition.  It follows that $b\geq c_i$ for each $i$. Next, it is straightforward to check that $\chi(K)=c_1+c_2+c_3-b$; thus, the topological type of $\K$ depends only on $b$ and the $c_i$.

Our first result is an existence theorem for bridge trisections, which we prove in Section \ref{sec:exist} using a structure we call a banded bridge splitting.

\begin{theorem}\label{thm:existence}
Every knotted surface $\K$ in $S^4$ admits a bridge trisection.
\end{theorem}

In addition, bridge trisections give rise to a new diagrammatic presentation for knotted surfaces.  A diagram for a tangle $(B,\A)$ is a generic projection of $\A$ to a disk $E$ together with crossing information at each double point of the projection, and any two tangle diagrams with the same number of strands can be glued together to get a classical link diagram.  We define a \emph{tri-plane diagram} $\Pau$ to be a triple of $b$--strand trivial tangle diagrams $(\Pau_{12},\Pau_{23},\Pau_{31})$ having the property that $\Pau_{ij} \cup \overline{\Pau_{ki}}$ is a diagram for an unlink $L_i$, where $\overline{\Pau_{ki}}$ denotes the mirror image of $\Pau_{ki}$.

We naturally obtain a tri-plane diagram from a bridge trisection by an appropriate projection, and conversely, every tri-plane diagram $\Pau$ gives rise to a bridge trisection of a knotted surface $\K(\Pau)$ in a prescribed way.  Details are supplied in Section~\ref{sec:prelims}, and the following corollary is immediate.

\begin{corollary}\label{coro:tri-plane-exist}
For every knotted surface $\K$ in $S^4$, there exists a tri-plane diagram $\Pau$ such that $\K=\K(\Pau)$.
\end{corollary}

A simple example of a nontrivial tri-plane diagram is shown in Figure \ref{fig:SpunTrefoilDiag}.  The diagram describes the spun trefoil, a knotted 2--sphere obtained from the trefoil knot.  Spun knots and twist spun knots provide us with many interesting examples of nontrivial bridge trisections and tri-plane diagrams and are explored in Section \ref{sec:exs}.

\begin{figure}[h!]
\centering
\includegraphics[scale = .4]{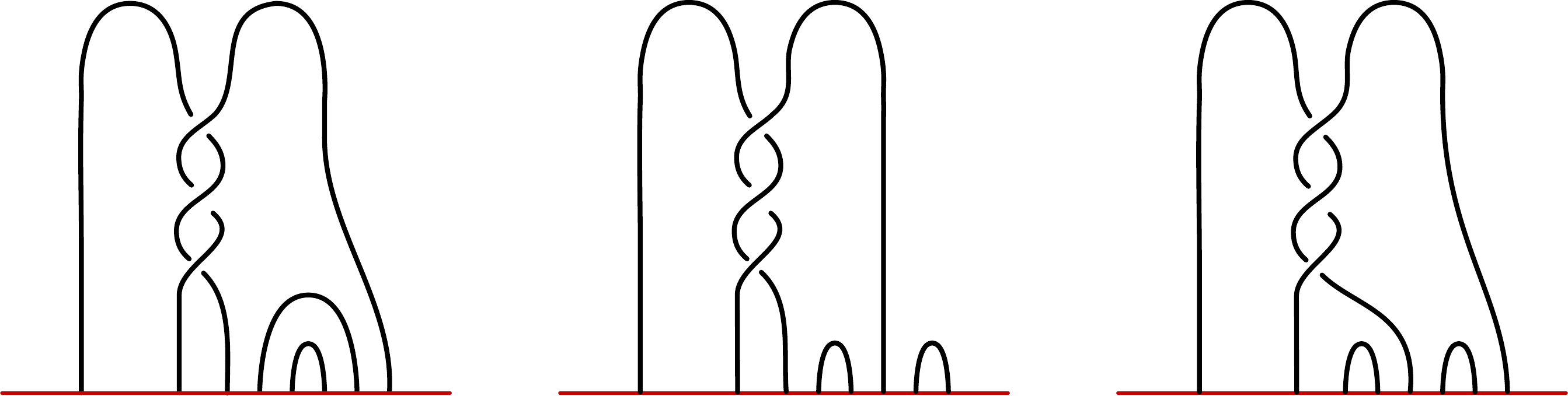}
\caption{A $(4,2)$--bridge tri-plane diagram for the spun trefoil.}
\label{fig:SpunTrefoilDiag}
\end{figure}

\begin{remark}
There are several other existing diagrammatic theories for knotted surfaces in $S^4$.  The interested reader may wish to investigate the immersed surface diagrams in $S^3$ studied by Carter-Saito \cite{CS} and Roseman \cite{roseman}, the braid presentations studied by Kamada \cite{kamada:braid}, and the planar diagrams known as \emph{$ch$-diagrams} studied by Yoshikawa \cite{yoshikawa}.  We recommend \cite{CKS} for a general overview.
\end{remark}

As with Heegaard splittings, classical bridge decompositions, and trisections of 4--manifolds, there is a natural stabilization operation associated to bridge trisections of knotted surfaces.  We describe this operation in detail in Section \ref{sec:stab}, where we relate this stabilization move to a similar operation on banded bridge splittings.  In Section~\ref{sec:unique}, we use this correspondence and prove the uniqueness result below.  We consider two trisections $\T$ and $\T'$ of a knotted surface $\K$ in $S^4$ to be \emph{equivalent} if there is a smooth isotopy of $(S^4,\K)$ carrying the components of $\T$ to the corresponding components of $\T'$. (A more detailed description of equivalence can be found in Section \ref{sec:unique}.)

\begin{theorem}\label{thm:uniqueness}
Any two bridge trisections of a given pair $(S^4,\K)$ become equivalent after a sequence of stabilizations and destabilizations.
\end{theorem}

By interpreting the stabilization operation diagrammatically, we prove that there is a set of diagrammatic moves, called \emph{tri-plane moves}, that suffice to pass between any two tri-plane diagrams of a given knotted surface.  The collection of moves is described in Subsections \ref{subsec:intromoves} and \ref{subsec:stab}.

\begin{theorem}\label{thm:tri-planemoves}
Any two tri-plane diagrams for a given knotted surface are related by a finite sequence of tri-plane moves.
\end{theorem}

For any knotted surface $\K$ in $S^4$, let $b(\K)$ denote the \emph{bridge number} of $\K$, where
$$b(\K) = \min\{b\,|\, \text{$\K$ admits a $b$--bridge trisection}\}.$$
A natural first goal is to understand surfaces of low bridge number, the collection of which we expect to include unknotted surfaces.

An orientable surface $\K$ in $S^4$ is said to be \emph{unknotted} if it bounds a handlebody~\cite{hoso-kawa}. A precise characterization of unknotted non-orientable surfaces is given in Section \ref{sec:class}, where we give a simple argument that surfaces admitting 1-- and 2--bridge trisections are unknotted and that the trisections are standard.

Given a knotted surface $\K$ in $S^4$, we can consider the double cover $X(\K)$ of $S^4$ branched over $\K$.  A $(b;c_1,c_2,c_3)$--bridge trisection of $\K$ gives rise to a $(b-1;c_1-1,c_2-1,c_3-1)$--trisection of the closed 4--manifold $X(\K)$.  Theorems in \cite{msz} and \cite{mz:genus2} classify balanced and unbalanced genus two trisections of 4--manifolds, respectively.  In particular, every genus two trisection is standard, and we obtain the following result as a corollary. 

\begin{theorem}\label{thm:standard}
Every knotted surface $\K$ with $b(\K)\leq 3$ is unknotted and any bridge trisection of $\K$ is standard.
\end{theorem}

More generally, we may consider the collection of all bridge trisections of the unknotted 2--sphere $\U \subset S^4$.  Theorem \ref{thm:standard} and work of Otal on 3--dimensional bridge splittings \cite{otal} motivate the following question.

\begin{question}\label{question:unknot}
	Is every $b$--bridge trisection of $\U$ standard?  Equivalently, is every $b$--bridge trisection of $\U$ with $b>1$ stabilized?
\end{question}

In contrast to the case $b(\K) \leq 3$, in Section \ref{sec:exs} we describe bridge trisections for certain classes of knotted surfaces, including spun knots and twist-spun knots.  From this it follows that there are infinitely many distinct 2--knots admitting $(4,2)$--bridge trisections.  In fact, we prove the following.

\begin{theorem}\label{thm:nontrivial}
	Let $b\geq 2$.  There exist infinitely many distinct 2--knots with bridge number $3b-2$.
\end{theorem}

The 2--knots constructed in Theorem \ref{thm:nontrivial} have balanced $(3b-2,b)$--bridge trisections and are formed by applying the spinning operation to torus knots.  See Section~\ref{sec:exs} for details.
 
Turning to questions about the knot groups, one of the most interesting conjectures in the study of knotted surfaces is the following.

\begin{UC}\label{conj:UC}
	A knotted surface is unknotted if and only if the fundamental group of the surface exterior is cyclic.	
\end{UC}

The Unknotting Conjecture is known to be true in the topological category for orientable surfaces \cite{freedman:disk,hillman-kawauchi,kawauchi:revised} and for projective planes \cite{lawson}. On the other hand, there are certain higher genus nonorientable counterexamples in the smooth category \cite{fkv88,fkv87}.  For a knotted surface $\K$ equipped with a bridge trisection, we have the next result regarding the fundamental group of the complement of $\K$.

\begin{proposition}
	Let $\K$ be a knotted surface in $S^4$ admitting a $(b;c_1,c_2,c_3)$--bridge trisection. Then, $\pi_1(S^4\setminus\K)$ has a presentation with $c_i$ generators and $b-c_j$ relations for any choice of distinct $i,j\in\{1,2,3\}$.
\end{proposition}

It follows that $(b;c_1,c_2,1)$--surfaces have complements with cyclic fundamental group.  Moreover, by the topological solutions to the Unknotting Conjecture referenced above, we have the next corollary.
\begin{corollary}\label{coro:unknotted}
	Every orientable $(b;c_1,c_2,1)$--surface is topologically unknotted.
\end{corollary}

Note that the adjective ``orientable'' is important in Corollary \ref{coro:unknotted}, since the Unknotting Conjecture is still open for general non-orientable surface knots.

\subsection*{Organization}\ 

We begin in Section \ref{sec:prelims} by discussing some classical aspects of bridge splittings in dimension three, after which we introduce bridge trisections, tri-plane diagrams, and tri-plane moves in detail and discuss how the branched double cover provides a connection with trisections.  In Section \ref{sec:exist} we prove the existence of bridge trisections and discuss the auxiliary object: banded bridge splittings.  In Section \ref{sec:class}, we give a classification of certain types of bridge trisections, including those up to bridge number three.  In Section \ref{sec:exs}, we describe the spinning and twist-spinning constructions and use them to produce knotted surfaces with arbitrarily large bridge number.  In Section \ref{sec:stab}, we give a detailed discussion of the stabilization and destabilization operation, and in Section \ref{sec:unique}, we prove that any two bridge trisections of a fixed surface have a common stabilization.  

\subsection*{Acknowledgements}\ 

This work benefited greatly from the interest, support, and insight of David Gay, for which the authors are very grateful.  Thanks is also due to Ken Baker, Scott Carter, Cameron Gordon, and Chuck Livingston for many helpful and interesting conversations.
The first author was supported by NSF grant DMS-1400543.  The second author was supported by NSF grant DMS-1203988.

\section{Preliminary topics}\label{sec:prelims}

We will assume that all manifolds are smooth and compact unless otherwise specified, and all 3-- and 4--manifolds are orientable.  We let $\nu( \cdot )$ denote an open regular neighborhood in the appropriate ambient manifold.

For $n=3$ or $4$, a collection $\D$ of properly embedded $(n-2)$--balls in an $n$--ball $X$ is called  \emph{trivial} if all disks are simultaneously isotopic into $\pd X$.  Equivalently, there is a Morse function $h\colon (X,\D) \rightarrow (-\infty,0]$ such that $h_X$ has one index zero critical point, $h^{-1}(0) = \pd X$, and each $(n-2)$--ball in $\D$ contains exactly one index zero critical point of $h_{\D}$.  When $n=3$, we call this pair a \emph{$b$--strand trivial tangle} and denote it $(B,\A)$, where $b = |\A|$.  In the case that $n=4$, we call the pair a \emph{trivial $c$--disk system}, where $c = |\D|$.

\subsection{Bridge splittings in dimension three}\label{sub:bridge3}\ 

Suppose $(B,\A)$ is a trivial tangle.  For each arc $a \in \A$, there is an embedded disk $\Delta_a$ such that $\Delta_a \cap \A = a$ and $\Delta_a \cap \pd B$ is an arc $a^*$ such that $\pd \Delta_a = a \cup a^*$.  We call $\Delta_a$ a \emph{bridge disk}, and we call the arc $a^* \subset \pd B$ a \emph{shadow} of the arc $a$.  Note that a given arc $a$ may have infinitely many different shadows given by infinitely many distinct isotopy classes of bridge disks.  We can always choose a collection $\Delta$ of pairwise disjoint bridge disks for $\A$.

For a link $L \subset S^3$, a \emph{$b$--bridge splitting} of $L$ is a decomposition
\[ (S^3,L) = (B_1,\A_1) \cup_{\Sigma} (B_2,\A_2)\]
such that $(B_i,\A_i)$ is a $b$--strand trivial tangle for $i=1,2$.  The surface $\Sigma$ is called a \emph{$b$--bridge sphere}.  We let $\Sigma_L$ denote $\Sigma \setminus \nu(L)$, and we consider two bridge surfaces $\Sigma$ and $\Sigma'$ to be \emph{equivalent} if $\Sigma_L$ is isotopic to $\Sigma_L'$ in $E(L) = S^3 \setminus \nu(L)$ (in other words, if $\Sigma$ is isotopic to $\Sigma'$ via an isotopy fixing $L$).  It is useful to note that for a bridge splitting of $(S^3,L)$, there is a Morse function $g\colon(S^3,L) \rightarrow \R$ such that $g_{S^3}$ has two critical points, all minima of $g_L$ occur below all maxima of $g_L$, and any level surface which separates the minima from the maxima of $g_L$ is a bridge sphere equivalent to $\Sigma$.

At any point of $L \cap \Sigma$, we may isotope $g$ to introduce an additional pair of canceling critical points for $g_L$, resulting in a new Morse function $g'$ and $(b+1)$--bridge sphere $\Sigma'$.  We call $\Sigma'$ \emph{perturbed} and say that $\Sigma'$ is an \emph{elementary perturbation} of $\Sigma$.  The reverse operation is called \emph{unperturbation}. The bridge sphere $\Sigma'$ is perturbed if and only if there is a pair of bridge disks $\Delta_1$ and $\Delta_2$ on opposite sides of $\Sigma'$ such that $\Delta_1 \cap \Delta_2$ is a single point contained in $L$.  Equivalently, $\Sigma'$ is perturbed if and only if there are arcs $a_1 \in \A_1$ and $a_2 \in \A_2$ with shadows $a_1^*$ and $a_2^*$ such that $a_1^* \cup a_2^*$ is an embedded arc in $\Sigma'$.  Lastly, if $\Sigma^*$ is obtained by a sequence elementary perturbations performed on $\Sigma$, we call $\Sigma^*$ a \emph{perturbation} of $\Sigma$.  Note that elementary perturbations are not unique; perturbing $\Sigma$ at two different points of $L \cap \Sigma$ may induce two distinct $(b+1)$--bridge spheres.  However, if $J$ is a component of $L$, then perturbations about each point of $J \cap \Sigma$ yield equivalent bridge spheres.

As might be expected, the structure of the collection of all bridge spheres for the unknot is rather simple; this is made precise by the next theorem.

\begin{theorem}\label{unknot}\cite{otal}
Every bridge sphere for the unknot is a perturbation of the standard 1--bridge sphere.
\end{theorem}

We say a $b$--bridge sphere $\Sigma$ for a link is \emph{reducible} if there is an essential curve $\gamma \subset \Sigma_L$ which bounds disks $D_1$ and $D_2$ in $B_1 \setminus \A_1$ and $B_2 \setminus \A_2$, respectively.  In this case, $L = L_1 \cup L_2$ is a split link, and $\Sigma = \Sigma_1 \# \Sigma_2$, where $\Sigma_i$ is a $b_i$--bridge sphere for $L_i$ with $b_1 + b_2  = b$.

\begin{theorem}\label{split}\cite{bachman-schleimer}
Every bridge sphere for a split link $L$ is reducible.
\end{theorem}

Combining Theorems \ref{unknot} and \ref{split}, we have the following result.

\begin{proposition}\label{linkbridge}
Every bridge sphere for the $n$--component unlink is a perturbation of the standard $n$--bridge sphere.
\begin{proof}
Suppose $\Sigma$ is a $b$--bridge sphere for the $n$--component unlink $L = L_1 \cup \dots \cup L_n$.  By repeated applications of Theorem \ref{split}, we may write $\Sigma = \Sigma_1 \# \dots \# \Sigma_n$, where $\Sigma_i$ is a  $b_i$--bridge surface for the unknot $L_i$.  If $b_i = 1$ for all $i$, then $\Sigma$ is the standard $n$--bridge sphere, and the statement holds vacuously.  Otherwise, $b_i > 1$ for some $i$, in which case $\Sigma_i$, and thus $\Sigma$, is perturbed by Theorem \ref{unknot}.
\end{proof}
\end{proposition}

Note that while Theorem \ref{unknot} implies that the unknot has a unique $b$--bridge sphere for every $b$, Proposition \ref{linkbridge} \emph{does not} imply the same is true for an unlink.  For instance, a 2--component unlink has two inequivalent 3--bridge splittings (corresponding to the number of bridges contained in each component).

\subsection{Extending bridge splittings to dimension four}\ 

Here we adapt the notion of a bridge splitting to a knotted surface $\K$ in $S^4$.  Na\"ively, we may attempt to write $(S^4,\K)$ as the union of two trivial disk systems.  However, such a decomposition is severely limiting, as is implied by the following standard proposition.

\begin{proposition}\label{prop:trivialdisks}\cite{kamada:braid}
If $X$ is a 4--ball containing collections $\D_1$ and $\D_2$ of trivial disks such that $\pd \D_1 = \pd \D_2$, then $\D_1$ is isotopic (rel boundary) to $\D_2$ in $X$.
\end{proposition}

In other words, a trivial disk system $(X,\D)$ is determined up to isotopy by the unlink $L=\pd \D$ in $\pd X = S^3$.  Thus, if $(S^4,\K)$ can be decomposed into two trivial disk systems, then $(S^4,\K)$ is the double of a single trivial disk system, and as such $\K$ is an unlink.  We rectify the situation by decomposing $(S^4,\K)$ into \emph{three} trivial disk systems as discussed in the introduction.  Recall that a $b$--bridge trisection $\T$ of $(S^4,\K)$ is a decomposition
\[ (S^4,\K) = (X_1,\D_1) \cup (X_2,\D_2) \cup (X_3,\D_3),\]
where
\be
\item $(X_i,\D_i)$ is a $c_i$--disk trivial system,
\item $(B_{ij},\A_{ij}) = (X_i,\D_i) \cap (X_j,\D_j)$ is a $b$--strand trivial tangle, and
\item $(\Sigma,\mathbf p) = (X_1,\D_1) \cap (X_2,\D_2) \cap (X_3,\D_3)$ is a 2--sphere $\Sigma$ containing a set of $\mathbf p$ of $2b$ points.
\ee
We call the subset $\Ss = (B_{12},\A_{12}) \cup (B_{23},\A_{23}) \cup (B_{31},\A_{31})$ the \emph{spine} of the bridge trisection, and we say that two bridge trisections are \emph{equivalent} if their spines are smoothly isotopic.  Observe that $(\pd X_i,\pd \D_i) = (B_{ij},\A_{ij}) \cup_{\Sigma} (B_{ki},\A_{ki})$ is a $b$--bridge splitting of the unlink $L_i=\pd \D_i$; hence, Proposition \ref{prop:trivialdisks} implies the following fact.

\begin{lemma}\label{spine}
A bridge trisection $\T$ is uniquely determined by its spine $\Ss$.
\end{lemma}

Next, we discuss connected and boundary-connected summation.  For a pair $M_1$ and $M_2$ of $n$-manifolds, the \emph{connected sum} $M_1 \# M_2$ is constructed by removing a neighborhoods of $\nu(p_l)$ of a point $p_l \in M_l$ and identifying the boundary of $M_1 \setminus \nu(p_1)$ with the boundary of $M_2 \setminus \nu(p_2)$.  If $M_1$ and $M_2$ have nonempty boundary, we may form the \emph{boundary-connected sum} $M_1 \natural M_2$ by a similar construction using points $q_l \subset \pd M_l$. 

Given two knotted surfaces $\K_1$ and $\K_2$ in $S^4$, we form the \emph{connected sum} $\K_1 \# \K_2$ by choosing points $p_1 \in \K_1$ and $p_2 \in \K_2$, removing a small neighborhoods of $p_1$ and $p_2$, and gluing $(S^4 \setminus \nu(p_1),\K_1 \setminus \nu(p_1))$ to $(S^4 \setminus \nu(p_2),\K_2 \setminus \nu(p_2))$ along their boundaries.  This operation is independent of the choices of $p_1$ and $p_2$ provided that $\K_1$ and $\K_2$ are connected. 

If $\K$ can be expressed as $\K_1 \# \K_2$, then there is a smoothly embedded 3--sphere which cuts $S^4$ into two 4--balls and meets $\K$ in a single unknotted curve.  We call such an $S^3$ a \emph{decomposing sphere}.  The pairs $(S^4, \K_1)$ and $(S^4, \K_2)$ can be recovered by cutting along the decomposing sphere, and capping off the resulting manifold pairs with copies of the standard trivial 1--disk system $(B^4, D^2)$.  We say that a decomposing sphere is \emph{nontrivial} if neither $\K_1$ nor $\K_2$ is an unknotted surface in $S^4$. 

Now, suppose that for $l=1,2$ the surface $\K_l$ is equipped with a $b_l$--bridge trisection $\T_l$ given by
$$(S^4,\K_l) = (X^l_1,\D^l_1) \cup (X^l_2,\D^l_2) \cup (X^l_3,\D^l_3),$$
with $(B^l_{ij},\A^l_{ij}) = (X^l_i,\D^l_i) \cap (X^l_j,\D^l_j)$ and trisection sphere $\Sigma_l$.  To construct the connected sum of $\K_1$ and $\K_2$, we may choose points  $p_l \in \K_l$.  As such, each point $p_l$ has a standard trisected regular neighborhood $\nu(p_l)$, so that $\pd(\Sigma_1\setminus \nu(p_1))$ and $\pd(\Sigma_2\setminus \nu(p_2))$ are identified, as are $\pd(B^1_{ij}\setminus \nu(p_1))$ and $\pd(B_{ij}^2\setminus \nu(p_2))$ for each pair of indices. 
 
We leave it as an exercise for the reader to verify that the result is a $(b_1 + b_2 - 1)$--bridge trisection, which we denote $\T_1\#\T_2$.  This new bridge trisection is given by the following decomposition of $\K=\K_1\#\K_2$:
$$(S^4,\K) = (X_1,\D_1) \cup (X_2,\D_2) \cup (X_3,\D_3),$$
where
\be
\item $(X_i,\D_i)=(X^1_i,\D^1_i)\natural(X^2_i,\D^2_i)$, and
\item $(B_{ij},\A_{ij}) = (X_i,\D_i) \cap (X_j,\D_j)=(B^1_{ij},\A^1_{ij})\natural(B^2_{ij},\A^2_{ij})$.
\ee

Notice that the result $\K=\K_1\#\K_2$ of the connected summation does not depend on the choices of points $p_1$ and $p_2$ up to the connected components of $\K_1$ and $\K_2$ containing each point, but the resulting bridge trisection $\T=\T_1\#T_2$ often will.

\subsection{Tri-plane diagrams}\label{subsec:introdiags}\ 

We may further reduce the information needed to generate any bridge trisection by projecting the arcs $\A_{ij}$ onto an embedded 2--complex.  Consider a $b$--bridge trisection of a knotted surface $\K$ labeled as above, and for each pair of indices let $E_{ij} \subset B_{ij}$ be an embedded disk with the property that $e = \pd E_{12} = \pd E_{23} = \pd E_{31}$.  We call the union $E_{12} \cup E_{23} \cup E_{31}$ a \emph{tri-plane}.

Suppose the points $\mathbf p = \K \cap \Sigma$ lie in the curve $e = E_{12} \cap E_{23} \cap E_{31}$.  We assign each $E_{ij}$ a normal vector field in $B_{ij}$ such that all three vector fields induce a consistent orientation on their common boundary curve $e$.  The knotted surface $\K$ intersects each 3--ball $B_{ij}$ in a $b$--strand trivial tangle $\A_{ij}$, and this triple of tangles can be projected onto the tri-plane to yield an immersed collection of arcs; that is, a 4--valent graph with boundary in $e$.  By viewing each projection from the perspective of the normal vector field, we can assign crossing information at each double point of our projection, and we obtain a triple of planar tangle diagrams $\Pau= (\Pau_{12}, \Pau_{23},\Pau_{31})$ with the property that for $\{i,j,k\} = \{1,2,3\}$, if $\overline{\Pau_{ki}}$ is the mirror image of $\Pau_{ki}$, then $\Pau_{ij} \cup \overline{\Pau_{ki}}$ is a classical link diagram for the unlink $\pd \D_i$ of $c_i$ components in the plane $E_{ij} \cup E_{ki}$. 

We call any triple $\Pau = (\Pau_{12},\Pau_{23},\Pau_{31})$ of planar diagrams for $b$--strand trivial tangles having the property that $\Pau_{ij} \cup \overline{\Pau_{ki}}$ is a diagram for an $c_i$--component unlink a \emph{$(b;c_1,c_2,c_3)$--bridge tri-plane diagram}.  Given a tri-plane diagram $\Pau$, we can build a smoothly embedded surface $\K(\Pau)$ in $S^4$ as follows:  The triple of diagrams $(\Pau_{12}, \Pau_{23},\Pau_{31})$ uniquely describes three trivial tangles $(B_{ij},\A_{ij})$ as well as a pairwise gluing of these tangles along their common boundary. Each union $(B_{ij},\A_{ij})\cup (B_{ki},\A_{ki})$ is an unlink $L_i$ in $S^3$, and by Proposition \ref{prop:trivialdisks}, we can cap off $L_i$ uniqely with a trivial disk system $(X_i,\D_i)$.  The result is an embedded surface $\K(\Pau)$ in $S^4$ that is naturally trisected:
$$(S^4, \K(\Pau))=(X_1,\D_1)\cup(X_2,\D_2)\cup(X_3,\D_3).$$
In short, the tri-plane diagram $\Pau$ determines the spine $\Ss = (B_{12},\A_{12}) \cup (B_{23},\A_{23}) \cup (B_{31},\A_{31})$ of the bridge trisected surface $\K(\Pau)$.

\begin{remark}
Technically, the union in the preceding paragraph should be written $(B_{ij},\A_{ij})\cup \overline{(B_{ki},\A_{ki})}$.  More precisely, we might suppose $(B_{i,j},\A_{i,j})$ inherits its orientation as a component of $\pd X_i$.  Thus, the orientation of $(B_{ki},\A_{ki})$ in $\pd X_i$ is opposite that which it inherits from $\pd X_k$.  In practice, however, this mirroring is only evident when we are working with the tri-planes diagrams $\Pau$; hence, we will suppress the mirror image notation except when discussing these diagrams.
\end{remark}

\subsection{Two simple examples}\label{subsec:introexs}\ 

To guide the intuition of the reader, we present two depictions of low-complexity trisections of unknotted 2--spheres in $S^4$.  For our illustrations, we consider $S^4$ as the unit sphere $\{(x_1,\dots,x_5): x_1^2 + \dots + x_5^2 = 1\}$ in $\R^5$.  Let $Y = \{(x_1,\dots,x_5) \in S^4: x_5=0\}$, so that $Y \cong S^3$, and let $\pi\colon\R^5 \rightarrow \R^2$ denote projection to the $x_1x_2$--plane.  The 0--trisection of $S^4$ is simply a lift of the obvious trisection of the unit disk $D = \pi(S^4)$ pictured below.

\begin{figure}[h!]
  \centering
    \includegraphics[width=.25\textwidth]{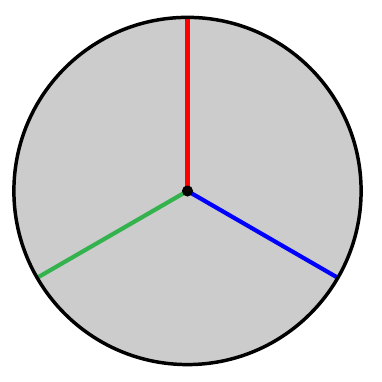}
    \caption{The standard trisection of the unit disk, which lifts to the standard trisection of $S^4$.}
    \label{disktrisection}
\end{figure}

In addition, if $B_{12} \cup B_{23} \cup B_{31}$ is a spine for this standard trisection, then the intersection of this spine with $Y$ is the union of three disks $E_{12} \cup E_{23} \cup E_{31}$.  Now, if $\K \subset S^4$ is an unknotted 2--sphere, then $\Ss$ is isotopic into $Y$, so that $\Ss=\K \cap B_{ij} = \K \cap E_{ij}$.  As such, we may construct a trisection of $\K$ by putting it into a nice position relative to the tri-plane in $Y$.

In Figures \ref{fig:1Bridge} and \ref{fig:2Bridge} below, we depict this situation in $\R^3$ by removing a point in $e = E_{12} \cap E_{23} \cap E_{31}$.  Figures \ref{fig:1Bridge} and \ref{fig:2Bridge} show one-bridge and two-bridge trisections (respectively) of an unknotted 2-sphere along with the associated tri-plane diagrams.

\begin{figure}[h!]
  \centering
    \includegraphics[width=.9\textwidth]{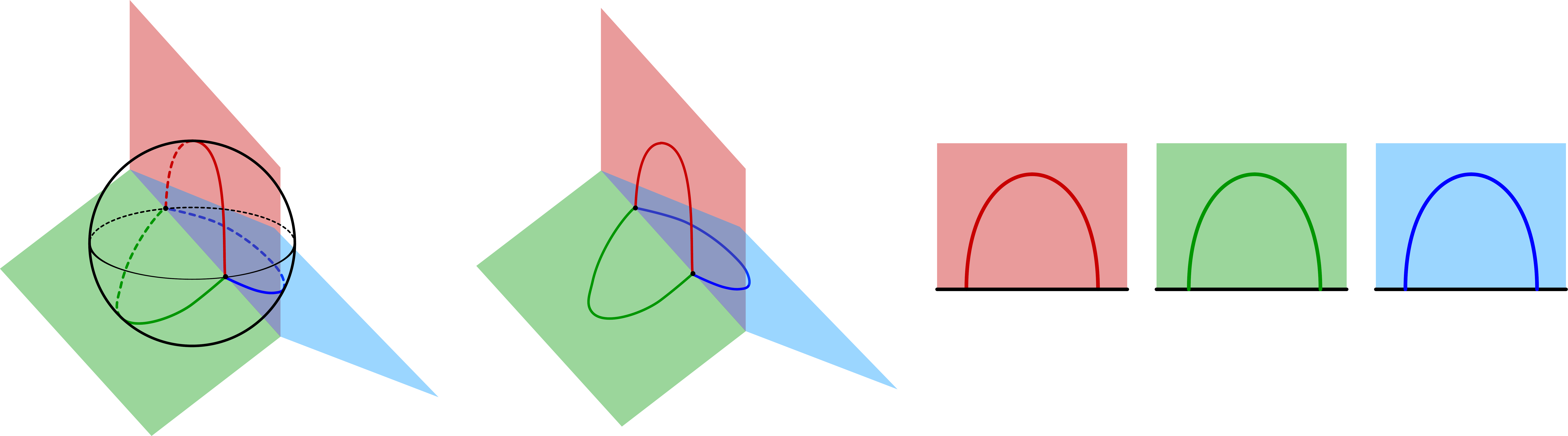}
    \caption{A 1--bridge trisection of an unknotted 2--sphere, depicted with the tri-plane in 3--space, along with the corresponding tri-plane diagram.}
    \label{fig:1Bridge}
\end{figure}

\begin{figure}[h!]
  \centering
    \includegraphics[width=.9\textwidth]{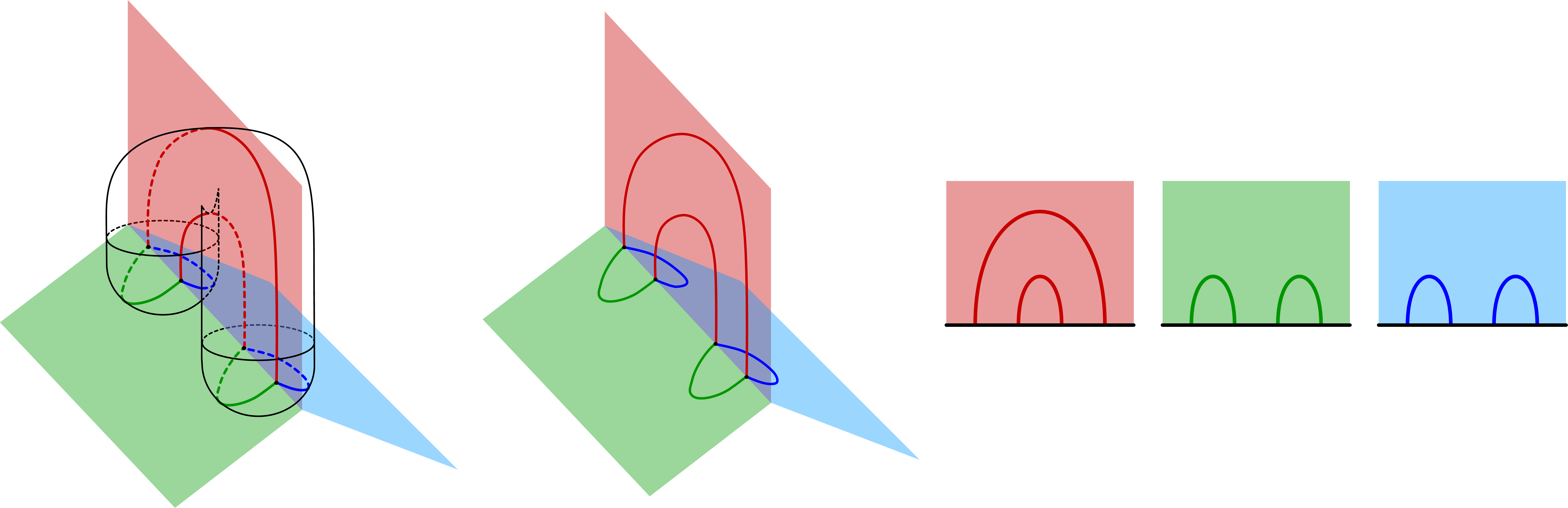}
    \caption{A 2--bridge trisection of an unknotted 2--sphere, depicted with the tri-plane in 3--space, along with the corresponding tri-plane diagram.}
    \label{fig:2Bridge}
\end{figure}

\subsection{Tri-plane moves}\label{subsec:intromoves}\ 

At the end of Section \ref{sec:unique}, we prove Theorem \ref{thm:tri-planemoves}, which asserts that any two tri-plane diagrams for a fixed knotted surface $\K$ in $S^4$ are related by a finite sequence of tri-plane moves.  There are three types of moves: interior Reidemeister moves,  mutual braid transpositions, and stabilization/destabilization.  We briefly describe these moves here, but we go into more detail regarding stabilization and destabilization in Section \ref{sec:stab}.

Given a tri-plane diagram $\Pau$, an \emph{interior Reidemeister move} is simply the process of performing a Reidemeister move within the interior of one of the $\Pau_{ij}$.  These moves correspond to isotopies of the corresponding knotted surface that are supported away from the bridge sphere.

A \emph{mutual braid transposition} is a braid move performed on a pair of adjacent strands contained in all three diagrams $\Pau_{12}$, $\Pau_{23}$, and $\Pau_{31}$.  This move corresponds to an isotopy that are supported in a neighborhood of two adjacent intersections of $\K$ with $\Sigma$ along the curve $e$.  See Figure \ref{fig:BridgeSphereBraiding} for an example.

\begin{figure}[h!]
\centering
\includegraphics[scale = .35]{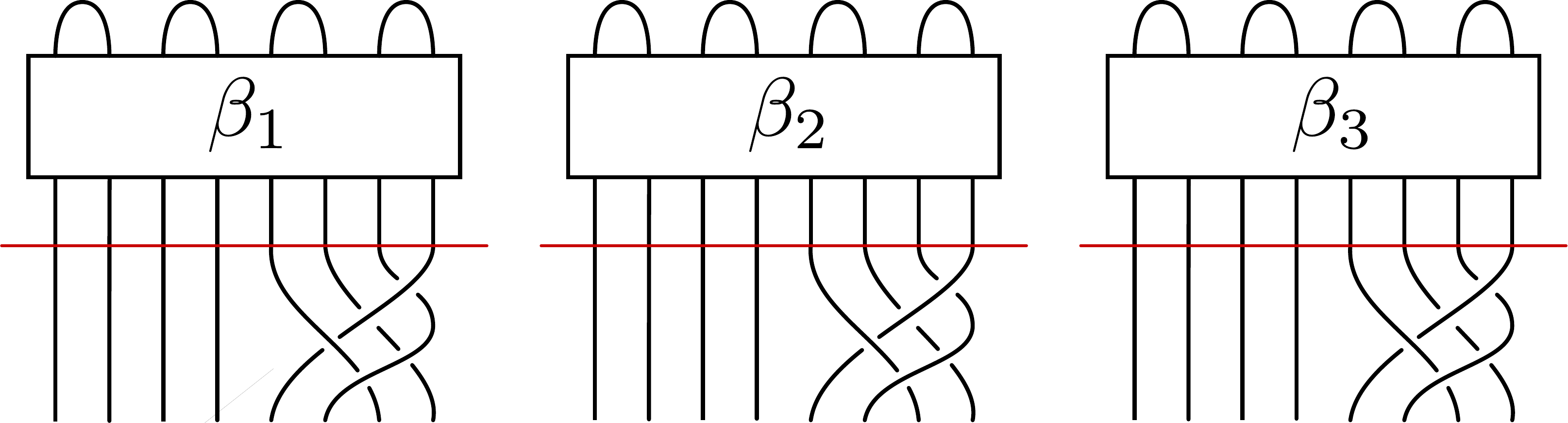}
\caption{The product of five mutual braid transpositions}
\label{fig:BridgeSphereBraiding}
\end{figure}

Figure \ref{fig:4Stab} shows an example of a stabilization move and its inverse, a destabilization move.  These moves are the most complicated, and so we postpone their discussion until Section \ref{sec:stab}.  Note that a stabilization move turns a $b$--bridge trisection into a $(b+1)$--bridge trisection.

\begin{figure}[h!]
\centering
\includegraphics[width=.95\textwidth]{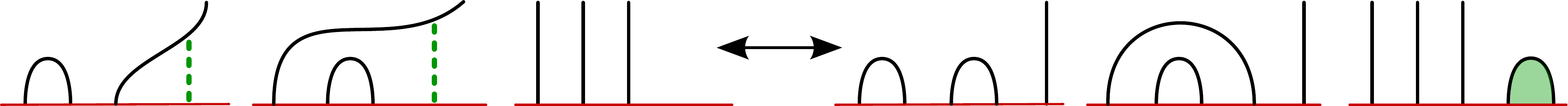}
\caption{A simple case of one of the stabilization and destabilization operations.}
\label{fig:4Stab}
\end{figure}

We will show in Section \ref{sec:stab} that any two tri-plane diagrams $\Pau$ and $\Pau'$ corresponding to the \emph{same} bridge trisection $\T$ are related by a sequence of interior Reidemeiester moves and mutual braid transpositions.  More generally, any two tri-plane diagrams $\Pau$ and $\Pau'$ yielding potentially different trisections of a knotted surface $\K$ in $S^4$ are related by a sequence of all three moves.

\subsection{Branched double covers of bridge trisections}\label{subsec:branched}\ 

We conclude this section by relating bridge trisections to trisections (of 4--manifolds) via the branched double cover construction.  First, we recall the definition of a trisection.

\begin{definition}
	Let $X$ be a closed, connected, orientable, smooth 4--manifold. A \emph{$(g;k_1,k_2,k_3)$--trisection of $X$} is a decomposition
	$$X=X_1\cup X_2\cup X_3$$
	such that
	\be
		\item $X_i\cong\natural^{k_1}(S^1\times B^3)$,
		\item $H_{ij} = X_i\cap X_j$ is a genus $g$ handlebody, and
		\item $\Sigma=X_1\cap X_2\cap X_3$ is a closed surface of genus $g$.
	\ee
	The union $H_{12}\cup H_{23}\cup H_{31}$ is called the \emph{spine} of the trisection.
\end{definition}

Note that the trisection (and hence the 4--manifold) is determined uniquely by its spine (by \cite{laudenbach-poenaru}), which can be encoded as a Heegaard triple $(\Sigma,\A,\n,\g)$, where $\A$, $\n$, and $\g$ are $g$--tuples of simple closed curves on $\Sigma$ describing cut systems for the handlebodies $H_{12}$, $H_{23}$ and $H_{31}$, respectively.

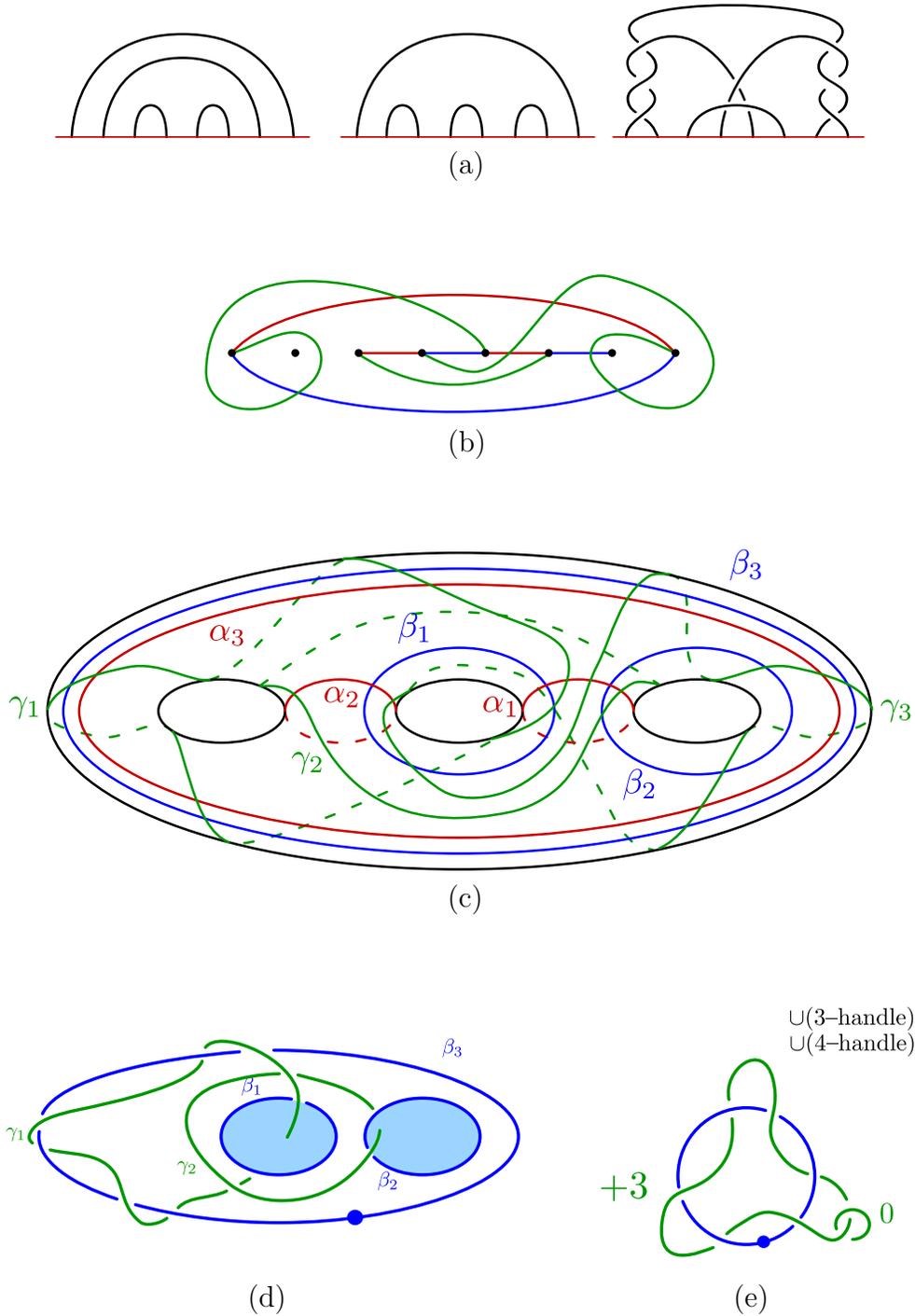
\begin{figure}[h!]
\centering
\includegraphics[scale=.4]{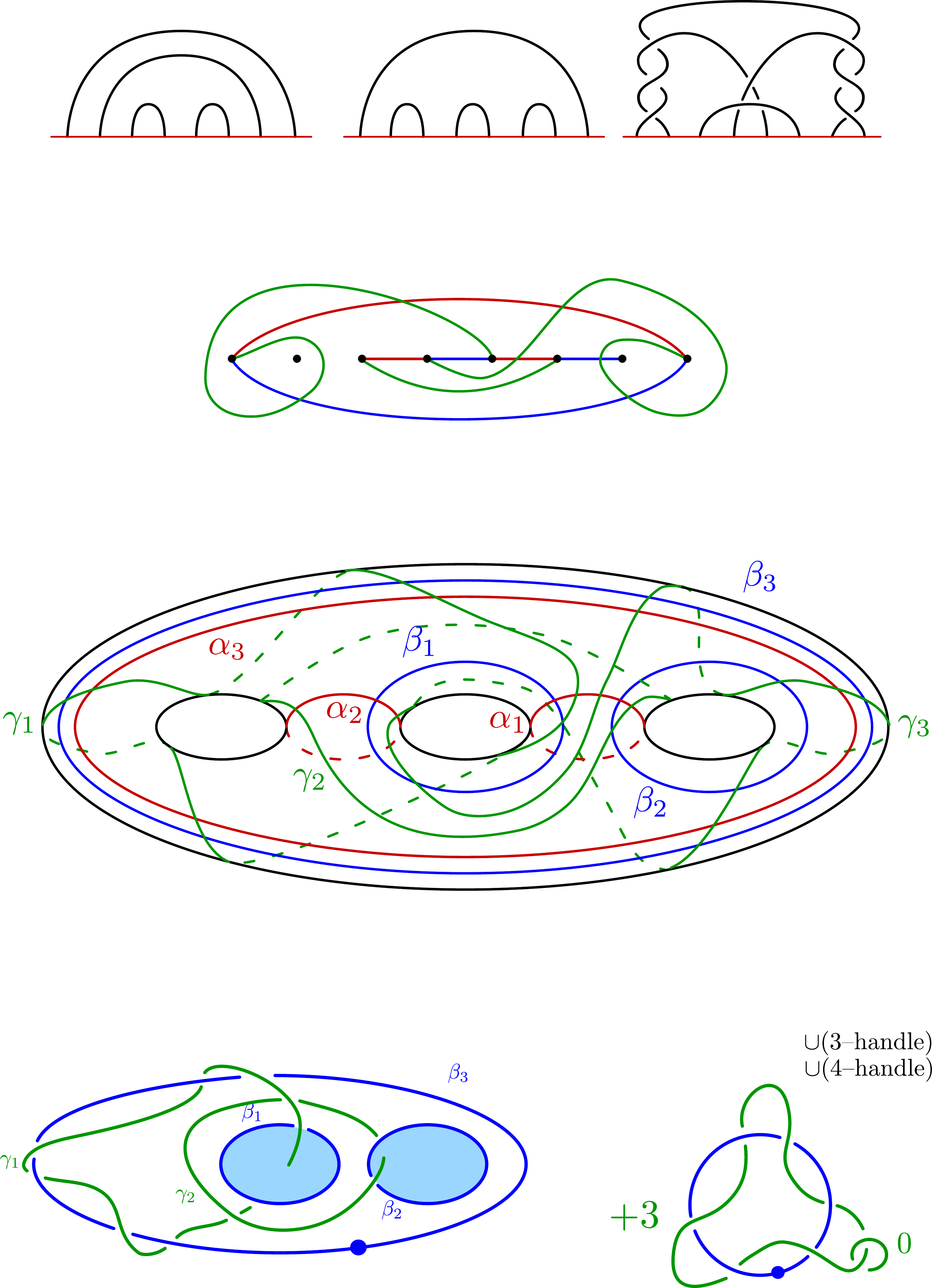}
\put(-189,440){(a)}
\put(-189,327){(b)}
\put(-189,142){(c)}
\put(-270,-20){(d)}
\put(-73,-20){(e)}
\caption{(a) A tri-plane diagram for the spun trefoil. (b) A choice of three bridge disks for each tangle, described via their intersection with the bridge sphere. (c) A trisection diagram for the branched double cover of the spun trefoil. (d) Two $\gamma$--curves that are dual to disks in $H_\beta$. (e) The resulting Kirby diagram.}
\label{fig:SpunCover}
\end{figure}

The manner in which a trisection of $X$ gives rise to a handle decomposition of $X$ is described in detail in \cite{gay-kirby:trisections}.  In brief, there is a handle decomposition of $X$ such that $X_1$ contains one 0--handle and $k_1$ 1--handles, $X_2$ contains $g-k_2$ 2--handles, and $X_3$ contains $k_3$ 3--handles and one 4--handle.  Moreover, we may obtain a Kirby diagram from this decomposition:  The $k_1$ dotted loops come from $k_1$ pairwise disjoint curves in $\Sigma$ bounding disks in both $H_{12}$ and $H_{31}$, and the attaching curves for the 2--handles come from $g - k_2$ pairwise disjoint surface-framed curves in $\Sigma$ that bound disks in $H_{23}$ and are primitive in $H_{12}$.  See below for an example.

Now, let $\K$ be a knotted surface in $S^4$, and let $X(\K)$ denote the double cover of $S^4$ branched along $\K$.  Any $(b;c_1,c_2,c_3)$--bridge trisection $\T$ of $\K$ induces a trisection $\widehat\T$ for $X(\K)$:  This follows almost immediately from the fact that the double cover of a $n$--ball branched along a collection of $c$ trivially embedded $(n-2)$--disks is a $n$--dimensional 1--handlebody of genus $c-1$.  In particular, the branched double cover of a trivial $c$--disk system is $\natural^{c-1}(S^1\times B^3)$, and the branched double cover of a trivial $b$--strand tangle is $\natural^{b-1}(S^1\times D^2)$.

Thus, since the $(b;c_1,c_2,c_3)$--bridge trisection $\T$ is determined by its spine, a triple of trivial $b$--strand tangles, we need only consider the branched double cover of this spine, which will be a triple of 3--dimensional handlebodies of genus $b-1$ .  This triple is enough to determine the trisection $\widehat \T$, but we can also see that the each trivial $c_i$--disk system lifts to a copy of $\natural^{c_i-1}(S^1\times B^3)$.  It follows that $\widehat \T$ is a $(b-1;c_1-1,c_2-1,c_3-1)$--trisection of $X(\K)$.

Moreover, a choice of bridge disks for each trivial tangle in $\T$ gives rise to a choice of compressing disks for the corresponding handlebody in $\widehat\T$. Note that only $b-1$ bridge disks are required in each tangle, the last one being redundant in the handlebody description.

As an example, let $\K$ denote the spun trefoil.  (See Section \ref{sec:exs} for a definition.) Figure \ref{fig:SpunCover} shows how to produce a simple Kirby diagram for $X(\K)$.  As an exercise, the reader can check that the tri-plane diagram in Figure \ref{fig:SpunCover}(a) can be obtained from the tri-plane diagram in Figure \ref{fig:SpunTrefoilDiag} by a sequence of tri-plane moves.  Since each tangle is 4--stranded, a choice of three bridge disks will determine the trisection of $X(\K)$.  The shadows of these bridge disks in the bridge sphere are shown in Figure \ref{fig:SpunCover}(b), and these arcs lift to curves on a genus three surface specifying three handlebodies.  The corresponding trisection diagram is shown in Figure \ref{fig:SpunCover}(c). To recover a Kirby diagram -- see \cite{gay-kirby:trisections} for details -- we push the $\gamma$--curves into the $\beta$--handlebody and notice that $\gamma_1$ and $\gamma_2$ are dual to the disks bounded by $\beta_1$ and $\beta_2$, respectively. This allows us to think of the link $L=\gamma_1\cup\gamma_2$ as the attaching circles for a pair of 2--handles, which are surface-framed by $\Sigma$. The identification $\alpha_3=\beta_3$ gives rise to a 1--handle, and the end result is a description of our manifold as surgery on the 2--component link $L$ in $S^1\times S^2$.  See Figure \ref{fig:SpunCover}(d). The resulting Kirby diagram is shown in Figure \ref{fig:SpunCover}(e).

\section{Existence of bridge trisections}\label{sec:exist}

In this section, we use hyperbolic splittings and banded link presentations (defined below) of knotted surfaces to prove the existence of bridge trisections.  We introduce a special type of banded link presentation, called a banded bridge splitting, which we show to be equivalent to a bridge trisection.  We will prove that every knotted surface admits a bridge trisection by showing that it admits a banded link presentation with a banded bridge splitting.  We begin with several definitions.

Let $L$ be a link in $S^3$.  A \emph{band} $\ups$ for $L$ is an embedding of the unit square $I\times I$ in $S^3$ such that $\ups\cap L=\{-1,1\}\times I$.  Let $L_\ups=(L\setminus \{-1,1\}\times I)\cup(I\times\{-1,1\})$.  Then $L_\ups$ is a new link and is said to result from \emph{resolving} the band $\ups$.  We often let $\ups$ denote a collection of pairwise disjoint bands and let $L_{\ups}$ denote the result of resolving all bands in $\ups$.

Note that a band $\ups$ is determined by its \emph{core}, the arc $y=I\times\{0\}$, and its \emph{framing}, a normal vector field along $y$ that is tangent to $\nu$. If $\Sigma$ is an embedded surface in $S^3$ with $y\subset\Sigma$, we say that $y$ is \emph{surface-framed} by $\Sigma$ if the framing of $y$ is either normal to $\Sigma$ at every point of $y$ or tangent to $\Sigma$ at every point of $y$.  Note that this can happen in two distinct ways: If $\ups$ is induced by a surface-framed arc in $\Sigma$ and $L$ meets $\Sigma$ transversely, then the band $\ups$ also meets $\Sigma$ transversely. On the other hand, if $L$ is contained in $\Sigma$ near the endpoints of $y$, then $\ups$ will be contained in $\Sigma$.

We say a Morse function $h\colon S^4\to\R$ is \emph{standard} if $h$ has precisely two critical points, one of index zero and one of index four.  For a compact submanifold $X$ of $S^4$ (of any dimension), let $X_{[t,s]} = h^{-1}\left([t,s]\right) \cap X$ and let $X_t = h^{-1}(t) \cap X$.  In particular, $S^4_{[t,s]} = h^{-1}([t,s]) \cong S^3\times[t,s]$.  Similarly, for any compact subset $Y\subset S^4_t$ with $t\in[s,r]$, we let $Y[s,r]$ denote the vertical cylinder $Y\times[s,r]$ obtained by pushing $Y$ along the flow of $h$ during time $[s,r]$.  We extend these definitions in the obvious way to any interval or point in $\R$.  

Now, we recall that for every knotted surface $\K$, there exists a Morse function $h\colon (S^4,\K) \rightarrow \R$ such that
\be
\item The function $h_{S^4}$ is standard.
\item Every minimum of $h_\K$ occurs in the level $h^{-1}(-1)$.
\item Every saddle of $h_\K$ occurs in the level $h^{-1}(0)$.
\item Every maximum of $h_\K$ occurs in the level $h^{-1}(1)$.
\ee
Following \cite{lomonaco:homotopy}, we call such a Morse function $h$ a \emph{hyperbolic splitting} of $(S^4,\K)$. 

In this case, each of $\K_{\pm \eps}$ is an unlink in the 3--sphere $S^4_{\pm\eps}$.  In addition, if $h_\K$ has $n$ saddle points, there are $n$ framed arcs $y = \{y_1,\dots,y_n\}$ (which can be chosen to be disjoint) such that attaching the corresponding bands $\ups$ to $\K_{-\eps}$ yields $\K_{\eps}$.  In this case, we may push the bands into $S^4_0$, after which $\K_0=\K_{-\eps}\cup\ups$.  To simplify notation, we will usually write $L = \K_{-\eps}$, so that $\K_{\eps} = L_{\ups}$.  We call $(L,\ups)$ a \emph{banded link}, noting that our definition requires that both $L$ and $L_{\ups}$ are unlinks.  Observe that every hyperbolic splitting yields a banded link. 

Conversely, if $(L,\ups)$ is a banded link, then we may construct a knotted surface $\K=\K(L,\ups)$, called the \emph{realizing surface} as follows:
\be
\item $\K_{(-\eps,0)} = L(-\eps,0)$,
\item $\K_0 = L\cup\ups$, 
\item $\K_{(0,\eps)} = L_\ups(0,\eps)$,
\item $\K_{-\eps}$ and $\K_{\eps}$ are collections of disks that cap off $\K$ along $L$ and $L_\ups$, respectively.
\ee
Note that the disks capping of $L$ and $L_{\ups}$ are unique up to isotopy in $B^4$ by Proposition~\ref{prop:trivialdisks}.  If follows that if a hyperbolic splitting $h$ of $(S^4,\K)$ induces a banded link $(L,\ups)$, then $\K = \K(L,\ups)$. 

Next, we introduce a decomposition of a banded link $(L,\ups)$, which gives rise to a canonical bridge trisection of $\K(L,\ups)$.  A \emph{banded $b$--bridge splitting} $\BB$ of a banded link $(L,\ups)$ is a decomposition
\[ (S^3,L,\ups) = (B_{12},\A_{12}) \cup_{\Sigma,y^*} (B_{31},\A_{31}),\]
where
\be
\item $(S^3,L) = (B_{12},\A_{12}) \cup_{\Sigma} (B_{31},\A_{31})$ is a $b$--bridge splitting of $L$,
\item the bands $\ups$ are described by the surface-framed arcs $y^*\subset\Sigma$,
\item there is a collection $\A^*_{12}$ of shadows for the arcs in $\A_{12}$ such that $\A_{12}^* \cup y^*$ is a collection of embedded, pairwise disjoint arcs in $\Sigma$.
\ee
The collection of shadow arcs in condition (3) is said to be \emph{dual} to $y^*$.  We will usually let $c_1$ denote the number of components of $L$, let $c_2=b - |y^*|$, and let $c_3$ denote the number of components of $L_{\ups}$.  In the case that $c_1 = c_2 = c_3$, we say that the banded bridge splitting $\BB$ is \emph{balanced}. 

Given a banded bridge splitting $\BB$ with components labeled as above, we will describe a process which builds a bridge trisection $\T(\BB)$.  For the first step, consider $S^3=B_{12}\cup B_{31}$ as the equator of $S^4_0$ of $S^4$ and let $X = B_{12}[-\eps,\eps]$.  We may push the bands $\ups$ along $\A_{12}[0]$ into the interior of $B_{12}[0]$ and define a subspace $\D$ of $X$ by
\be
\item $\D_{[-\eps,0)} = \A_{12}[-\eps,0)$,
\item $\D_{0}= \A_{12}[0] \cup \ups[0]$,
\item $\D_{(0,\eps]} = (\A_{12})_\ups(0,\eps]$,
\ee
where $(\A_{12})_\ups$ denotes the result of banding the strands of $\A_{12}$ along $\ups$.  Note that $(\A_{12})_{\ups}[0] \subset \D_0$.  We also observe that $\D$ may be considered to be $\K(L,\ups) \cap X$ using our definition of the realizing surface $\K(L,\ups)$.  In the next lemma, we examine $(X,\D)$ and its subspaces more closely.

\begin{lemma}\label{lemma:trivialarcs}
The arcs $(\A_{12})_{\ups}$ are trivial in $B_{12}$, and $\D$ is a trivial $c_2$--disk system.
\begin{proof}
Consider a band $\ups_i\in\ups$, and recall that $\ups_i=I\times I$, with $\{-1,1\}\times I\subset \A_{12}$.  Isotope $\ups_i$ into $B_{12}$ so that a single arc of $\pd \ups_i$ is contained in $\Sigma$, label this arc $\ups_i^+$, and give the other arc of $I\times\{-1,1\}$ the label $\ups_i^-$.  Extending this convention to the collection $\ups$ of bands gives two collections $\ups^+$ and $\ups^-$ of associated boundary arcs. 

After pushing $\ups$ into the interior of $B_{12}$, let $\Delta$ be a set of bridge disks for $\A_{12}$ yielding the shadows $\A_{12}^*$ dual to the framed arcs $y^* \subset \Sigma$, and let $C$ be a connected component of $\A_{12} \cup \upsilon$.  Since no component of $\A_{12}^* \cup y^*$ is a simple closed curve, it follows that if $C$ contains $b_C$ arcs of $\A_{12}$, then $C$ must contain a collection $\upsilon_C \subset \upsilon$ of $b_C - 1$ bands (possibly $b_C-1 = 0$) of $\upsilon$, and each band separates $C$, so that attaching $\upsilon_C$ yields $b_C$ arcs of $(\A_{12})_{\ups}$.  

Recall that the arcs $y^*_C$ associated to $\upsilon_C$ have the surface framing, so that there is an isotopy of $\upsilon_C$ in $B_{12}$ which pushes $\upsilon_C^+$ onto $y^*_C \subset \Sigma$.  By way of this isotopy, we see that $b_C - 1$ of the $b_C$ arcs of $(\A_{12})_{\ups}$ are trivial; the bridge disks are given by the trace of the isotopy.  Let $\Delta_C \subset \Delta$ denote the $b_C$ dual bridge disks corresponding to the arcs of $\A_{12}$ in $C$.  Assuming $\upsilon_C$ has been isotoped so that $\upsilon_C^+ \subset \Sigma$, we have a slight push off of $\Delta_C \cup \upsilon_C$ is a bridge disk for the remaining arc in $(\A_{12})_{\ups}$ arising from attaching $\upsilon_C$ to $\A_{12}$.  We conclude that all arcs of $(\A_{12})_{\ups}$ are trivial.  See Figure \ref{banddisks}. 

\begin{figure}[h!]
  \centering
    \includegraphics[width=.8\textwidth]{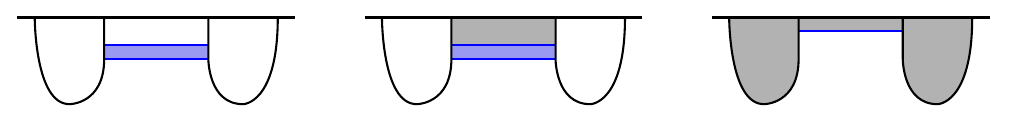}
    \caption{At left, a band connecting two bridge disks.  At middle, a bridge disk for one of the arcs resulting from attaching the band.  At right, a bridge disk for the other resulting arc after canceling the first resulting arc. }
    \label{banddisks}
\end{figure}

For the second part of the proof, we first note that $\D$ is isotopic to $\D'$ given by
\be
\item $\D'_{[-\eps,\eps)} = \A_{12}[-\eps,\eps)$,
\item $\D'_{\eps}= \A_{12}[\eps] \cup \ups[\eps]$.
\ee
The collection $\D'$ is further isotopic to the set $\D''$ given by
\be
\item $\D''_{-\eps} = \Delta[-\eps]$,
\item $\D''_{(-\eps,\eps)} = \A_{12}^*(\eps,\eps)$,
\item $\D''_{\eps} = \Delta[\eps] \cup \ups[\eps]$.
\ee
Cutting $\D''$ along $\ups[\eps]$ yields a collection of $b$ pairwise disjoint disks, and since no component of $\A_{12}^* \cup y^*$ is a simple closed curve, each band in $\ups[\eps]$ separates $\D''$.  It follows that $\D'' \subset \pd X$ is a collection of $b - |\ups|$ disks, and $(X,\D)$ is a trivial $c_2$--disk system.
\end{proof}
\end{lemma}




\begin{lemma}\label{lemma:morsetri}
A banded bridge splitting $\BB$ for a banded link $(L,\ups)$ gives rise to a bridge trisection $\T(\BB)$ for $\K(L,\ups)$.
\begin{proof}
It suffices to describe the manner in which $\BB$ induces a spine $\Ss(\BB)$ for a bridge trisection of $\K(L,\ups)$.  As above, consider the decomposition $S^3=B_{12}\cup_\Sigma B_{31}$ of the equatorial 3--sphere $S^4_0$. We define the three pieces of our spine as the following subsets of the product neighborhood $S^3[-\eps,\eps]$:
\begin{eqnarray*}
(B_{12}',\A_{12}') &=& (B_{12}[-\eps],\A_{12}[-\eps]) \cup (\pd B_{12}[-\eps,0],\pd \A_{12}[-\eps,0]), \\
(B_{23}',\A_{23}') &=& (B_{12}[\eps],(\A_{12})_{\ups}[\eps]) \cup (\pd B_{12}[0,\eps],\pd \A_{12}[0,\eps]), \\
(B_{31}',\A_{31}') &=& (B_{31}[0],\A_{31}[0]). 
\end{eqnarray*}

By definition, $(B_{12}',\A_{12}')$ and $(B_{31}',\A_{31}')$ are trivial tangles, and by Lemma \ref{lemma:trivialarcs}, $(B_{23}',\A_{23}')$ is also a trivial tangle.
In addition, $(B_{12}',\A_{12}') \cup (B_{31}',\A_{31}')$ is isotopic  to $(S^3[-\eps],L[-\eps])$ and $(B_{31}',\A_{31}') \cup (B_{23}',\A_{23}')$ is isotopic to $(S^3[\eps],L_{\ups}[\eps])$, and so these two unions describe unlinks.
Finally, by Lemma \ref{lemma:trivialarcs}, the union $(B_{12}',\A_{12}') \cup (B_{23}',\A_{23}')$ is also an unlink, namely $(\pd X,\pd \D)$, and thus $\Ss(\BB) = (B_{12}',\A_{12}') \cup (B_{31}',\A_{31}') \cup (B_{23}',\A_{23}')$ is the spine of a bridge trisection of $\T(\BB)$.  By construction, $\T(\BB)$ is a bridge trisection for the knotted surface $\K(L,\ups)$, as desired.  

We note for completeness that the rest of the bridge trisection of $(S^4,\K(L,\ups))$ can be described as follows: $(X_2,\D_2) =(X,\D)$, $(X_1,\D_1)$ and $(X_3,\D_3)$ can be described as $(B_{31}[-\eps,0],\A_{31}[-\eps,0])\cup(S^4_{(-\infty,-\eps]},\K_{-\eps})$ and $(B_{31}[0,\eps],\A_{31}[0,\eps])\cup(S^4_{[\eps,\infty)},\K_\eps)$, respectively. See the schematic in Figure \ref{scheme}. Note that the bridge surface for $\T(\BB)$ may be described as $\Sigma[0]$.

\end{proof}
\end{lemma}

\begin{figure}[h!]
  \centering
    \includegraphics[width=.6\textwidth]{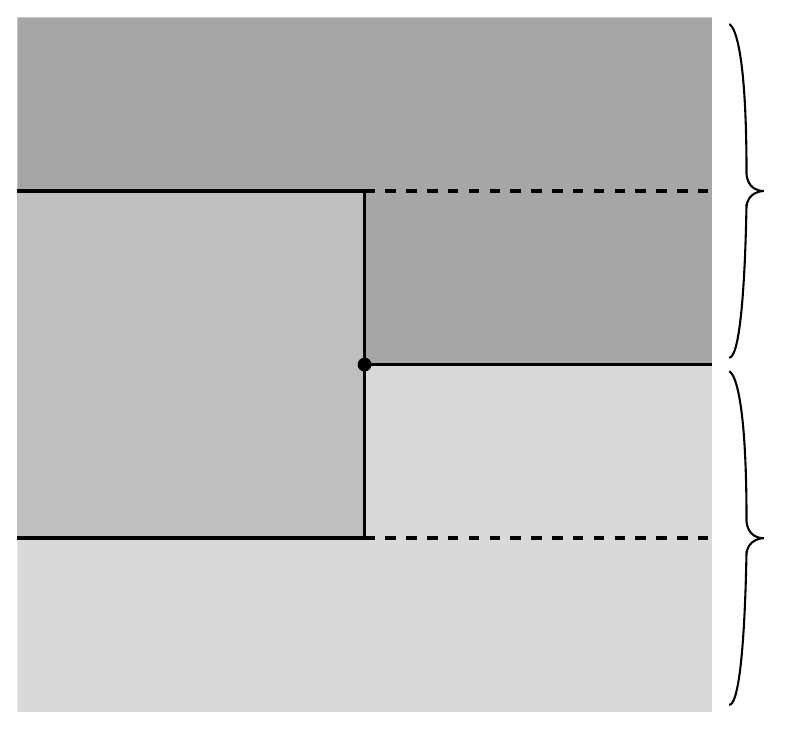}
    \put(-157,114){$\Sigma$}
    \put(-223,114){$(X_2,\D_2)$}
    \put(-10,170){$(X_3,\D_3)$}
    \put(-10,59){$(X_1,\D_1)$}
    \put(-275,59){$-\eps$}
    \put(-267,170){$\eps$}
    \caption{A schematic diagram of the bridge trisection induced by a banded bridge splitting.}
    \label{scheme}
\end{figure}
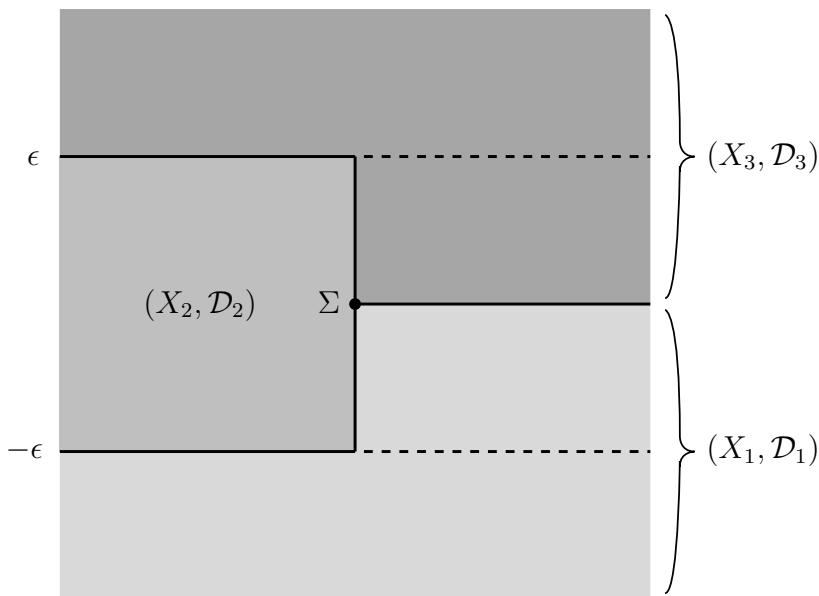

This process may also be reversed, as we see in the next lemma.

\begin{lemma}\label{lemma:trimorse}
A bridge trisection $\T$ of $(S^4,\K)$ induces a banded link presentation $(L,\ups)$ of $\K$ and a banded bridge splitting $\BB$ of $(S^3,L,\ups)$.
\end{lemma}

\begin{proof}
Suppose that $\T$ has spine $\Ss = (B_{12},\A_{12})\cup(B_{23},\A_{23})\cup (B_{31},\A_{31})$.  By Proposition \ref{linkbridge}, the bridge splitting $(B_{12},\A_{12}) \cup_{\Sigma} (B_{23},\A_{23})$ is standard, so we can choose collections of shadows $\A_{12}^*$ and $\A_{23}^*$ for $\A_{12}$ and $\A_{23}$, respectively, so that $\A_{12}^* \cup \A_{23}^*$ is a union of $c_2=|L_2|$ pairwise disjoint simple closed curves in $\Sigma$.  For each component $C_i$ of $\A_{12}^* \cup \A_{23}^*$, fix a single arc $a_i^* \in \A_{23}^*$.  For each $i$, $1\leq i \leq c_2$, let $\gamma_i = \overline{C_i \setminus a_i^*}$.  By a standard argument, if $a_i'$ is a slight pushoff of $\gamma_i$ which shares its endpoints, then $\A_{23}^{**} = \A_{23}^* \setminus \{a_1^* \cup \dots \cup a_{c_2}^*\} \cup \{a_1' \cup \dots \cup a_{c_2}'\}$ is also a set of shadows for $\A_{23}$ with the property that $\A_{12}^* \cup \A_{23}^{**}$ is a collection of pairwise disjoint simple closed curves. 

Let $y^* = \A_{23}^{**} \setminus \{a_1' \cup \dots \cup a_{c_2}'\}$, and let $\ups$ be a set of bands for the unlink $(S^3,L_1) = (B_{12},\A_{12}) \cup (B_{31},\A_{31})$ corresponding to the arcs $y^* \subset \Sigma$ with the surface framing.  We may push $\ups$ into $B_{12}$ and consider the arcs $(\A_{12})_{\ups}$ resulting from their attachment.  By Lemma \ref{lemma:trivialarcs}, there is a collection of shadows for $(\A_{12})_{\ups}$ isotopic to $\A_{23}^{**}$; hence, $(\A_{12})_{\ups}$ is isotopic rel boundary to $\A_{23}$.  It follows that the link $L_{\ups}$ is isotopic to $\A_{23} \cup \A_{31}$ and is also an unlink, so that $(L,\ups)$ is a banded link. 

We claim that
\[ (S^3,L,\ups) = (B_{12},\A_{12}) \cup_{\Sigma,y^*} (B_{31},\A_{31})\]
is a banded bridge splitting, which we call $\BB$, and the realizing surface for $(L,\ups)$ is $\K$.  The first claim follows immediately from the definition of a bridge trisection and from our construction of the arcs $y^*$, since $\A_{12}^* \cup y^* = \gamma_1 \cup \dots \cup \gamma_{c_2}$. 

To prove the second claim, it suffices to show that $\T(\BB) = \T$ using the proof of Lemma \ref{lemma:morsetri}. This also follows from the constructions of $\A_{12}'$ and $y^*$:  If $\Ss' = (B_{12}',\A_{12}') \cup (B_{23}',\A_{23}')\cup (B_{31}',\A_{31}') $ is a spine for $\T(\BB)$, then $(B_{31}',\A_{31}')$ and $(B_{12}',\A_{12}')$ are isotopic in $S^4$ to $(B_{31},\A_{31})$ and $(B_{12},\A_{12})$, respectively, by the proof of Lemma \ref{lemma:morsetri}.  Moreover, a set of shadows for $(B_{23}',\A_{23}')$ is given by the union of the arcs in $y^*$ and pushoffs of the components $\gamma_i$ of $\A_{12}^* \cup y^*$.  But these traces are precisely $y^* \cup \{a_i'\} = \A_{23}^{**}$.  Since two tangles with sets of identical traces must be isotopic rel boundary, we have $(B_{23}',\A_{23}')$ is isotopic to $(B_{23},\A_{23})$; therefore, $\Ss'$ is isotopic to $\Ss$ and $\T$ is equivalent to $\T(\BB)$, as desired.
\end{proof}

Notice that in the proof of Lemma \ref{lemma:trimorse} there are often many pairwise non-isotopic choices for the arcs $y^*$, and thus we see that a bridge trisection may induce many different banded links $(L,\ups)$ and banded bridge splittings $\BB$.  However, if we convert a bridge trisection $\T$ to a banded bridge splitting $\BB$ via Lemma \ref{lemma:trimorse}, then the bridge trisection $\T(\BB)$ given by Lemma \ref{lemma:morsetri} is isotopic to $\T$. 

\begin{remark}\label{rmk:handles}
Lemma \ref{lemma:trimorse} reveals that a $(b;c_1,c_2,c_3)$--bridge trisection $\T$ for $\K$ induces a particular handle decomposition of $\K$:  The lemma produces a banded link presentation $(L,\ups)$, where $L$ is a $c_1$--component unlink, $L_{\ups}$ is a $c_3$--component unlink, and $\ups$ consists of $b - c_2$ bands.  Thus, $\K$ has a handle decomposition with $c_1$  0--handles, $b-c_2$ 1--handles, and $c_3$ 2--handles.  However, notice that the entire construction is symmetric in the $B_{ij}$, so a $(b;c_1,c_2,c_3)$--bridge trisection induces a handle decomposition with $c_i$ 0--handles, $b-c_j$ 1--handles, and $c_k$ 2--handles for any $\{i,j,k\} = \{1,2,3\}$.
\end{remark}

We are ready to prove the main result of this section -- the existence of bridge trisections -- by showing that for every knotted surface $\K$ in $S^4$, there is a banded bridge splitting of a banded link presentation $(L,\ups)$ for $\K$.

\begin{reptheorem}{thm:existence}
Every knotted surface $(S^4,\K)$ admits a $b$--bridge trisection for some $b$.
\end{reptheorem}

\begin{proof}
Choose a banded link presentation $(L,\ups)$ such that $\K = \K(L,\ups)$, and let $g\colon(S^3,L) \rightarrow \R$ be a Morse function such that a level 2--sphere $\Sigma$ is a bridge sphere for $L$.  We will show that there is an isotopy of $g$ resulting in a level bridge sphere for $L$ which is a perturbation of $\Sigma$ and which, when paired with a collection of surface-framed arcs $y^*$ giving rise to $\ups$, yields a banded bridge splitting $\BB$ for $(L,\ups)$.  In an attempt to avoid an abundance of unwieldy notation, we will let $g$ and $\Sigma$ denote the Morse function and level bridge sphere which result from a specified isotopy, despite the fact that they are different than our original $g$ and $\Sigma$.  We also note that any perturbation of $\Sigma$ may be achieved by an isotopy of $g$, and so if we specify such a perturbation, it will be implied that we isotope $g$ accordingly. 

Fix a collection $y$ of framed arcs which give rise to $\ups$.  By following the flow of $g$, we may project $y$ onto immersed arcs $\pi_g(y)$ in the bridge surface $\Sigma$.  We wish to isotope $g$ so that $\pi_g(y)$ is a collection of pairwise disjoint embedded arcs in the bridge surface $\Sigma$.  By perturbing $\Sigma$, we may ensure that arcs in $\pi_g(y)$ have disjoint endpoints.  Moreover, we may remove crossings of $\pi_g(y)$ by perturbing further, as in Figure \ref{levelb}, after which we may assume $\pi_g(y)$ is a collection of embedded arcs in $\Sigma$, and thus there is an isotopy carrying $y$ to $\pi_g(y)$. 

\begin{figure}[h!]
  \centering
    \includegraphics[width=.8\textwidth]{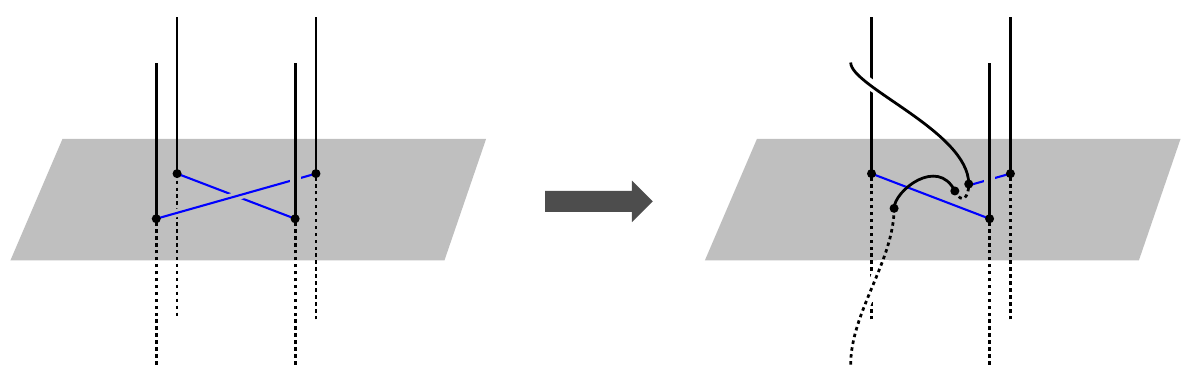}
    \caption{Crossings of arcs in $\pi_g(y)$ may be removed by perturbing.}
    \label{levelb}
\end{figure}

Lastly, if the surface framing of an arc $\pi_g(y_i)$ in $\Sigma$ does not agree with the framing of $y_i$, we perturb $\Sigma$ near the endpoints of $\pi_g(y_i)$, and push $\pi_g(y_i)$ off of and back onto $\Sigma$ with the surface framing.  See Figure \ref{levelc}. Note that this entire process may be achieved by isotopy of $g$, after which we may assume that the surface framings of arcs in $\pi_g(y)$ agree with the surface framings of arcs in $y$. 

\begin{figure}[h!]
  \centering
    \includegraphics[width=.8\textwidth]{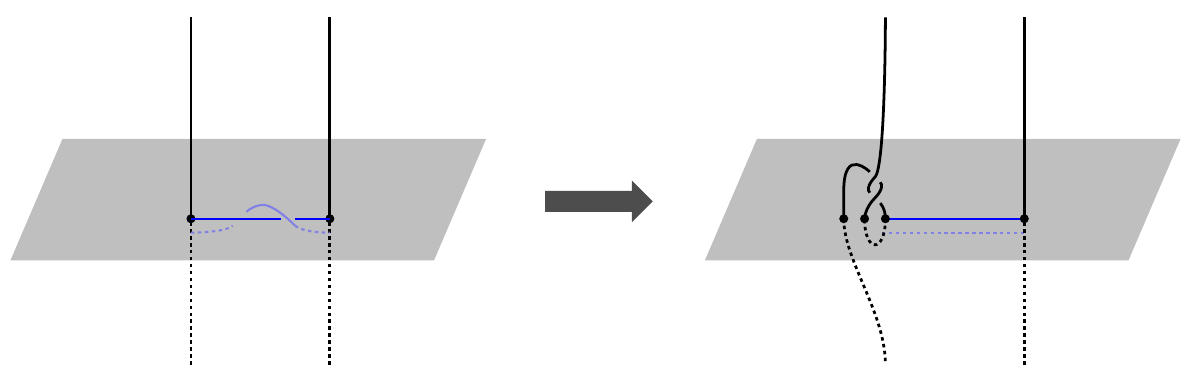}
    \caption{After perturbing, we may assume that the surface framing of each arc in $\pi_g(y)$ agrees with the surface framing of the corresponding arc in $y$.}
    \label{levelc}
\end{figure}

The next step in this process is to further isotope $g$ to get a banded bridge splitting of $(L,\ups)$.  The bridge sphere $\Sigma$ splits $S^3$ into two 3--balls, which we call $B_{12}$ and $B_{31}$.  For each arc $\pi_g(y_i) \in \pi_g(y)$, perturb $\Sigma$ near the endpoints of $\pi_g(y_i)$ so that there are bridge disks $\Delta_i$ and $\Delta_i'$ for $L$ in $B_{12}$ on either end of $\pi_g(y_i)$, and such that $\{\Delta_i\} \cup \{\Delta_i'\}$ is a collection of pairwise disjoint bridge disks.  See Figure \ref{fig:manypert}.  Let $y^* = \pi_g(y)$, and let $\delta^*$ be the union of the shadows $\Delta_i \cap \Sigma$ and $\Delta_i' \cap \Sigma$ for $1 \leq i \leq n$.  Then $\delta^*$ intersects $y^*$ only in points contained in $L$, and each connected component $C$ of $\delta^*\cup y^*$ contains three arcs and is not a simple closed curve.  

\begin{figure}[h!]
  \centering
    \includegraphics[width=.8\textwidth]{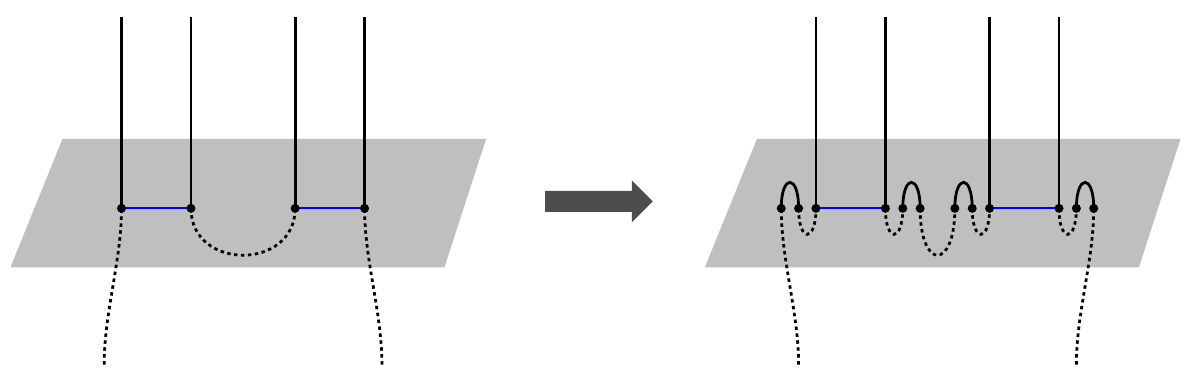}
    \caption{The result of perturbing near the endpoints of each arc in $\pi_g(y)$.}
    \label{fig:manypert}
\end{figure}

Let $D_i$ be the frontier of $\nu(\Delta_i \cup y^*_i \cup \Delta_i')$ in $B_{12}$, so that $\{D_i\}$ is a collection of $n$ pairwise disjoint disks.  By a standard cut-and-paste argument, we may choose a collection $\Delta$ of pairwise disjoint bridge disks for $\A_{12} = L \cap B_{12}$ such that $\Delta \cap D_i = \emptyset$ for all $i$.  Thus, $\Delta$ gives rise to a set of shadow arcs $\A_{12}^*$ for $\A_{12}$ such that $\A_{12}^* \cup y^*$ is a collection of pairwise disjoint embedded arcs. (Note that the collection $\Delta$ may contain more disks than the collection $\{D_i\}$, since some disks bridge disks for $\A_{12}$ are not adjacent to any arcs of $y^*$.) 

Let $\A_{31} = L \cap B_{31}$.  We conclude that
\[ (S^3,L,\ups) = (B_{12},\A_{12}) \cup_{\Sigma,y^*} (B_{31},\A_{31})\]
is a banded bridge splitting $\BB$, and thus $(S^4,\K)$ admits a bridge trisection $\T(\BB)$, by Lemma \ref{lemma:morsetri}.
\end{proof}

\section{Classification of simple bridge trisections}\label{sec:class}

In this section, we discuss several facts about surfaces that admit low-complexity bridge trisections.  Although the cases are introduced as $c_1 = b$, $c_1 = b-1$, and $c_2 = 1$, the conclusions apply for any reindexing of the $c_i$'s.

\subsection{The case $c_1 = b$}\ 

\begin{proposition}
	If $\T$ is a $(b;b,c_2,c_3)$--bridge trisection of a knotted surface $\K$ in $S^4$, then $c_2=c_3$, and $\K$ is the unlink of $c_2$ unknotted 2--spheres.
\end{proposition}
\begin{proof}
By Remark \ref{rmk:handles}, $\T$ induces a handle decomposition with $c_2$  0--handles, $b-c_1$ 1--handles, and $c_3$ 2--handles.  If $c_1=b$, it follows that $\K$ admits a handle decomposition with no 1--handles, and must be the union of two trivial disk systems.  By Proposition \ref{prop:trivialdisks}, that $\K$ is an unlink of $c_2=c_3$ unknotted 2--spheres.   
\end{proof}

One corollary of this proposition is that bridge number detects the unknot.

\begin{corollary}
	Let $\U$ denote the unknotted 2--sphere in $S^4$.  Let $\K$ be a knotted surface with $b(\K)=1$.  Then $\K=\U$.
\end{corollary}

\subsection{The case $c_1 = b-1$}\ 

Suppose that $\K$ admits a $(b;b-1,c_2,c_3)$--bridge trisection.  Following the discussion in Subsection \ref{subsec:branched}, the double branched cover $X(\K)$ of $\K$ admits a $(b-1;b-2,c_2-1,c_3-1)$--trisection, $\widetilde T$.  In this case, we may apply the main theorem of \cite{msz}, which asserts that $X(\K)$ is the connected sum of copies of $S^1 \X S^3$ and at most one copy of $\CP^2$ or $\overline{\CP}^2$, and $\widetilde T$ is the connected sum of standard genus one trisections. 

Thus, since the branched double cover of $\T$ is standard, it seems reasonable to conjecture that $\T$ is also standard and $\K$ is an unlink.

\begin{conjecture}
	Every $(b;b-1,c_2,c_3)$--surface is an unlink of unknotted 2--spheres and at most one unknotted projective plane.	
\end{conjecture}

Note that if $\K$ has such a trisection, then it has a handle decomposition with a single band.  Hence, the conjecture is related to the question of whether attaching a single nontrivial band to a unlink ever yields an unlink.  It follows that $c_3=c_2$ or $c_3=c_2-1$.

\subsection{The case $c_2 = 1$}\label{subsec:b1}\ 

Suppose that $\K$ admits a $(b;c_1,1,c_3)$--bridge trisection $\T$.  Since the bridge splitting of the unknot $(S^3,L_2) = (B_{12},\A_{12})\cup_{\Sigma} (B_{23},\A_{23})$ is standard, there is a tri-plane diagram $\Pau$ for $\T$ such that $\Pau_{12} \cup \overline{\Pau_{23}}$ is the standard diagram pictured in Figure \ref{fig:b1Diag}.  This choice of trivialization has the desirable property that it is preserved under connected sum, as shown in in Figure \ref{fig:StdConnSum}.  If $\T$ is balanced, then by considering Euler characteristic, we know that either $\K \cong \#^{b-1} \RP^2$ or $\K \cong \#^{b'} T^2$, where $b' = \frac{b-1}{2}$. (In the second case, $b$ must be odd.)

\begin{figure}[h!]
\centering
\includegraphics[scale = .45]{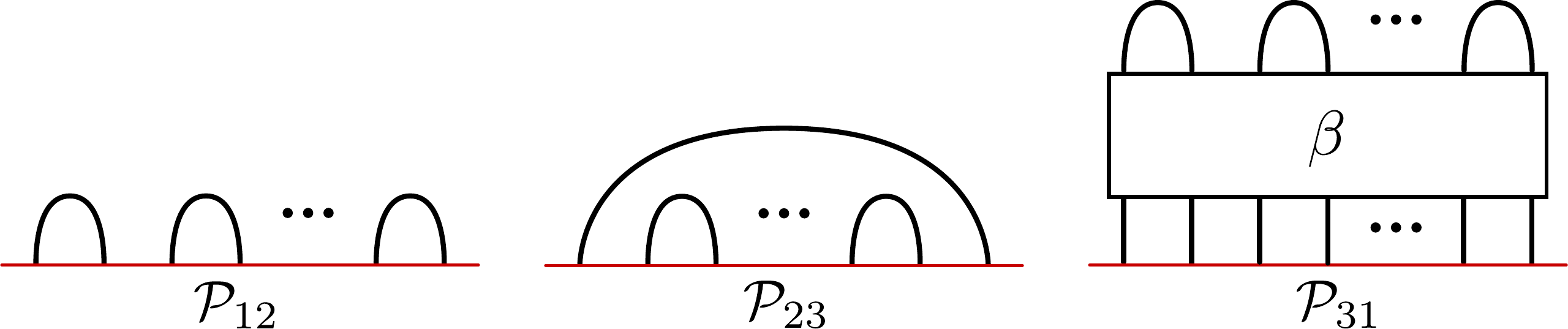}
\caption{A tri-plane diagram for a $(b,1)$--bridge trisection, with the standard trivialization of the first two tangles.}
\label{fig:b1Diag}
\end{figure}

\begin{figure}[h!]
\centering
\includegraphics[scale = .4]{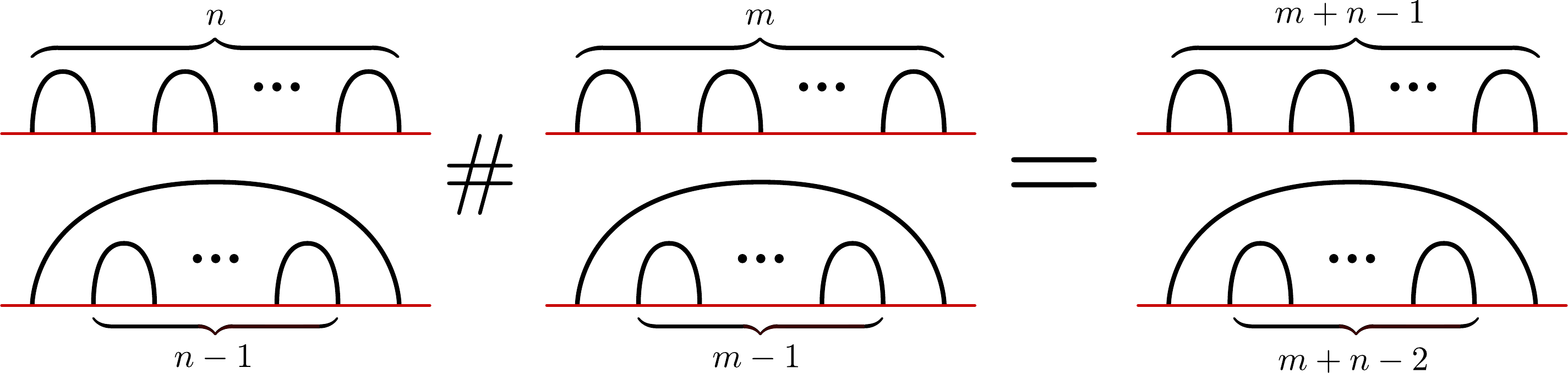}
\caption{The standard trivialization has the property that the form of the first two tangle diagrams is preserved under connected sum.}
\label{fig:StdConnSum}
\end{figure}

First, consider the case $b=2$.  In this case, each $(B_{ij},\A_{ij})$ is a rational tangle with property that the union of any pair is unknotted or unlinked.  It follows this that the slopes of the three tangles must have distance zero or one pairwise.  Our assumption that the first two tangles are standard tells us that they correspond to slopes $0$ and $\infty$.  Thus, the slope of the third tangle must be $0$, $\infty$, or $\pm1$.  If the slope of the third tangle is 0 or $\infty$, the bridge trisection is equivalent to the unbalanced trisection of the unknot shown in Figure \ref{fig:2Bridge}.  In the balanced case, the third slope is $\pm 1$, and there are precisely two surfaces admitting $(2,1)$--bridge trisections, denoted $P_+$ and $P_-$ and pictured in Figure \ref{fig:ProjectivePlanes}(b).  Each of these is homeomorphic to $\RP^2$, and they are distinguished as embeddings in $S^4$ by the Euler number of their normal bundle: $e(P_\pm)=\pm2$. A movie for $P_+$ is shown in Figure \ref{fig:ProjectivePlanes}(a).

\begin{figure}[h!]
\centering
\includegraphics[scale = .4]{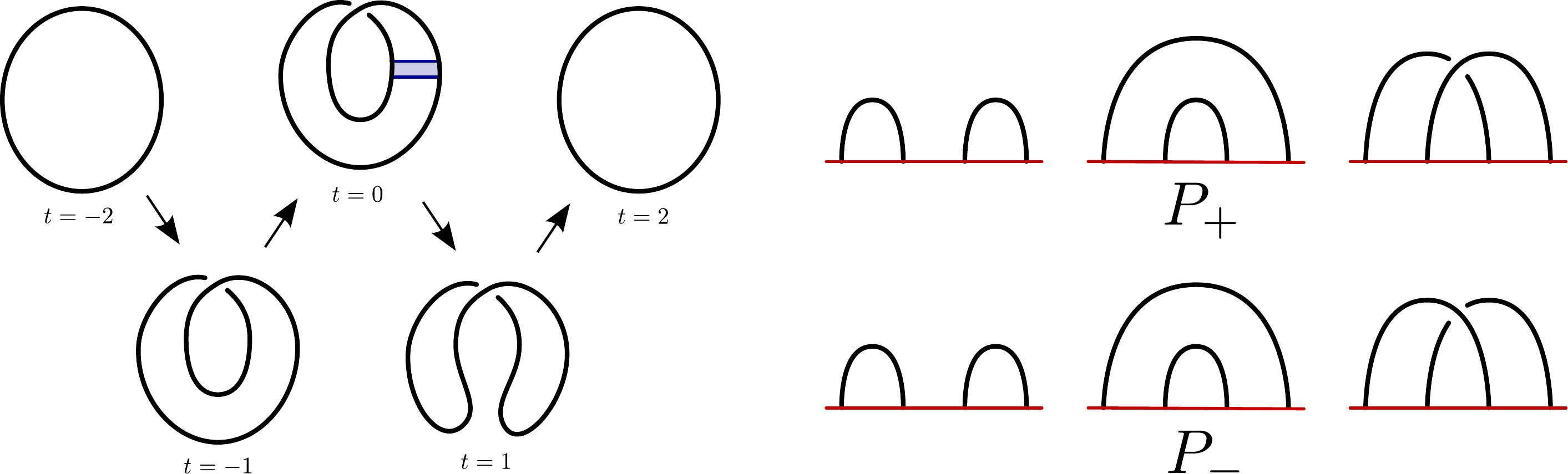}
\put(-260,-20){(a)}
\put(-84,-20){(b)}
\caption{(a) A moving picture description of the unknotted projective plane $P_+$.  (b) The two possible $(2,1)$--bridge surfaces, $P_+$ and $P_-$.}
\label{fig:ProjectivePlanes}
\end{figure}

Let $\K_{i,j} = (\#^iP_+)\#(\#^jP_-)$.  Following \cite{hoso-kawa}, we will say that a non-orientable surface knot $\K$ is \emph{unknotted} if $\K$ is ambient isotopic to $\K_{i,j}$ for some $i,j\in\Z$.  Figure \ref{fig:K32} shows the surface $\K_{3,2}$.  In other words, there are precisely two unknotted projective planes ($P_+$ and $P_-$), and precisely $n+1$ unknotted $\#^n\RP^2$, which are formed as connected sums of $P_+$ and $P_-$ and are distinguished by the Euler class of the their normal bundles \cite{massey}.

\begin{figure}[h!]
\centering
\includegraphics[scale = .3]{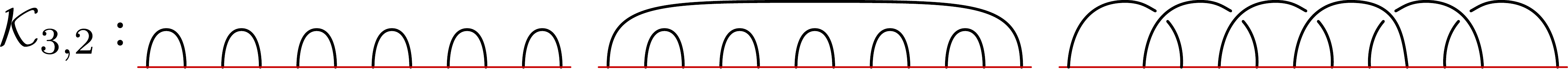}
\caption{One of the standard unknotted non-orientable surfaces homeomorphic to $\#^5\RP^2$.}
\label{fig:K32}
\end{figure}

\subsection{Unknotted surfaces and tri-plane diagrams without crossings}\ 

One obvious way to measure the complexity of a tri-plane diagram is to count the number of crossings.  Crossing number may be a useful way to catalogue knotted surfaces, just as it has been useful to organize classical knots in $S^3$.  As in the classical case, zero crossing diagrams represent simple spaces.

\begin{proposition}\label{zerocrossing}
An orientable knotted surface $\K$ is unknotted if and only if it admits a tri-plane diagram $\Pau$ without crossings.
\end{proposition}
\begin{proof}
First, suppose that $\K$ has a tri-plane diagram $\Pau$ which contains no crossings.  As in Figure \ref{fig:1Bridge}, we can embed a tri-plane in $S^3$, and since $\Pau$ has no crossings, the tangles $\A_{ij}$ embed in the tri-plane.  Moreover, the tri-plane cuts $S^3$ into three 3--balls, $Z_1$, $Z_2$, and $Z_3$, and the trivial disks $\D_i$ embed in $Z_i$.  Hence, the entire surface $\K$ is isotopic into $S^3 \subset S^4$.  By \cite{hoso-kawa}, $\K$ is unknotted if and only if $\K$ is isotopic into $S^3 \subset S^4$, completing one direction of the proof. 

For the reserve implication, observe that Figure \ref{fig:3BridgeDiags} contains a zero-crossing diagram of a torus $T^2$, which must be unknotted by the above arguments.  If $\K$ is an unknotted surface of genus $g$, $\K$ is unique up to isotopy in $S^4$, and thus we may construct a zero-crossing tri-plane diagram for $\K$ by taking the connected sum of $g$ copies of the 3--bridge diagram for $T^2$.
\end{proof}

\subsection{Classifying 3--bridge trisections}\ 

Having classified 1--bridge and 2--bridge trisections, we turn our attention to 3--bridge trisections.  Following the discussion above, there are three unknotted Klein bottles, and 3--bridge trisections for these three surfaces are shown in Figure \ref{fig:3BridgeDiags}, along with a 3--bridge trisection of the unknotted torus.

\begin{figure}[h!]
\centering
\includegraphics[scale = .3]{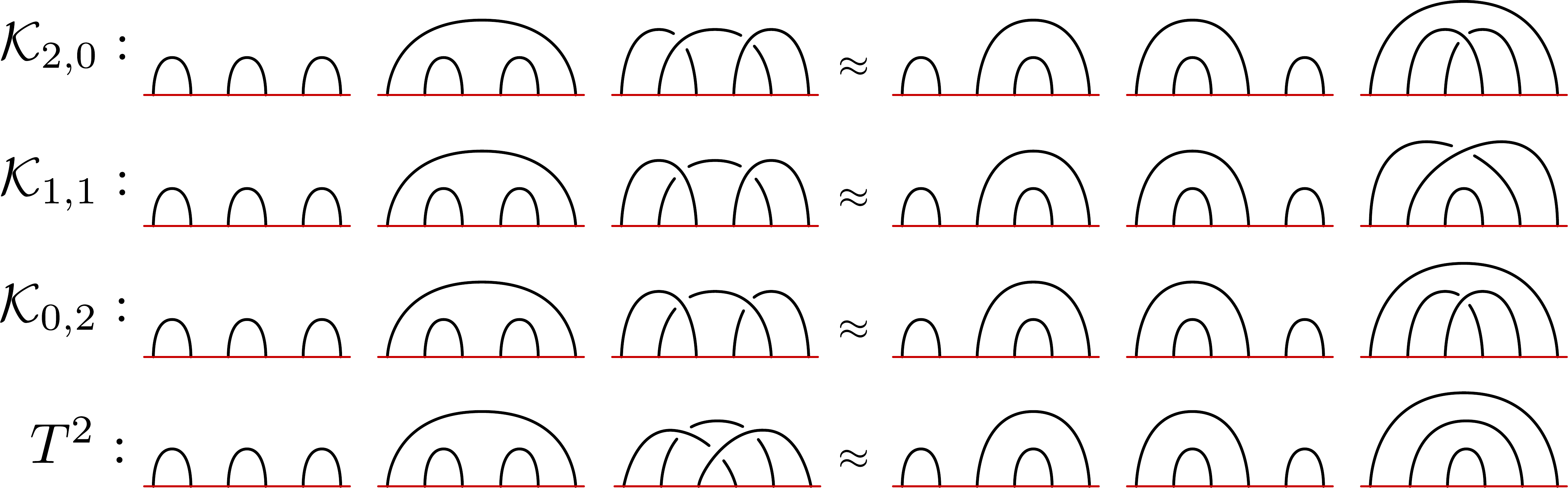}
\caption{Tri-plane diagrams for the four standard 3--bridge surfaces.}
\label{fig:3BridgeDiags}
\end{figure}

So far, we have discovered seven simple balanced bridge trisections: the unique 1--bridge trisection (corresponding to $\U$), the two balanced two-bridge trisections (corresponding to the two unknotted $\RP^2$s), and the four balanced 3--bridge trisections in Figure \ref{fig:3BridgeDiags}.  We will henceforth refer to these seven trisections as \emph{standard}.  Moreover, any trisection obtained as the connected sum of some number of these standard trisections, or any stabilization thereof, will also be called \emph{standard}.

If $\K$ admits a 3--bridge trisection, then $\Sigma_2(\K)$ admits a 2--trisection, as discussed above in Subsection \ref{subsec:branched}. In \cite{mz:genus2}, it is shown that every balanced 2--trisection is standard, and in \cite{msz} the unbalanced case is resolved.  These results imply a classification of 3--bridge surfaces. 

\begin{reptheorem}{thm:standard}
Every knotted surface $\K$ with $b(\K)\leq 3$ is unknotted and any bridge trisection of $\K$ is standard.
\end{reptheorem}

\begin{proof}
	Let $\K$ be a knotted surface in $S^4$ and let $\T$ be a 3--bridge trisection of $\K$.  Let $\widetilde\T$ denote the genus two trisection of $X(\K)$.  Let $(B_{12},\A_{12})\cup(B_{23},\A_{23})\cup(B_{31},\A_{31})$ be the spine of $\T$. The branched double cover of this spine is a triple of handlebodies with common boundary surface $H_{12}\cup H_{23}\cup H_{31}$, which determine $\widetilde T$.  A triple of choices of bridge disks for the three tangles $(B_{ij},\A_{ij})$ lift to give a triple of cut systems for the three handlebodies $H_{ij}$.
	
	By Theorem 1.3 of \cite{mz:genus2}, $\widetilde\T$ is standard. This means that there is a sequence of triples of cut systems for the $H_{ij}$, each of which arises from the previous via a single handleslide in one of the handlebodies, such that the terminal triple is one of the standard trisections described in \cite{mz:genus2}.
	
	Because we are working on a genus two surface, we can apply \cite{haas-susskind}, which states that every cut system can be arranged to respect the hyperbolic involution of the handlebodies.  It follows that each triple of cut systems descends to a triple of collections of bridge disks for the $(B_{ij}, \A_{ij})$.  In other words, each handleslide performed upstairs descends to a bridge disk slide downstairs.
	
	Since the terminal triple of cut systems in this sequence is standard, the bridge disk systems in the quotient must be standard.  It follows that $\T$ is standard.
\end{proof}

Now that we have dispensed of $b$--bridge trisections for $b\leq 3$, we natuarlly turn our attention to 4--bridge trisections.  We resume this thread in Section \ref{sec:exs}, where we prove that there are infinitely many nontrivial $b$--bridge surfaces with $b\geq 4$.  Before proceeding further, we must discuss the fundamental group of knotted surface complements.

\subsection{The fundamental group of a knotted surface}\label{subsec:fund}\ 

Let $\K$ be a knotted surface in $S^4$, and let $\pi(\K)=\pi_1(S^4\setminus\K)$.  The next proposition results from the techniques used in Section \ref{sec:exist}. 

\begin{proposition}\label{prop:fund1}
Let $\T$ be a $(b;c_1,c_2,c_3)$--bridge trisection of $\K$.  Then $\pi(\K)$ has a presentation with $c_i$ generators and $b-c_j$ relations, for any $\{i,j,k\}=\{1,2,3\}$.
\end{proposition}

\begin{proof}
Lemma \ref{lemma:trimorse} and Remark \ref{rmk:handles} tell us how to turn $\T$ into a banded link presentation of $\K$ whose corresponding handle decomposition $\K$ has $c_i$ 0--handles, $b-c_j$ 1--handles, and $c_k$ 2--handles for any bijection $\{i,j,k\}\leftrightarrow\{1,2,3\}$.  This decomposition, in turn, induces a handle decomposition of $S^4\setminus\nu(\K)$ with one 0--handle, $c_i$ 1--handles, $b-c_j$ 2--handles, $c_k$ 3--handles, and one 4--handles. (See \cite{gompf-stipsicz} for details.) In any handle decomposition of a manifold with one 0--handle, the 1--handles give rise to generators of the fundamental group, while the 2--handles give rise to relations.  It follows that we have a presentation for $\pi_1(S^4\setminus\nu(\K))$ with $c_i$ generators and $b-c_j$ relations.
\end{proof}

Returning to the case in which $\K$ admits a $(b;c_1,1,c_3)$--bridge trisection, we notice that, for any such $\K$, we have that $S^4\setminus\K$ has a presentation with one generator.  It follows that $\pi(\K)$ is cyclic.  The group will be $\Z$ or $\Z_2$ according with whether $\K$ is orientable or non-orientable, respectively.  Using this, we obtain the following fact.

\begin{proposition}\label{TopUnknotted}
If $\K$ is orientable and admits a $(b;c_1,1,c_3)$--bridge trisection, then $\K$ is topologically unknotted.
\end{proposition}

\begin{proof}
Since $\pi(\K)\cong\Z$, the result follows from \cite{hillman-kawauchi} and Kawauchi's revision of his earlier proof \cite{kawauchi:revised}.
\end{proof}

Notice also that if $\K\cong\RP^2$ and $\pi(\K)\cong\Z_2$, then $\K$ is topologically unknotted by a result of Lawson \cite{lawson}.  The general non-orientable case seems to be unknown. This raises the following question.

\begin{question}
Can a surface admitting a $(b;c_1,1,c_3)$--bridge trisection be smoothly knotted?
\end{question}



\section{Nontrivial examples}\label{sec:exs}

In this section, we consider bridge trisections with $b\geq 4$.  In particular, we show how the spinning and twist-spinning constructions can be used to produce interesting $b$--bridge surfaces for arbitrarily large $b$.

\subsection{Spun knots and links}\label{subsec:spin}\ 

The first examples of knotted two-spheres were the spun knots constructed by Artin \cite{artin}.  The construction is as follows:  Let $(S^3, K)$ be a knot, and let $(B^3, K^\circ)$ be the result of removing a small, open ball centered on a point in $K$, so that $K^\circ$ is a knotted arc in $B^3$ with endpoints on the north and south poles, labeled $\mathbf n$ and $\mathbf s$, respectively.  Then, the \emph{spin} $\Ss(K)$ of $K$ is given by
$$(S^4,\Ss(K)) = ((B^3,K^\circ)\times S^1)\cup ((S^2,\{ \mathbf n, \mathbf s\})\times D^2).$$
This gives the familiar description of $S^4$ as $(B^3\times S^1)\cup (S^2\times D^2)$ and realizes $\Ss(K)$ by capping off the annulus $K^\circ\times S^1$ with a pair of disks $\{\mathbf n,\mathbf s\}\times D^2$, one at each pole.

There is an alternate description that splits $(S^4,\Ss(K))$ into two pieces:  First, consider the pair $(B^4, D)=(B^3, K^\circ)\times I$.  Here $\partial(B^4, D)=(S^3, K\#(-K))$, where $-K$ denotes the mirror reverse of $K$.  In fact, $D$ is the standard ribbon disk for $K\#(-K)$, which we call the \emph{half-spin of $K$}. (See Figure \ref{fig:HalfSpin}.)  The double of this ribbon disk gives the spin of $K$: 
$$(B^4, D) \cup (B^4, D)\cong(S^4,\Ss(K)).$$

\begin{figure}[h!]
\centering
\includegraphics[width=.9\textwidth]{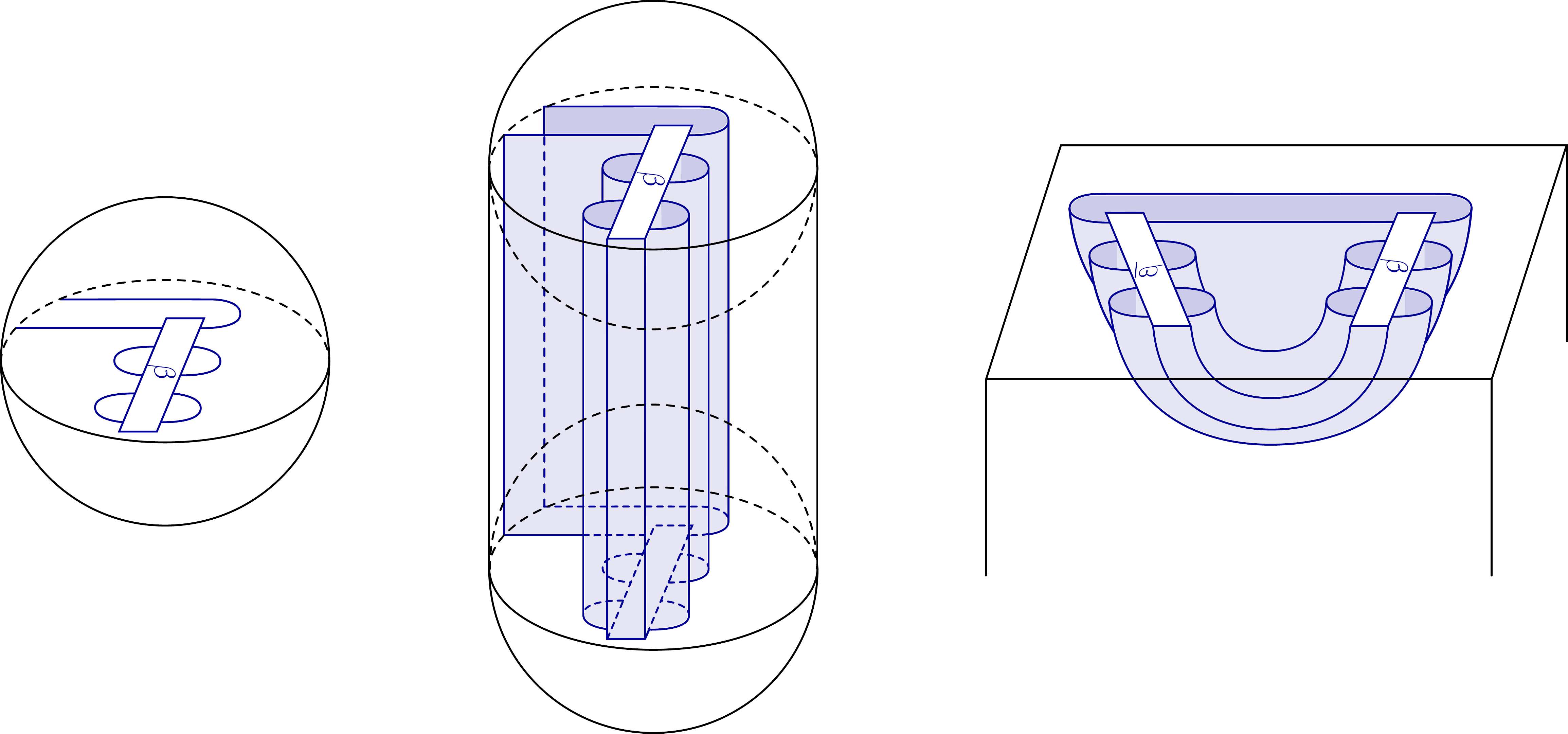}
\put(-355,-15){(a)}
\put(-235,-15){(b)}
\put(-90,-15){(c)}
\caption{(a) The punctured knot $(B^3,K^\circ)$, where $K$ is the plat closure of the braid $\beta$.  The half-spun disk $D$ shown (b) as the product $(B^3,K^\circ)\times I$ and (c) as the standard ribbon disk for $K\#(-K)$ in $B^4$.}
\label{fig:HalfSpin}
\end{figure}

The half-spun disk $D$ can be obtained by attaching bands to $K\#(-K)$.  Alternatively, we can turn this picture upside down and think of $D$ as the result of attaching bands to an unlink $L$ to produce $K\#(-K)$, as shown in Figure \ref{fig:SpinBands}(a).  The bands appearing in this latter view are dual to those appearing in the former.  To double $D$, we take two copies of this picture, one of which has been turned upside down, and glue them together. This corresponds to adding a dual band for each band in Figure \ref{fig:SpinBands}(a). Doing so, we arrive at the banded link description $(L,\ups)$ of $\Ss(K)$ shown in Figure \ref{fig:SpinBands}(b), where half of the bands in $\ups$ come from each of the copies of $D$, and one set has been dualized.

\begin{figure}[h!]
\centering
\includegraphics[scale = .3]{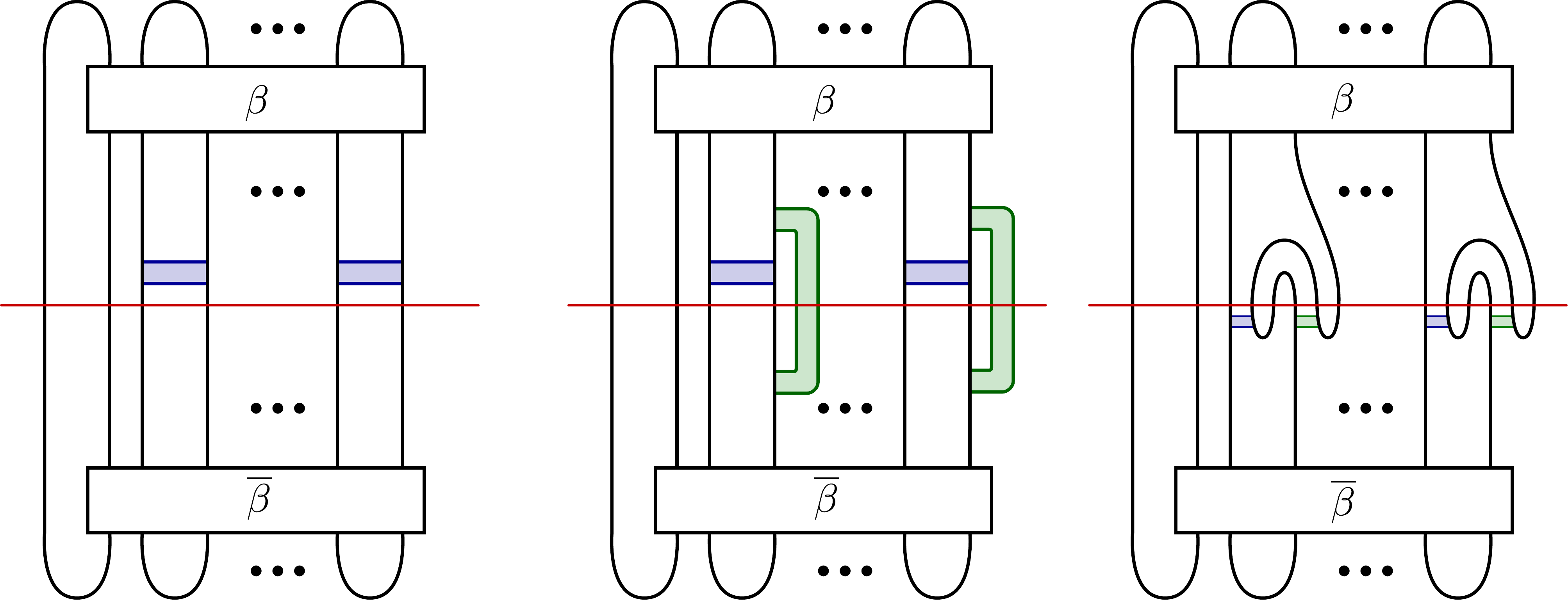}
\put(-300,-15){(a)}
\put(-175,-15){(b)}
\put(-60,-15){(c)}
\caption{(a) A description of the half-spin of $K$, where $K$ is the knot or link given as the plat closure of the braid $\beta$. (b) A banded link diagram for the spin $\Ss(K)$. A perturbation of the bridge sphere produces the banded bridge splitting shown in (c).}
\label{fig:SpinBands}
\end{figure}

Then next step is to transform this banded link presentation into a banded bridge splitting.  This is accomplished by perturbing the link $L$ so that the bands of $\ups$ are level in the bridge sphere and dual to a subset of the bridge disks for one tangle.  The resulting banded bridge splitting is shown in Figure \ref{fig:SpinBands}(c).

Lemma \ref{lemma:morsetri} describes how to transform this banded bridge splitting into a bridge trisection.  The associated tri-plane diagram is shown in Figure \ref{fig:SpinTrisection}. (See Section \ref{sec:exist} for details.)

\begin{figure}[h!]
\centering
\includegraphics[scale = .3]{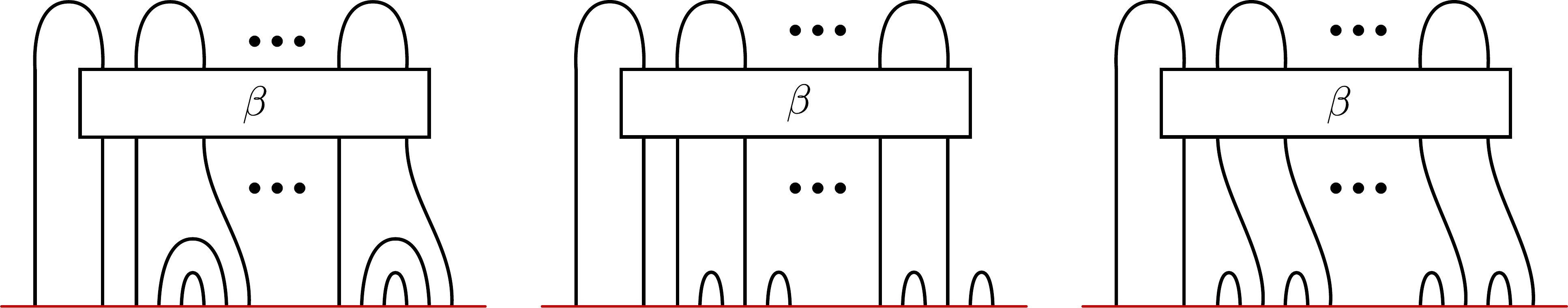}
\caption{A $(3b-2,b)$--bridge tri-plane diagram for the spin $\Ss(K)$ of the $b$--bridge knot or link $K$ given as the plat closure of the braid $\beta$.}
\label{fig:SpinTrisection}
\end{figure}


\subsection{Twist-spinning knots and links}\label{subsec:twist}\ 

In 1965, Zeeman gave a generalization of Artin's construction called \emph{twist-spinning} \cite{zeeman}.  Again, we refer the reader to \cite{CKS} for the standard development.  From our view-point, instead of doubling the half-spun disk as before, we will glue two copies of the half-spun disk together with a twisting diffeomorphism of the boundary.  We can realize this diffeomorphism as the time-one instance of an isotopy of $K^\circ$ in $B^3$.  To form the half-spin of $K$, we take the product $(B^3,K^\circ)\times I$, which we can think of as the trace of the identity isotopy of $K^\circ$ in $B^3$.  Now, we exchange the identity isotopy for one that rotates $K^\circ$ around its axis $m$ times.

\begin{figure}[h!]
\centering
\includegraphics[scale = .3]{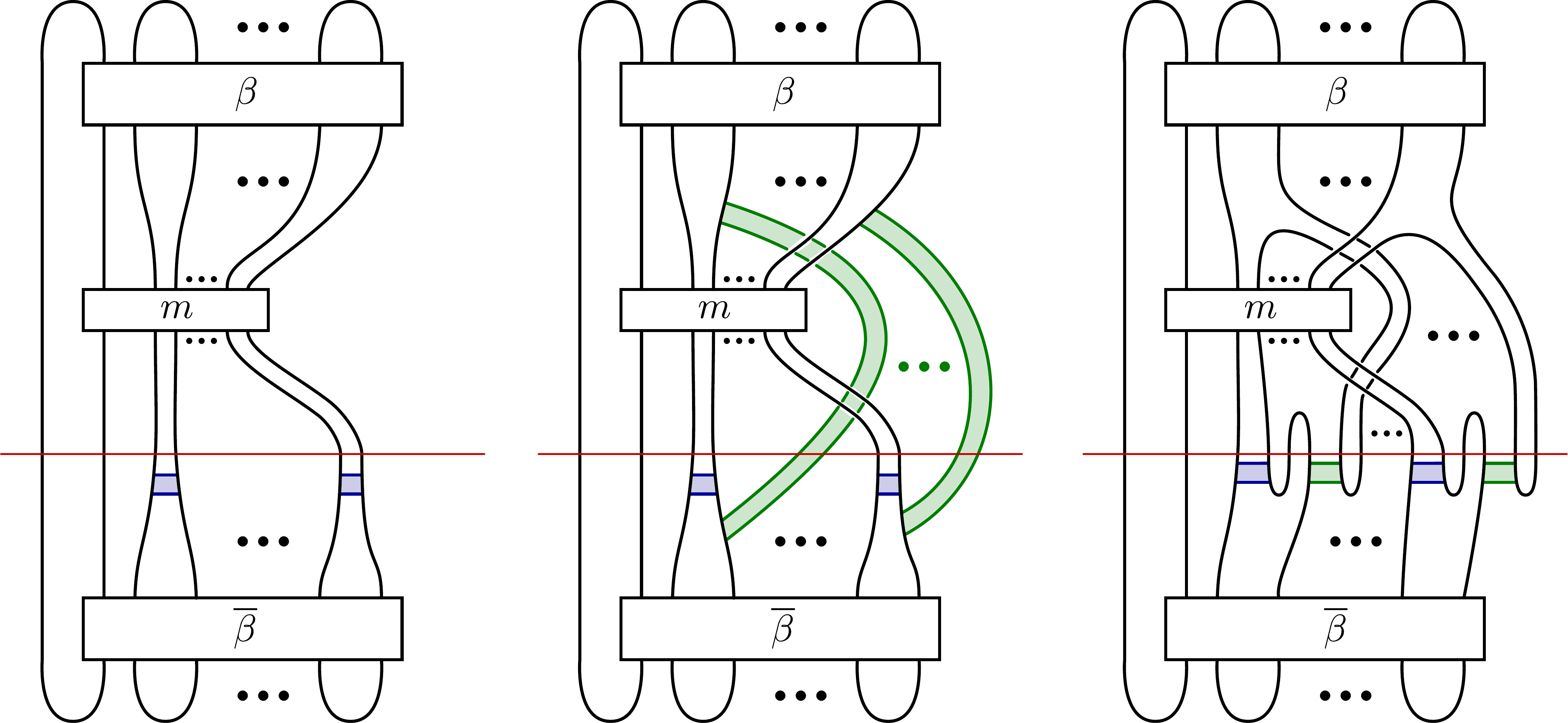}
\put(-320,-15){(a)}
\put(-193,-15){(b)}
\put(-65,-15){(c)}
\caption{(a) A description of the $m$--twisted half-spin of $K$, where $K$ is the knot or link given as the plat closure of the braid $\beta$. (b) A banded link diagram for the $m$--twist spin $\Ss(K)$. A perturbation of the bridge sphere produces the banded bridge splitting shown in (c).}
\label{fig:TwistBands}
\end{figure}

The result is a new knotted disk $D_m$. Just like the half-spin $D$, the disk $D_m$ is obtained by attaching bands to an unlink $L$ to form $K\#(-K)$. However, in the twisted case, the components of the unlink have been twisted $m$ times, as shown in Figure \ref{fig:TwistBands}(a).  Note that $\partial D=\partial D_m = K\#(-K)$ and that $D$ and $D_m$ are actually isotopic as properly embedded disks in the four-ball although they are not isotopic rel $\pd$.  Also, we recover the original half-spun disk when $m=0$; i.e., $D_0=D$.  It follows that $\Ss_0(K)=\Ss(K)$.

Because $D$ and $D_m$ have a common boundary but are not isotopic rel boundary, we can form a new knotted sphere by gluing these two disks along $K\#(-K)$.  The result is the \emph{$m$-twist spin of $K$}, which is denoted $\Ss_m(K)$: $$(S^4, \Ss_m(K))=(B^4, D)\cup(B^4, D_m).$$

A handle decomposition for $\Ss_m(K)$ is shown in Figure \ref{fig:TwistBands}(b).  As before, half of the bands correspond to each disk in the decomposition of $\Ss_m(K)=D_m\cup D$.  In Figure \ref{fig:TwistBands}(c), we have perturbed the bridge splitting of $L$ to obtain a banded bridge splitting.  The induced bridge trisection is shown in Figure \ref{fig:TwistTrisection}.

\begin{figure}[h!]
\centering
\includegraphics[scale = .3]{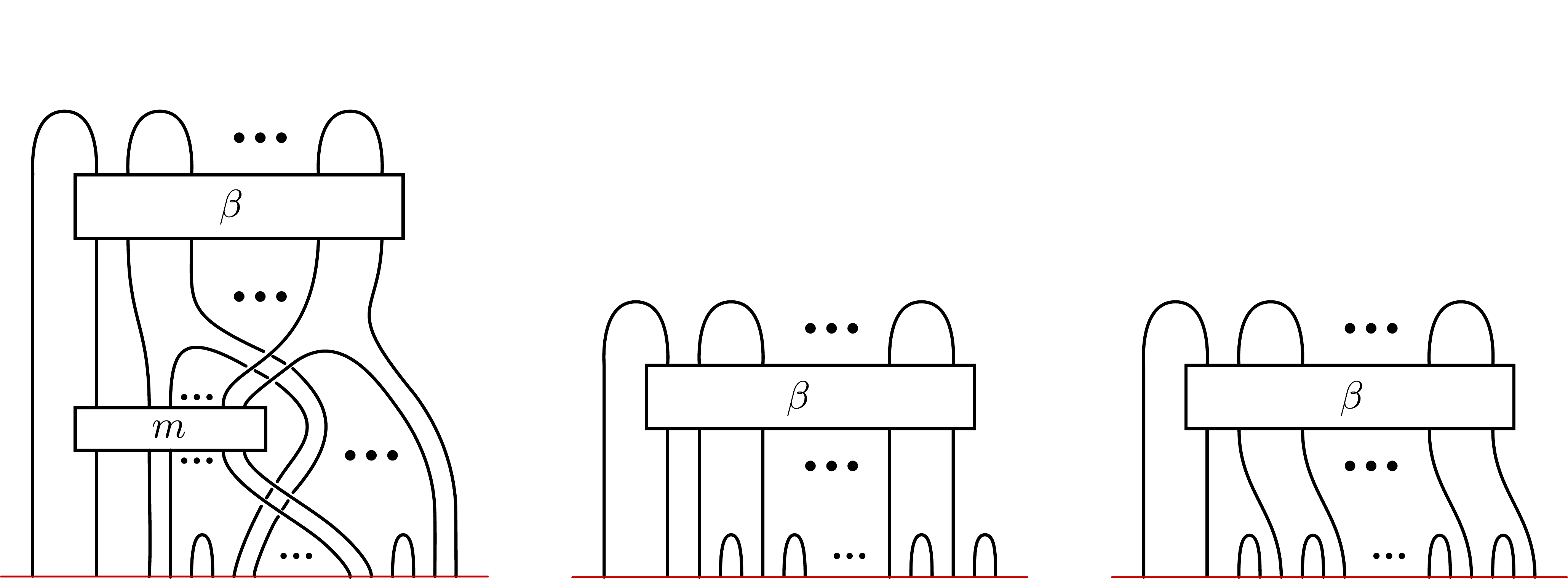}
\caption{A $(3b-2,b)$--bridge tri-plane diagram for the $m$--twist spin $\Ss_m(K)$ of the $b$--bridge knot or link $K$ given as the plat closure of the braid $\beta$.}
\label{fig:TwistTrisection}
\end{figure}

\subsection{Nontrivial bridge trisections}\label{subsec:nontrivial}\ 

Now that we have shown how to construct and trisect the spin $\Ss(K)$ and twist spin $\Ss_m(K)$ associated to a knot $K \subset S^3$, we can produce nontrivial surfaces in $S^4$ with arbitrarily large bridge number.  The following proposition is an immediate consequence of the discussion above and is shown in Figures \ref{fig:SpinTrisection} and \ref{fig:TwistTrisection}.

\begin{proposition}\label{prop:BridgeSpin}
Let $K$ be a $b$--bridge knot or link.  Then, for $m\in\Z$,  $\Ss_m(K)$ admits a $(3b-2,b)$--bridge trisection.
\end{proposition}


A natural question to ask is whether or not a minimal bridge splitting of $K$ gives rise to a minimal bridge trisection of $\Ss_m(K)$ or $\T(K)$. 

\begin{question}\label{question:bridge}
For which knots $b$--bridge knots $K$ does it hold that $b(\Ss_m(K)) = 3b-2$?
\end{question}

Note that if $m=\pm 1$, then $\Ss_m(K)$ is unknotted \cite{zeeman}, so $b(\Ss_{\pm 1}(K))=1$ for all $K$ in $S^3$.  In this case one can ask, with an eye toward Question \ref{question:unknot}, whether or not the resulting bridge trisection is stabilized. 

On the other hand, we will show that spun torus knots satisfy Question \ref{question:bridge}. Thus, for every $b\geq 1$ there are spun knots with $b(\Ss(K)) = 3b-2$.  In order to prove this, we need to obtain a better understanding of the fundamental group of the complement of a knotted surface in $S^4$ from the perspective of bridge trisections.

Let $\K$ be a knotted surface in $S^4$.  A \emph{meridian} for $\K$ is a curve isotopic in $\pd \overline{\nu(\K)}$ which is isotopic to $ \{\text{pt}\} \X S^1$, viewing $\overline{\nu(\K)}$ as $\K \X D^2$.  A generator $x\in\pi(\K)$ is called \emph{meridional} if $x$ is represented by a meridian of $\K$.  A presentation of $\pi(\K)$ is meridional if each generator in the presentation is meridional.  The \emph{meridional rank} of $\K$ is the minimum number of generators among meridional presentations of $\pi(\K)$ and will be denoted $\text{mrk}(\K)$.

Note that the presentation produced in Proposition \ref{prop:fund1} above is meridional. 

\begin{corollary}\label{coro:meridional}
	If $\K$ admits a $(b;c_1,c_2,c_3)$--bridge trisection, then $\text{mrk}(\K)\leq\min\{c_i\}$.
\end{corollary}

We can define meridional generators and meridional rank analogously for knots in $S^3$, and we have the following question, which generalizes a questions posed by Cappell-Shaneson and appearing as Problem 1.11 in \cite{kirby:problems}.

\begin{question}
	If $K$ is a $b$--bridge knot, then does $\text{mrk}(K)=b$?
\end{question}

It is clear that $b(K)\geq\text{mrk}(K)$.  As a corollary to their work on the $\pi$--orbifold group of a knot, Boileau-Zimmermann \cite{boileau-zimmermann} proved that $\mrk(K)=2$ for two-bridge knots, while Rost-Zieschang \cite{rost-zieschang} showed that $\mrk(T_{p,q})=b(T_{p,q})=\min\{p,q\}$.

As a final preliminary, we offer the following.

\begin{proposition}\label{prop:mrkspin}
	Let $K$ be a knot in $S^3$,  Then $\mrk(K)=\mrk(\Ss(K))$.
\end{proposition}

\begin{proof}
Recall the standard decomposition from the beginning of this section: $$(S^4, \Ss(K)) = ((B^3, \K^\circ)\times S^1)\cup((S^2,\{\mathbf n, \mathbf s\})\times D^2).$$
Consider the inclusion $\iota:(B^3,K^\circ)\hookrightarrow (S^4,\Ss(K))$ that maps $(B^3,K^\circ)$ to $(B^3, \K^\circ)\times\{0\}$. Let $\mu_1, \ldots, \mu_r$ be a collection of meridians of $K$ representing a meridional generating set for $\pi(K)$.  Then $\iota(\mu_i)$ is a meridian of $\Ss(K)$ for each $i$, and the induced map $\iota_*:\pi(K)\to\pi(\Ss(K))$ is an isomorphism.  It follows that $\iota_*(\mu_1),\ldots,\iota_*(\mu_r)$ is a meridional generating set for $\pi(\Ss(K))$, so $\mrk(\Ss(K))\leq\mrk(K)$.

Conversely, let $\mu'_1, \ldots, \mu'_s$ be a collection of meridians to $\Ss(K)$ representing a meridional generating set for $\pi(\Ss(K))$.  Let $\Gamma$ be the union of a base point $x_0$, the curves $\mu_1',\dots,\mu_s'$, and $s$ arcs connecting $x_0$ to $\mu_i'$.  Note that each $\mu_i'$ bounds a disk $D_i$ in $\overline{\nu(\Ss(K))}$.  There exists a 2-sphere $S$ in $S^4$ containing $\mathbf n$ and $\mathbf s$ such that $(S^4,\Ss(K)) \setminus \nu(S) \cong(B^3,K^{\circ}) \X S^1$.  After a small perturbation, we may assume that $\Gamma \cap S = \emp$, so that $\Gamma \subset (B^3,K^{\circ}) \X S^1$.  Let $\rho$ denote the natural projection of $(B^3,K^{\circ}) \X S^1$ onto the first factor.  After another small perturbation, $\rho(\Gamma)$ is embedded in $B^3 \setminus \nu(K^\circ)$, and $\Gamma$ is isotopic to $\rho(\Gamma)$ in $S^4 \setminus \nu(\Ss(K))$.

Now, each disk $D_i \subset \overline{\nu(\Ss(K))}$ projects to an immersed disk $\rho(D_i)$ in $\overline{\nu(K^{\circ})}$.  It follows from Dehn's Lemma that $\iota^{-1}(\mu_1),\ldots,\iota^{-1}(\mu_r)$ is a collection of meridians of $K$, and since $\iota_*^{-1}$ is an isomorphism, this is a generating set of meridians for $\pi(K)$.  It follows that $\mrk(K)\leq \mrk(\Ss(K))$, and the proof is complete.  

\end{proof}

We are now well-equipped to prove our next result.

\begin{reptheorem}{thm:nontrivial}
	There exist infinitely many distinct 2--knots with bridge number $3b-2$ for any $b\geq 2$.
\end{reptheorem}

\begin{proof}
	Let $K$ be the torus knot $T_{b,q}$ with $b<q$, so $\mrk(K)=b(K)=b$. Let $\K=\Ss(K)$. By Proposition \ref{prop:BridgeSpin}, $\K$ admits a $(3b-2,b)$--bridge trisection, and by Corollary \ref{coro:meridional}, $\mrk(\K)\leq b$.
	
	On the other hand, by Lemma \ref{prop:mrkspin}, $\mrk(\K)=\mrk(K)=b$. It follows that $b(\K)=3b-2$.  Since there are infinitely many torus knots of the form $T_{b,q}$ for each $b\geq 2$, the result follows.
\end{proof}

We remark that one could easily prove an analogous result involving knotted tori using the turned torus construction, showing that there are infinitely many knotted tori with bridge number $3b$ for each $b\geq 2$.  On the other hand, it is a little less clear how one would extend these results to non-orientable surfaces.  This would be the final step in showing that there are infinitely many knotted surfaces with bridge number $b$ for each $b\geq 4$.

\section{Stabilization of bridge trisections and banded bridge splittings}\label{sec:stab}

In this section, we define stabilization operations for both bridge trisections and banded bridge splittings and prove that our definitions are equivalent.  We will use the banded bridge splitting version of stabilization to prove Theorem \ref{thm:uniqueness} in Section~\ref{sec:unique}.

\subsection{Stabilization of bridge trisections}\label{subsec:stab}\ 

Suppose $\K$ is a knotted surface in $S^4$ equipped with a $b$--bridge trisection $\T$, where components of $\T$ are labeled as above.  Choose one of the trivial disk systems, say $(X_1,\D_1)$.  Recall that in Subsection \ref{subsec:introexs}, we considered $S^4$ to be $\{(x_1,\dots,x_5): x_1^2 + \dots x_5^2 =1\}$ in $\R^5$, and we let $Y = \{(x_1,\dots,x_5)\in S^4:x_5=0\}$, so that $Y \cong S^3$.   Suppose $\rho: S^4 \rightarrow Y$ is the natural projection map, and let $Z_i = \rho(X_i)$.  We may arrange the standard trisection of $S^4$ so that if $E_{ij} = \rho(B_{ij}) = B_{ij} \cap Y$, then $E_{12} \cup E_{23} \cup E_{31}$ is a tri-plane in $S^4$ which cuts $Y$ into the three 3--balls $Z_1$, $Z_2$, and $Z_3$, and $\Sigma \cap Y = \rho(\Sigma) = \pd E_{ij}$. 

By Proposition \ref{linkbridge}, the bridge splitting $(\pd X_1,\pd \D_1) = (B_{12},\A_{12}) \cup_{\Sigma} (B_{31},\A_{31})$ is standard, and as such, there is an isotopy of $\D_1$ so that $L_1 = \pd \D_1 \subset E_{12} \cup E_{31}$.  In other words, we are reasserting the fact that $\K$ has a tri-plane diagram $\Pau = (\Pau_{12},\Pau_{23},\Pau_{31})$ such that $\Pau_{12}$ and $\Pau_{31}$ contain no crossings.  It follows that $L_1$ may be capped off with disks $D_1 \subset Z_1$, which implies that $\D_1$ is isotopic to $D_1$.  Stated another way, we may isotope $\K$ so that for a small neighborhood $\nu(Z_1)$ in $Y$, we have $\K \cap \nu(Z_1) = \rho(\K) \cap \nu(Z_1)$. 

Now, let $\Delta$ be a disk embedded in $Z_1$ which has the following properties:
\be
\item The boundary $\pd \Delta$ is the endpoint union of arcs $\delta_1 \subset D_1$, $\delta_{12} \subset E_{12}$ and $\delta_{31} \subset E_{31}$.
\item $\Delta \cap D_1 = \delta_1$,
\item $\Delta \cap \pd Z_1 = \delta_{12} \cup \delta_{31}$.
\ee
It follows that $\Delta$ meets $\rho(\Sigma)$ in a single point $p$, and we call $\Delta$ a \emph{stabilizing disk}.  To define stabilization, we will consider the standard trisection of $S^4$ to be fixed and isotope $\K$.  For a stabilizing disk $\Delta$, there is an isotopy of $\K$ supported in $\nu(Z_1)$ that consists of pushing $D_1$ across $\Delta$ into $Z_2 \cup Z_3$.  Let $\K'$ be the result of this isotopy, and let $\D_i' = \K' \cap X_i$.  See Figures \ref{fig:StabDisk} and \ref{fig:StabDisk2}.

\begin{figure}[h!]
  \centering
    \includegraphics[width=.8\textwidth]{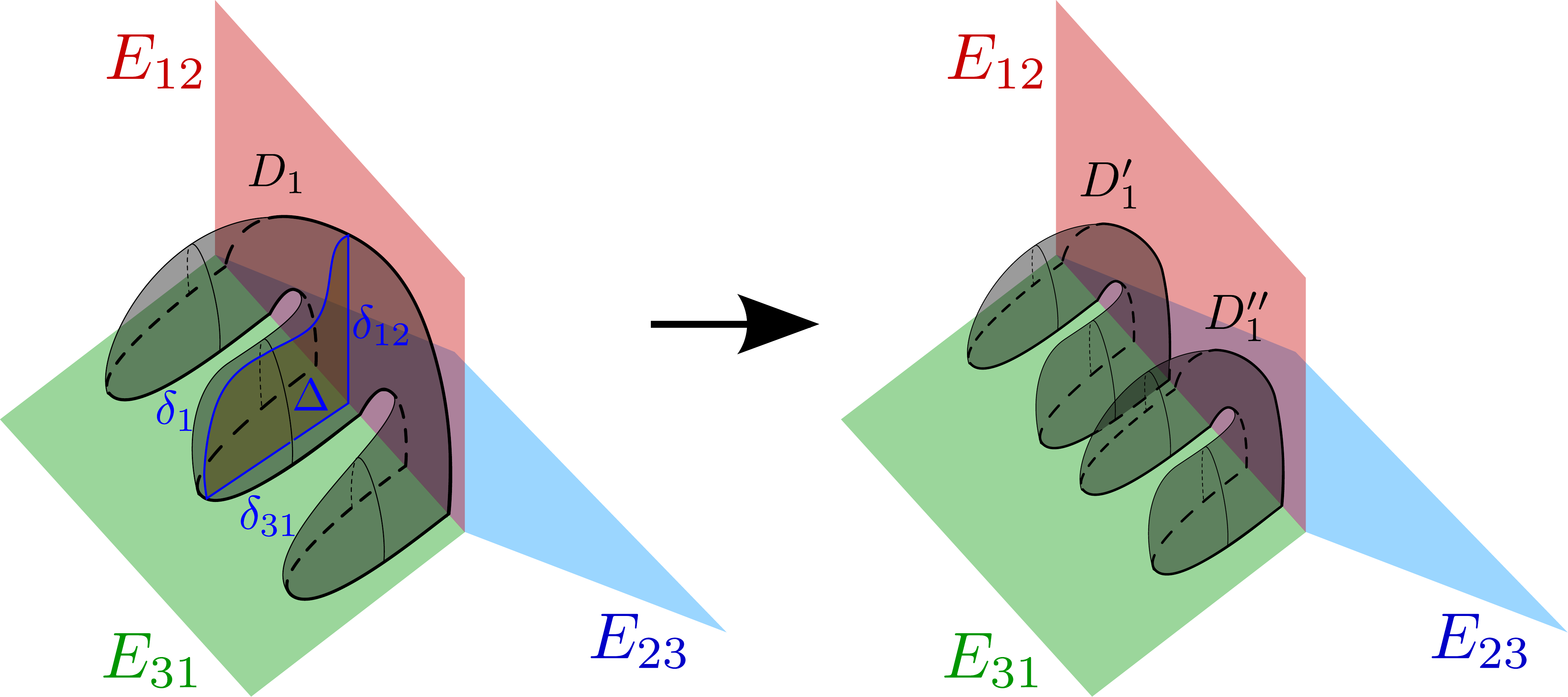}
    \caption{An example of the stabilization operation, which can be thought of as a boundary compression along $\Delta$ that transforms $D_1$ into $D_1'\cup D_1''$.}
    \label{fig:StabDisk}
\end{figure}

\begin{figure}[h!]
  \centering
    \includegraphics[width=.8\textwidth]{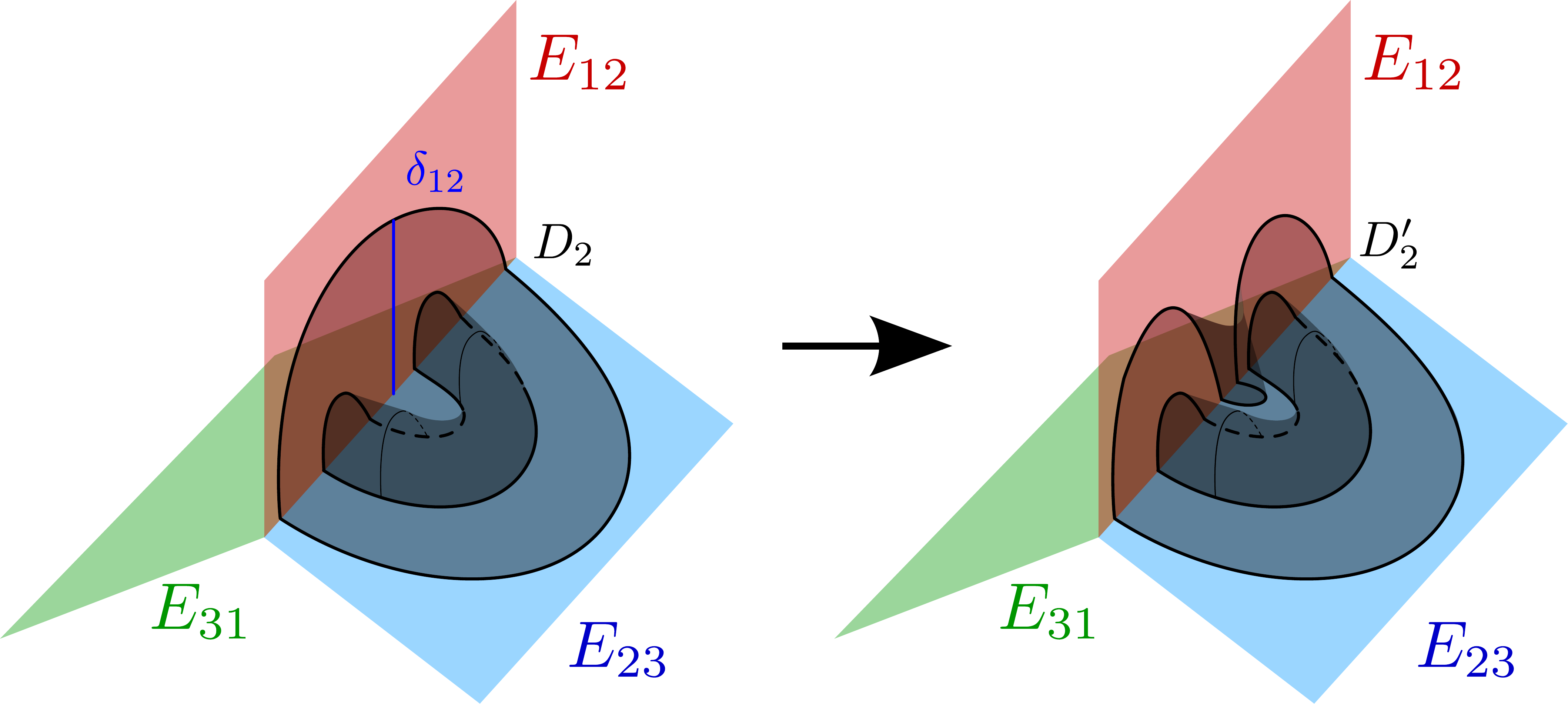}
    \caption{Compressing along a stabilizing disk in $Z_1$ corresponds to dragging a portion of a disk $D_2$ through the bridge sphere in $Z_2$ (and similarly in $Z_3$).}
    \label{fig:StabDisk2}
\end{figure}

\begin{lemma}\label{stabilized}
The decomposition $\T'$ given by $(S^4,\K') = (X_1,\D_1') \cup (X_2,\D_2') \cup (X_3,\D_3')$, is a $(b+1)$--bridge trisection of $(S^4,\K')$.
\end{lemma}

\begin{proof}
First, we will show that each $(X_i,\D_i')$ is a trivial disk system.  The collection $\D_2'$ is obtained from $\D_2$ by dragging a disk component of $\D_2$ along the arc $\delta_{12}$ in $\nu(Z_1) \cap X_2$; thus, $(X_2,\D_2')$ is a trivial $c_2$--disk system (see Figure \ref{fig:StabDisk2}). Similarly, $\D_3'$ is obtained from $\D_3$ by dragging a component along $\delta_{31}$, so that $(X_3,\D_3')$ is a trivial $c_3$--disk system.  The collection $\D_1'$ is obtained from $\D_1$ by surgering a disk component of $D_1$ along the disk $\Delta$. (Such an operation is commonly called \emph{boundary-compressing}.)  This result of this boundary compression on $D_1$ is a pair of disks $D_1'$ and $D_1''$.  See Figure \ref{fig:StabDisk}. Since $\D_1$ is isotopic to a collection of $c_1$ disks properly embedded in the 3--ball $Z_1$, the result $\D_1'$ of boundary compressing $\D_1$ along $\Delta$ is a collection of $c_1 + 1$ disks in $Z_1$.  These disks are necessarily trivial, and it follows that $(X_1,\D_1')$ is a trivial $(c_1+1)$--disk system.

Let $\A_{ij}' \subset B_{ij}$ denote $\D_i' \cap \D_j'$.  To complete the proof, we must show that each $(B_{ij},\A_{ij}')$ is a trivial tangle.  Note first that $\A_{23}'$ differs from $\A_{23}$ by a single trivial arc, the boundary of a small neighborhood in $E_{23}$ of the point $p = \Delta \cap \rho(\Sigma)$, and thus $(B_{23},\A_{23}')$ is a trivial tangle.  Considering $\A_{12}'$, we note that the arc $\delta_{12} \subset E_{12}$ meets $\A_{12}$ in a single point, and as such there is a bridge disk $\Delta_{12}$ for $\A_{12}$ which contains $\delta_{12}$.  In addition, $\A_{12}'$ results from doing surgery on $\A_{12}$ along $\delta_{12}$, which splits $\Delta_{12}$ into two bridge disks and leave all other bridge disks for $\A_{12}$ intact.  It follows that $(B_{12},\A_{12}')$ is a trivial tangle, and a parallel argument shows that $(B_{31},\A_{31}')$ is trivial as well.  Finally, we have $|\K' \cap \Sigma| = |\K \cap \Sigma| + 2$, completing the proof.
\end{proof}

We say the $(b+1)$--bridge trisection $\T'$ is \emph{stabilized}.  Since this construction is not symmetric in the indices $i,j,k$, when necessary we will say that $\T'$ is an \emph{elementary stabilization} of $\T$ \emph{toward} $B_{23}$.  We call any bridge trisection which is the result of some number of elementary stabilizations a \emph{stabilization} of $\T$.  Note that stabilization depends heavily on the choice of the stabilizing disk $\Delta$.

Observe that the stabilization process described in Lemma \ref{stabilized} creates a new bridge disk $\Delta' \subset E_{23}$ for the new arc in $\A_{23}'$.  This disk has the property that $\Delta' \cap \Delta = \{p\}$, and if we isotope $\K'$ along $\Delta'$ (i.e. perform a boundary-compression of $\K' \cap \nu(Z_1)$ along $\Delta'$), we recover our original surface $\K$ and original $b$--bridge trisection of $\K$.  For this reason, we will call $\Delta'$ a \emph{destabilizing disk}. 

Destabilizing disks play a role in the next lemma, which characterizes stabilization in terms of collections of shadow arcs (see Figure \ref{fig:shadowstab}).

\begin{figure}[h!]
  \centering
    \includegraphics[width=.4\textwidth]{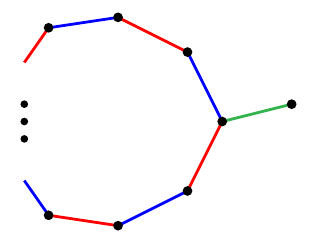}
    \caption{The arrangement of shadow arcs necessary and sufficient for a bridge trisection to be stabilized, as shown in Lemma \ref{shadowstabilized}.}
    \label{fig:shadowstab}
\end{figure}

\begin{lemma}\label{shadowstabilized}
A $(b+1)$--bridge trisection $\T'$ (with components labeled as above) is stabilized if and only if there exist
\be
\item sequences of shadows $a^*_1,\dots,a^*_n$ for arcs in $\A_{ij}'$ and  $b^*_1,\dots,b^*_n$ for arcs in $\A_{ki}'$ such that $a^*_1 \cup b^*_1 \cup \dots \cup a^*_n \cup b^*_n$ is a simple closed curve in $\Sigma$, and
\item a shadow $c^*$ for $\A_{jk}'$ that meets $a^*_1 \cup \dots \cup b^*_n$ in a single point which is one of its endpoints.  
\ee
\end{lemma}
\begin{proof}
Suppose first without loss of generality that $\T'$ is an elementary stabilization of another bridge trisection $\T$ in $X_1$.  Arranging $\T$ as above, there is a disk $\Delta$ embedded in $Z_1$ with the property that $\Delta \cap \Sigma$ is a single point $p$.  To obtain a collection of bridge disks for $\A_{12}'$ and $\A_{31}'$ in $\T'$, we surger bridge disks for $\A_{12}$ and $\A_{13}$ along the arcs $\delta_{12}$ and $\delta_{31}$ in $\partial \Delta$.  Let $D$ denote the component of $D_1$ which meets $\Delta$.  Since $D$ is isotopic into $E_{12} \cup E_{31}$, we may find a collection of shadows for the arcs in $D$ whose union is a simple closed curve in $\Sigma$.  This is not quite the collection of shadows we will need to perform the surgery; let $a_1,b_1,\dots,a_m,b_m$ be a collection of shadows such that $a_l$ is a shadow for $\A_{12}$, $b_l$ is a shadow for $\A_{31}$, and $a_1 \cup b_1 \cup \dots a_m \cup b_m$ is the wedge of two circles, a simple closed curve pinched along the point $p$ in the interior of arcs $a_r$ and $b_s$. 

In this setting, we may view the construction of $\T'$ as splitting the point $p$ into two points, $p_1$ and $p_2$.  This splits the shadow $a_r$ into two arcs $a_r'$ and $a_r''$ and splits $b_s$ into $b_s'$ and $b_s''$, where each of these new arcs is a shadow for $\A_{12}'$ or $\A_{31}'$.  Moreover, this stabilization process creates a new trivial arc in $\A_{23}'$, which has a shadow $c^*$ connecting $p_1$ to $p_2$ and avoiding the arcs $a_l$ and $b_l$.  Since this process also splits the wedge of two circles $a_1 \cup b_1 \cup \dots \cup a_m \cup b_m$ into two disjoint curves, we conclude that there are sequences of shadow arcs $a^*_1,\dots, a^*_n$ and $b^*_1,\dots,b^*_n$ whose union is a simple closed curve meeting $c^*$ in a single point of $\K \cap \Sigma$. In fact, the proof reveals that there are two such sequences.  See Figure \ref{fig:shadowstab2}.

\begin{figure}[h!]
  \centering
    \includegraphics[width=1\textwidth]{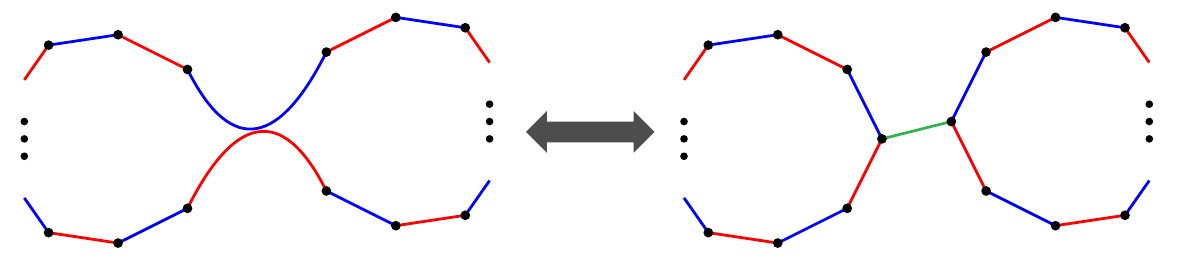}
    \caption{A simple closed curve of shadow arcs (left) which is pinched into two curves and a new arc $c^*$ during the stabilization process (right).}
    \label{fig:shadowstab2}
\end{figure}

For the reverse direction, suppose that there are sequences of shadows $a_1^*,\dots,a_n^*$ for $\A_{12}'$ and $b_1^*,\dots,b_n^*$ for $\A_{31}'$ whose union is a curve in $\Sigma$ which meets a shadow $c^*$ for $\A_{23}'$ in a single endpoint.  Then there is a component $D$ of $\D_1'$ which is isotopic to a disk $D^*$ in $\Sigma$ bounded by $a_1^* \cup b_1^* \cup \dots \cup a_n^* \cup b_n^*$.  After a standard cut-and-paste argument (see the proof of Theorem \ref{thm:existence}), we may assume that the interior of $D^*$ contains no point of $\K' \cap \Sigma$.  By Proposition \ref{linkbridge}, the splitting $(\pd X_1,\pd \D_1) = (B_{12},\A_{12}') \cup_{\Sigma} (B_{31},\A_{31}')$ is standard.  By assumption, the endpoint of the arc $c^*$ that is not contained in $D^*$ must be contained in another component $D'$ of $\D_1$, and again using cut-and-paste techniques, we see that $D'$ is isotopic to a disk $(D')^*$ contained in $\Sigma$ such that $(D')^* \cap D^* =\emp$ and the interior of $(D')^*$ contains no points of $\K' \cap \Sigma$. 

Finally, we may arrange $Y \cong S^3 \subset S^4$ and a tri-plane $E_{12} \cup E_{31} \cup E_{23}$ so that, after pushing $D^*$ and $(D'^*)$ into $Z_1$, $\pd D^* \cup \pd (D')^* \subset E_{12} \cup E_{31}$, and the shadow $c^*$ arises from a bridge disk $\Delta'$ for an arc $c' \in \A_{23}'$ which is a slight pushoff of $c^*$ into $E_{23}$ contained in $\nu(Z_1)$.  Boundary-compressing $\K' \cap \nu(Z_1)$ along $\Delta'$ into $Z_1$ merges $D^* \cup (D')^*$ into a single disk $D''$ and gives rise to a new boundary compressing disk $\Delta$ which satisfies the conditions above.  We leave it to the reader to check the details that $\Delta'$ is a destabilizing disk and that this is precisely the inverse operation of stabilization. The result is a $b$--bridge trisection $\T$ such that stabilizing $\T$ along the disk $\Delta$ again yields $\T'$.
\end{proof}

We call the operation described in Lemma \ref{shadowstabilized} \emph{destabilization}.  Succinctly, if $\T'$ is a $(b+1)$--trisection with shadows satisfying the conditions of Lemma \ref{shadowstabilized}, destabilization is the process of boundary-compressing $\K'$ along a bridge disk giving rise to the shadow $c^*$, yielding a $b$--bridge trisection.  By Lemma \ref{shadowstabilized}, we see that stabilizations may be quantitatively different, depending on the cardinality $n$ of the sequences $a_1^*,\dots,a_n^*$ and $b_1^*,\dots,b_n^*$ of shadow arcs.  To emphasize the value of $n$, we will sometimes say that a stabilized bridge trisection $\T$ is \emph{$n$--stabilized}. 

We may also depict stabilization and destabilization on the level of tri-plane diagrams.  To $n$--stabilize a tri-plane diagram $\Pau$, suppose that $D_1$ is a component of $\D_1$ such that $D_1 \cap (E_{12} \cup E_{31})$ is the standard diagram pictured in Figure \ref{fig:b1Diag}, bounding a disk $D$ in $E_{12} \cup E_{31}$.  Let $\Delta$ be a stabilizing disk, observing that $\Delta$ cuts $D$ into two components $D'$ and $D''$.  By the proof of Lemma \ref{shadowstabilized}, we need only consider one of the components, since the existence of one implies the existence of the other.  We may construct a new tri-plane diagram $\Pau'$ for the stabilized bridge trisection $\T'$ by surgering $\Pau_{12}$ along the arc $\delta_{12}$, surgering $\Pau_{31}$ along the arc $\delta_{31}$, and adding the boundary arc $(\partial\nu(p)) \cap E_{23}$ to $\Pau_{23}$.  If $D'$ contains $2n$ points of $\K \cap \Sigma$, then $\Pau'$ is $n$--stabilized.  See Figure \ref{diagramstable}.  Note that a stabilized bridge trisection is both $n'$--stabilized and $n''$--stabilized, for parameters $n'$ and $n''$ coming from each component $D'$ and $D''$ of the disk $D$ cut along $\Delta$. 

\begin{figure}[h!]
  \centering
    \includegraphics[width=.9\textwidth]{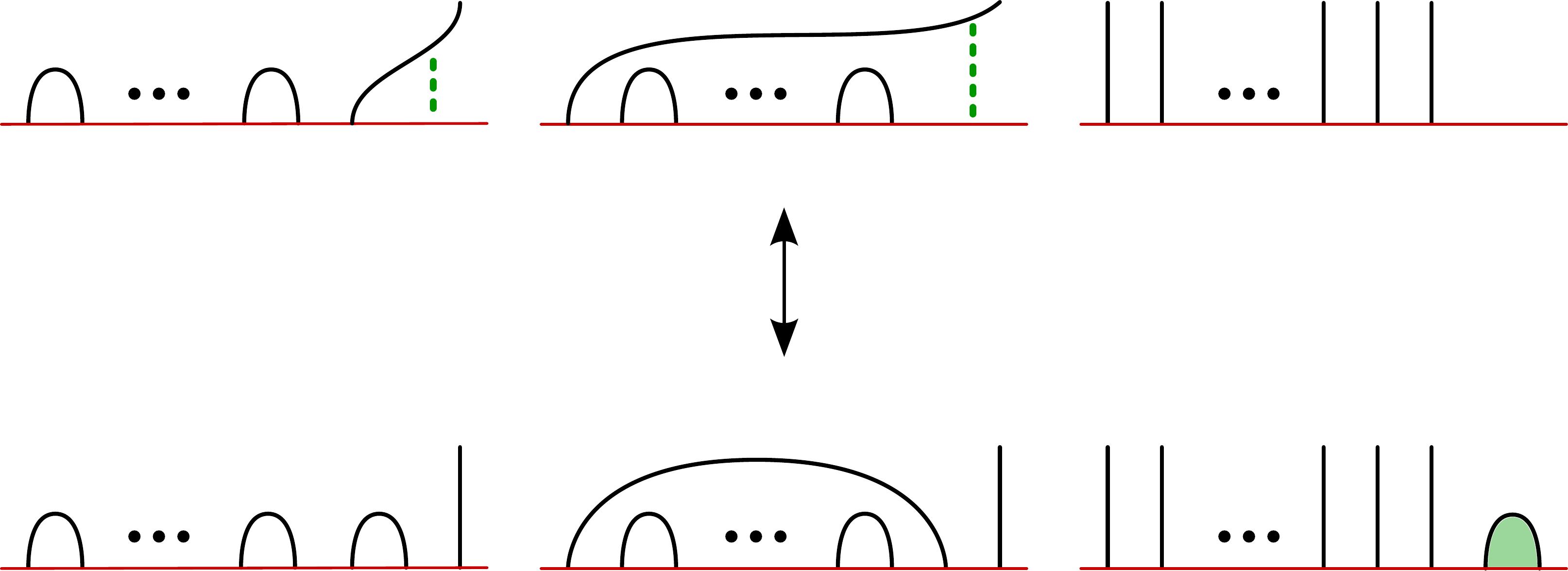}
    \caption{A general depiction of the stabilization and destabilization operation.  Distinct operations result from varying the number of arcs represented by the dots.}
    \label{diagramstable}
\end{figure}

In the reverse direction, we also see how to destabilize a tri-plane diagram:  Suppose that a bridge trisection $\T'$ has a diagram $\Pau' = (\Pau_{12}',\Pau_{23}',\Pau_{31}')$ with the property that $\Pau_{12}' \cup \Pau_{31}'$ contains a standard component as in Figure \ref{fig:b1Diag}, giving rise to a sequence of shadows $a_1^*,b_1^*,\dots,b_{n-1}^*,a_n^*$ whose union is an embedded arc in $e = \pd E_{ij}$, and $\Pau_{31}'$ contains a crossing-less arc with shadow $c^*$ which meets $a_1^* \cup b_1^* \cup \dots \cup b_{n-1}^* \cup a_n^*$ in a single endpoint of $a_n^*$.  Then there is an arc in $\Pau_{31}$ with shadow $b_n^*$ such that $a_1^* \cup b_1^* \cup \dots \cup a_n^* \cup b_n^*$ is an embedded curve in $\Sigma$ meeting $c^*$ in a single endpoint, and we see that $\T'$ is $n$--stabilized by Lemma \ref{shadowstabilized}.  Further, the bridge disk yielding $c^*$ in $\Pau_{23}'$ is a destabilizing disk, and we may destabilize the diagram $\Pau'$ by pushing this arc through $\Sigma$ as shown in Figure \ref{diagramstable}. 

As an example, the two-bridge diagram pictured in Figure \ref{fig:2Bridge} is 1--stabilized, and destabilizing yields the diagram in Figure \ref{fig:1Bridge}. A more general 1--stabilization is shown in Figure \ref{fig:4Stab}. For an interesting example, consider the 4--bridge trisection $\T_4$ of the unknot shown in the bottom half of in Figure \ref{fig:42Stab}.  Note that if a bridge trisection $\T$ is $n$--stabilized, then a component of some $L_i$ contains $n$--bridges.  For the bridge trisection $\T_4$, each component of $L_1$, $L_2$, and $L_3$ is in two-bridge position; therefore, $\T_4$ cannot be 1--stabilized.  However, the diagram \emph{is} 2--stabilized, and destabilizing yields the 3--bridge diagram at top of Figure \ref{fig:42Stab}.  This particular stabilization corresponds with the operation shown in Figures \ref{fig:StabDisk} and \ref{fig:StabDisk2}.

\begin{figure}[h!]
  \centering
    \includegraphics[width=.9\textwidth]{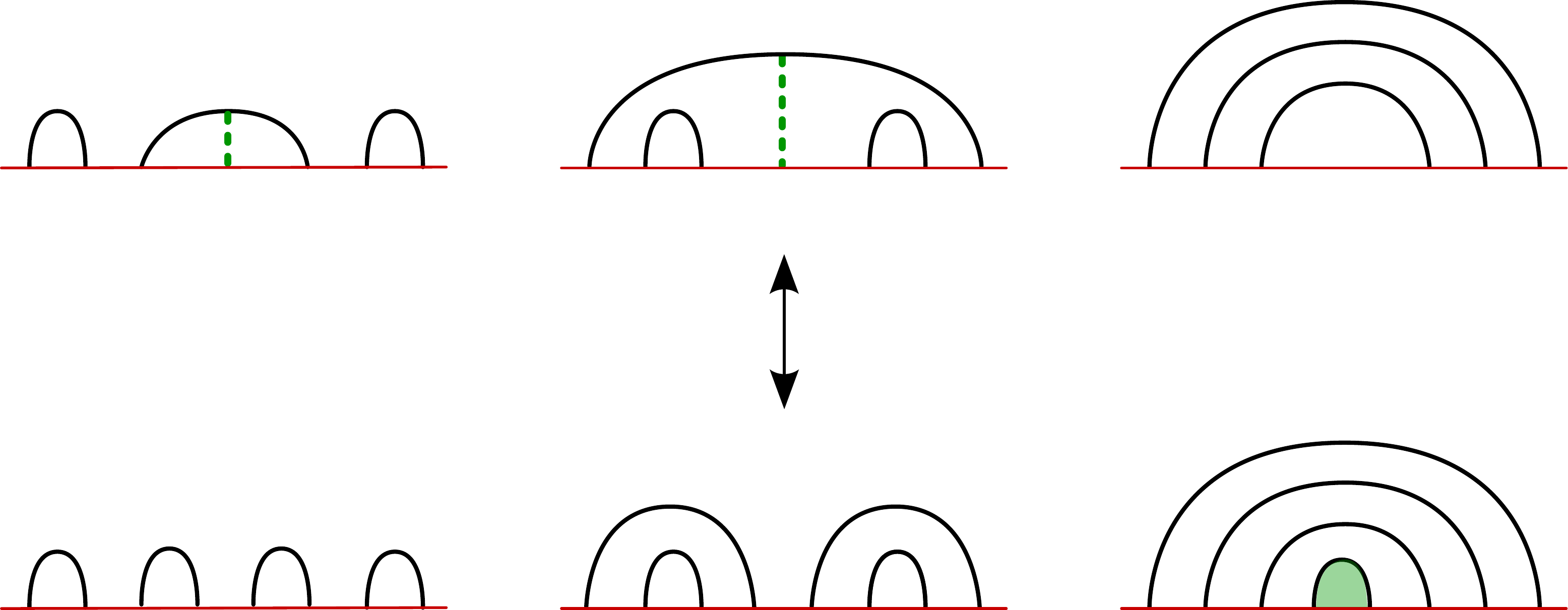}
    \caption{An example of a 2--stabilization.  The bottom diagram is a $(4,2)$--tri-plane diagram for the unknotted 2--sphere and is not 1--stabilized.}
    \label{fig:42Stab}
\end{figure}

\subsection{Stabilization of banded bridge splittings}\ 

Suppose that $\BB$ is a banded $b$--bridge splitting given by $(S^3,L,\ups) = (B_{12},\A_{12}) \cup_{\Sigma,y^*} (B_{31},\A_{31})$, and consider a point $p \in L \cap \Sigma$.  We may perturb $\Sigma$ at $p$ to obtain a new banded $(b+1)$--bridge splitting $\BB'$ given by $(S^3,L,\ups) = (B_{12}',\A_{12}') \cup_{\Sigma',y'} (B_{31}',\A_{31}')$.  We call $\BB'$ an \emph{elementary stabilization} of $\BB$, and we call any banded bridge splitting which is the result of some number of elementary stabilizations a \emph{stabilization} or \emph{stabilized}. 

If the point $p \in L \cap \Sigma$ at which the perturbation is carried out is not an endpoint of an arc in $y^*$, then up to isotopy, elementary stabilization is unique.  However, if $p$ is an endpoint of an arc $y_p \in y^*$, then there are two distinct ways in which we may construct $\BB'$:  Let $y_p'$ denote the arc in $y'$ corresponding to $y_p$, and let $a_{12}' \in \A_{12}'$ and $a_{31}' \in \A_{31}'$ be the canceling pair of arcs created by perturbation.  Then $y_p'$ shares an endpoint with either $a_{12}'$ or $a_{31}'$.  If $y_p'$ shares an endpoint with $a_{12}'$, we say $\BB'$ is stabilized \emph{toward $B_{31}'$}.  On the other hand, if $y_p'$ shares an endpoint with $a_{31}'$, we say $\BB'$ is stabilized \emph{toward $B_{12}'$.}  See Figure \ref{estab}. 

\begin{figure}[h!]
  \centering
    \includegraphics[width=.7\textwidth]{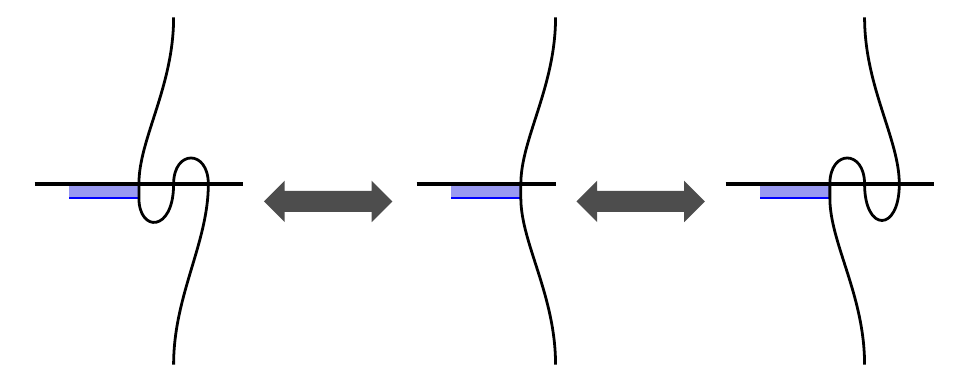}
    \caption{The two ways to perturb the banded link $L$ near an endpoint of $y^*$.  At left, a stabilization toward $B_{31}'$, and at right, a stabilization toward $B_{12}'$.}
    \label{estab}
\end{figure}

At this point, we have defined stabilization for both banded bridge splittings and bridge trisections; hence, now we must show that our definitions are equivalent via the correspondences between these objects introduced in Section \ref{sec:exist}.

\begin{proposition}\label{stables}
Suppose that $\T$ is a bridge trisection of $(S^4,\K)$.  Then $\T$ is stabilized if and only if there is a banded link $(L,\ups)$ with a stabilized banded bridge splitting $\BB$ such that $\T = \T(\BB)$.
\end{proposition}

\begin{proof}
Suppose first that there is a banded link $(L,\ups)$ with a stabilized banded bridge splitting $\BB$ such that $\T = \T(\BB)$, and let $\BB$ be given by $(S^3,L,\ups) = (B_{12},\A_{12}) \cup_{\Sigma,y^*} (B_{31},\A_{31})$.  Suppose that $\BB$ is stabilized at the point $p \in L \cap \Sigma$, which is an endpoint of canceling arcs $a_{12} \in \A_{12}$ and $a_{31} \in \A_{31}$ with shadows $a_1^*$ and $c^*$, respectively, which meet in a single point (namely, $p$).  By the definition of stabilization, we may assume that $a_1^*$ is contained in the collection $\A_{12}^*$ of shadows dual to $y^*$ and that $\text{int}(c^*) \cap \A_{12}^* = \emp$.  There are three cases to consider:  First, suppose that none of the endpoints of $a_{12}$ nor $a_{31}$ is the endpoint of an arc in $y^*$.  By the proof of Lemma \ref{lemma:trimorse}, we have $\A_{23}$ is isotopic rel boundary to $(\A_{12})_{\ups}$ (after pushing $\ups$ into $B_{12}$).  Thus, if $a_{12}$ does not meet a band in $\ups$, we have $a_{12}$ is in $\A_{23}$ as well, and so $a_1^*$ is also a shadow for $\A_{23}$.  By letting $b_1^*$ be a slight pushoff of $a_1^*$, we see that $\T$ is 1--stabilized. 

The second case is similar: Suppose that there is an arc $y^*_p$ which shares an endpoint with $c^*$.  Then no arc of $y^*$ meets $a_{12}$, so once again $a_{12} \in \A_{23}$ and we can see that $\T$ is 1-stabilized. 

Finally, suppose that an arc $y^*_p$ shares an endpoint with $a^*_1$, and let $y^*_1 = y^*_p$.  Since the shadows $\A_{12}^*$ are dual to $y^*$, we have $a_1^*$ and $y_1^*$ are contained in an embedded arc component $C$ of $\A_{12}^* \cup y^*$.  Moreover, since $\BB$ is stabilized, the other endpoint of $a_{12}$ cannot meet $y^*$, and thus we may describe $C$ as $a_1^* \cup y_1^* \cup \dots \cup y_{n-1}^* \cup a_n^*$, where arcs occur in order of adjacency.  By the proof of Lemma \ref{lemma:morsetri}, if $y^*_n$ is a slight pushoff of $C$ away from its endpoints, then $y^*_n$ is a shadow for an arc of $\A_{23}$.  Finally, since $a_1^* \cup c^*$ meets at most one arc in $y^*$ and $\text{int}(c^*) \cap \A_{23}^* = \emp$, it follows that $c^*$ intersects the simple closed curve $C \cup y_n^*$ in a single point, its endpoint $p$, and we conclude that $\T$ is $n$--stabilized, as desired.


For the reverse implication, suppose that $\T$ is $n$--stabilized for some $n$, so that there are three collections of bridge disks for $(B_{ij},\A_{ij})$, subsets of which meet $\Sigma$ in arcs $\{a_l^*\}$, $\{b_l^*\}$, and $c^*$ such that $a_1^* \cup \dots \cup b_n^*$ is a simple closed curve meeting $c^*$ in a single endpoint.  We suppose further that this endpoint is $a_1^* \cap b_n^*$.  Moreover, by Proposition \ref{linkbridge}, the bridge splitting $(S^3,L_2) = (B_{12},\A_{12}) \cup (B_{23},\A_{23})$ is standard, and as such we may choose collections $\Delta_{12}$ and $\Delta_{23}$ of bridge disks for $(B_{12},\A_{12})$ and $(B_{23},\A_{23})$ (possibly after relabeling components of the spine of $\T$) with the following properties:
\be
\item $\{a_l^*\} \subset \A_{12}^*$ and $\{b_l^*\} \subset \A_{23}^*$, where $\A_{ij}^* = \Delta_{ij} \cap \Sigma$,
\item $\A_{12}^* \cup \A_{23}^*$ is a collection of pairwise disjoint embedded closed curves in $\Sigma$, and
\item $(\A_{12}^* \cup \A_{23}^*) \cap \text{int}(c^*) = \emp$.
\ee
Condition (3) is obtained by a standard cut-and-paste argument.

Following the proof of Lemma \ref{lemma:trimorse}, we may construct a banded link $(L,\ups)$ such that $\K = \K(L,\ups)$ by the following process:  Let $C_1,\dots,C_m$ denote the components of $\A_{12}^* \cup \A_{23}^*$ in $\Sigma$.  After relabeling, we may suppose that $C_1 = a_1^* \cup \dots \cup b_n^*$.  In addition, we let $\widehat{b}$ denote the arc in $\A_{23}^*$ which meets an endpoint of $c^*$ but is not in $C_1$ (there is precisely one such arc), and suppose $C_2$ contains $\widehat{b}$.  For each other component $C_l$, fix an arc $\widehat{b}_l \in \A_{23}^*$.  Now, let $y^* = \A_{23}^* \setminus \{b_n^*,\widehat{b},\widehat{b}_3,\dots,\widehat{b}_m\}$, and let $\ups$ be a collection of bands for $L = \A_{12} \cup \A_{31}$ induced by arcs $y^* \subset \Sigma$ with the surface framing.  By the proof of Lemma \ref{lemma:trimorse}, $(L,\ups)$ is a banded link and $\K = \K(L,\ups)$.  Moreover, $(S^3,L,\ups) = (B_{12},\A_{12}) \cup_{\Sigma,y^*} (B_{31},\A_{31})$ is a banded bridge splitting we label $\BB$.

We claim that $\BB$ is stabilized.  First, we note that by our choice of $\A_{12}^*$, the bridge sphere $\Sigma$ is perturbed at the point $a_1^* \cap c^*$, and by our choice of $y^*$ and $\A_{23}^*$, the arc $c^*$ does not meet any of the bands in $y^*$.  It will be useful here to consider the bridge sphere $\Sigma$ as fixed, isotoping the link $L$ and the arcs $y^*$.  By an isotopy of $L$, we can unperturb $\Sigma$ to get a new bridge splitting $(S^3,L) = (B_{12},\A_{12}') \cup_{\Sigma} (B_{31},\A_{31}')$.  If $n = 1$, then no arc of $y^*$ meets $a_1^* \cup c^*$ and $y^*$ remains a collection of surface-framed arcs dual to a set of shadows for $\A_{12}'$; thus, the result is again a banded bridge splitting and $\BB$ is stabilized.

Otherwise, $n > 1$ and the only arc of $y^*$ meeting the arc $a_1^* \cup c^*$ is $b_1^*$, where $b_1^*$ meets this arc at the endpoint of $a_1^*$.  Since $y^*$ does not meet the interior of $a_1^* \cup c^*$, unperturbing $\Sigma$ does not disturb the arcs $y^*$ (as shown in Figure \ref{estab}).  We need only observe that there is a collection of shadows $(\A_{12}')^*$ for $\A_{12}'$ so that $(\A_{12}')^* \cup y^*$ is a collection of embedded arcs in $\Sigma'$.  However, the only difference between this set and the embedded arcs of $\A_{12}^* \cup y^*$ is that arc components of $C_1$ and $C_2$ have been joined at their endpoints along $c^*$.  It follows that $(S^3,L,\ups) = (B_{12},\A_{12}') \cup_{\Sigma,y^*} (B_{31},\A_{31}')$ is a banded bridge splitting, and $\BB$ is again stabilized, completing the proof.
\end{proof}

\section{Stable equivalence of bridge trisections}\label{sec:unique}

In this section, we show that there is a sequence of stabilizations and destabilizations connecting any two bridge trisections of the same knotted surface $\K$ in $S^4$.  For this, we will require several new concepts not yet discussed in this paper.  The first is a notion of bridge splittings for compact 1--manifolds embedded in compact 3--manifolds.  The following definitions are closely related to the material presented in Subsection \ref{sub:bridge3}. 

Define a \emph{punctured 3--sphere} to be any 3--manifold obtained from $S^3$ by removing some number of disjoint open 3--balls.  Suppose that $\Sigma$ is a 2--sphere, and define a \emph{compression body} $C$ to a be a product neighborhood $\Sigma \X [0,1]$ of $\Sigma$ with a collection of 2--handles attached to $\Sigma \X \{0\}$.  Note that $C$ is a punctured 3--sphere.  (As an aside, we also note that there is a more general definition of a compression body, but this one will suffice for our purposes.)  Let $\pd_+ C$ denote $\Sigma \X \{1\} \subset \pd C$ and $\pd_- C = \pd C \setminus \pd_+ C$.  We say that a properly embedded arc in $C$ is \emph{$\pd_+$--parallel} if it is isotopic into $\pd_+C$ and \emph{vertical} if is it isotopic to $\{x\} \X [0,1]$ for a point $x \in \Sigma$.  An arc is \emph{trivial} if it is vertical or $\pd_+$--parallel.

We call a properly embedded 1--manifold $T$ in a punctured 3--sphere $B$ a \emph{tangle}.  A \emph{bridge splitting} for a tangle $(B,T)$ is defined as the decomposition
\[ (B,T) = (B_1,\A_1) \cup_{\Sigma} (B_2,\A_2),\]
where $\A_i$ is a collection of trivial arcs in the compression body $B_i$, and $\Sigma = \pd_+ B_1 = \pd_+ B_2$.  We say that $\Sigma$ is a \emph{bridge sphere} for $(B,T)$.  As above, an \emph{elementary perturbation} $\Sigma'$ of $\Sigma$ is obtained by adding a canceling pair of $\pd_+$--parallel arcs to $\A_1$ and $\A_2$, and a surface $\Sigma^*$ which is the result of some number of elementary perturbations performed on $\Sigma$ is called a \emph{perturbation} of $\Sigma$.


In Theorem 2.2 of \cite{zupan:pants}, it is shown that any two bridge splittings for a link $L$ in $S^3$ have a common perturbation.  Although Theorem 2.2 is stated for $B = S^3$, the verbatim proof suffices in the case that $\pd B, \pd T \neq \emp$, and so we do not include it here.  See also \cite{hayashi:stable}.

\begin{theorem}\label{common}\cite{zupan:pants}
Suppose that $\Sigma_1$ and $\Sigma_2$ are bridge splittings for a tangle $T$ in a punctured 3--sphere $B$.  Then there is a surface $\Sigma^*$ which is a perturbation of both surfaces.
\end{theorem}

The other tool we will need in this section is a set of moves which allows us to pass between any two banded link presentations of a knotted surface in $S^4$.  The sufficiency of natural set of moves was conjectured by Yoshikawa in \cite{yoshikawa} and proved by Swenton \cite{swenton} and Kearton-Kurlin \cite{kearton-kurlin}.  The moves are most easily understood by examining Figure \ref{moves}.  We give precise statements of their definitions below.  Although these definitions are cumbersome, they will help to streamline the proof of Theorem \ref{thm:uniqueness}.  For each move, we replace a banded link $(L,\ups)$ with a new banded link $(L',\ups')$.

\begin{figure}[h!]
  \centering
    \includegraphics[width=.46\textwidth]{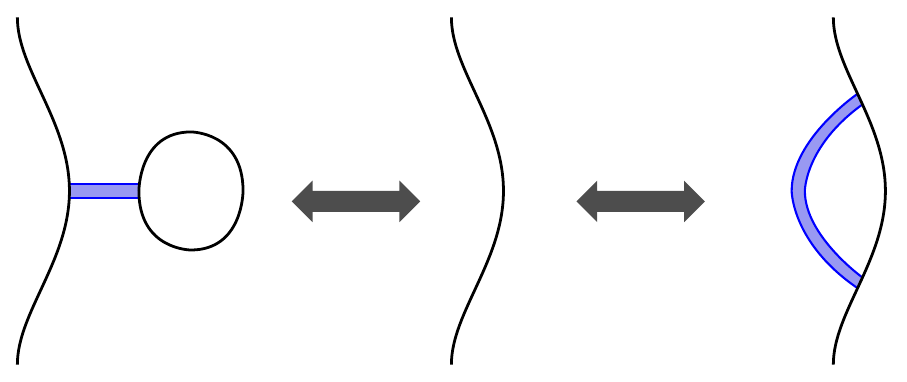} \\   \includegraphics[width=.46\textwidth]{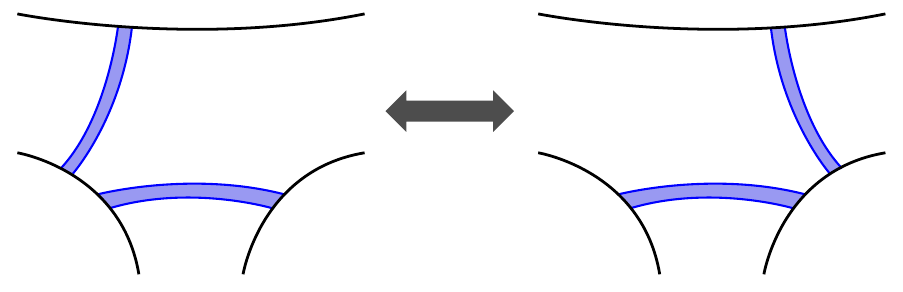}  \,\,\,\,\,\, \,\,\, \includegraphics[width=.46\textwidth]{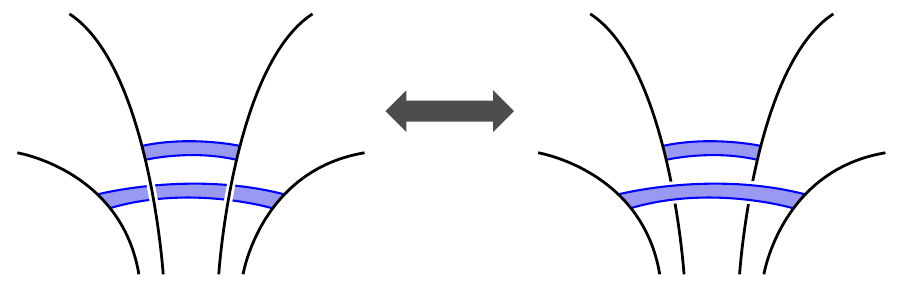}
    \caption{The banded link moves: The cup (top left), cap (top right), band slide (bottom left) and band swim (bottom right).  Note that all of these moves represent isotopies of the corresponding realizing surface $\K(L,\ups)$.}
    \label{moves}
\end{figure}

\bi
\item \emph{Cup}:  Let $L'$ be a split link consisting of $L$ and an unknotted component $J$, and let $\ups'$ be the union of $\ups$ and a trivial band connecting $J$ to $L$. 
\item \emph{Cap}:  Let $L' = L$, and let $\ups'$ be the union of $\ups$ and a band $\ups_*$ such that $L'_{\ups_*}$ is a split link containing $L$ and an unknotted component.
\item \emph{Band slide}:  Let $L' = L$.  Suppose $\ups_1$ and $\ups_2$ in $\ups$ are described by framed arcs $y_1$ and $y_2$ and that $L$ contains an arc $z$ connecting boundary points of $y_1$ and $y_2$ which does not meet $\ups$ in its interior.  Choose a framing on $z$ which is tangent to $y_1$ and $y_2$ at $\pd z$, so that the arc $y_1 \cup z \cup y_2$ has a coherent framing.  Let $y'$ be the push-off of $y_1 \cup z \cup y_2$ along this framing.  Then $y'$ is a framed arc with $\pd y' \subset L$, and we replace $\ups_1$ with a band $\ups'$ corresponding to $y'$ to get a new banded link $(L',\ups')$.
\item \emph{Band swim}:  Let $L' = L$.  Suppose $\ups_1$ and $\ups_2$ in $\ups$ are described by framed arcs $y_1$ and $y_2$, and let $z$ be a framed arc connecting a point in the interior of $y_1$ to a point in the interior of $y_2$ so that the framing of $z$ is tangent to $y_1$ and $y_2$ at $\pd z$ and the framings of $y_1$ and $y_2$ are tangent to $z$ at $\pd z$.  Extend the framing of $z \cup y_2$ to a two-dimensional regular neighborhood $N$, and let $c$ be the curve boundary of $N$.  Then $y_1$ cuts $c$ into $c_1$ and $c_2$, where $c_1$ is isotopic into $y_1$.  In $y_1$, replace $c_1$ with $c_2$ to get a new framed arc $y'$, and replace $y_1$ with $y'$ to obtain a new banded link $(L',\ups')$.
\ei
The definition of band swim given above seems especially awkward; however, this definition will become useful when all of the framed arcs included in the definition are contained in a single surface with the surface framing, in which case the neighborhood $N$ and thus the arc constructed by the band swim are also contained in the surface with the surface framing.  See Figure \ref{swim}.

\begin{figure}[h!]
  \centering
    \includegraphics[width=.8\textwidth]{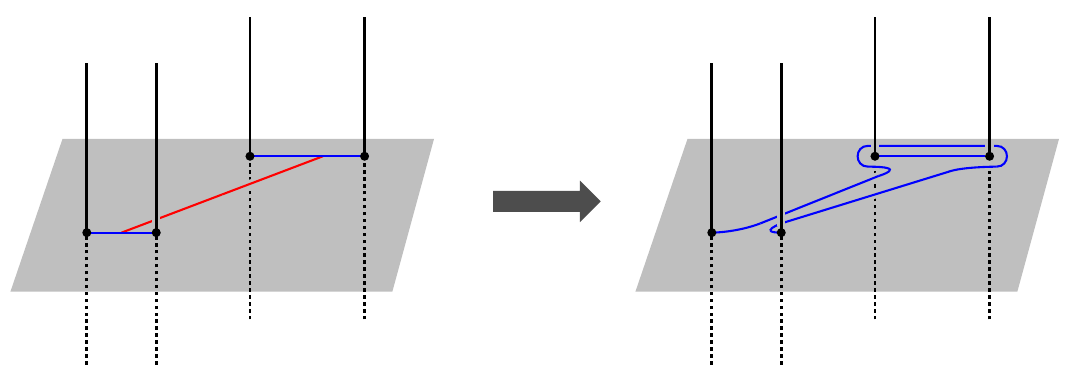}
    \caption{A band swim performed on surface framed arcs.  Only the arcs are pictured; the bands have been suppressed.}
    \label{swim}
\end{figure}

\begin{theorem}\cite{kearton-kurlin, swenton}\label{bandmoves}
If $(L_1,\ups_1)$ and $(L_2,\ups_2)$ are banded links corresponding to two hyperbolic splittings of $(S^4,\K)$, then there is a sequence of cup/cap moves, band slides, and band swims taking $(L_1,\ups_1)$ to a banded link which is isotopic to $(L_2,\ups_2)$ in $S^3$.
\end{theorem}

\begin{remark}
The next theorem may be considered to be a type of Reidemeister-Singer Theorem for bridge trisections.  However, the Reidemeister-Singer Theorem and its various analogues state that two splitting surfaces have a common stabilization.  This seems not to be the case for bridge trisections; rather, the proof of Theorem \ref{thm:uniqueness} reveals that to pass between two trisections $\T$ and $\widehat \T$ of a knotted surface $\K$, it may be necessary to stabilize, destabilize, stabilize, destabilize, etc...  Likewise, we note that two elementary stabilizations of a bridge trisection $\T$ need not be equivalent and also need not commute (for instance, their respective stabilizing disks may intersect), so that the pair need not have a common elementary stabilization.
\end{remark}

\begin{reptheorem}{thm:uniqueness}
Any two bridge trisections of a given pair $(S^4,\K)$ become equivalent after a sequence of stabilizations and destabilizations.\end{reptheorem}

\begin{proof}
Suppose that $\T$ and $\widehat \T$ are two bridge trisections of $(S^4,\K)$, and let $\BB$ and $\widehat \BB$ be banded bridge splittings for banded link presentations $(L,\ups)$ and $(\widehat L,\widehat \ups)$ of $(S^4,\K)$ induced by Lemma \ref{lemma:trimorse}.  By Theorem \ref{bandmoves}, there is a sequence of cup/cap moves, band slides, and band swims taking $(L,\ups)$ to a banded link which is isotopic to $(\widehat L,\widehat \ups)$ in $S^3$.  Thus, it suffices to show that each of these moves may be induced by an appropriate sequence of bridge trisection stabilizations and destabilizations.

Since we will need to stabilize and destabilize the banded bridge splitting $\BB$ numerous times over the course of the proof, we will often abuse notation and preserve that notation for $\BB$ and its components despite that these do, in fact, change under stabilization and destabilization.  We use this convention to limit the unwieldy notation that would result from giving each stabilization and destabilization of $\BB$ a distinct name.

Suppose first that $(L,\ups)$ is related to another banded link $(L',\ups')$ by a single cup move, which may be performed in a small neighborhood of a point $p \in L$.  Let $\BB$ be given by $(S^3,L,\ups) = (B_{12},\A_{12}) \cup_{\Sigma,y^*} (B_{31},\A_{31})$.  By definition of the cup move, the point $p$ is not contained in $y^*$.  Generically, we may also assume that $p \notin \Sigma$ so that $p$ is contained in the interior of an arc $a \in \A_{12}$ or $\A_{13}$.  If one of the endpoints of $a$ does not meet $y^*$, we may slide $p$ along $a$ into $\Sigma$.  Otherwise, we may stabilize $\BB$ toward $B_{12}$ or $B_{31}$ at a point of $\pd a$, after which we may slide $p$ into $\Sigma$.  Now, in a small neighborhood of $p$, we add an unlinked, unknotted component $L_0$ in 1--bridge position to $L$ to get $L'$ and add a single unknotted surface-framed arc $y_0$ connecting $L$ to $L_0$ to the arcs $y^*$ to get a new collection $y'$ which yields the bands $\ups'$.  Let $a_{12}'$ be the arc $L_0 \cap B_{12}$ and let $a_{31}' = L_0 \cap B_{31}$.  Letting $\A_{12}' = \A_{12} \cup \{a_{12}'\}$ and $\A_{31}' = \A_{31} \cup \{a_{31}'\}$, we have $(S^3,L',\ups') = (B_{12},\A_{12}') \cup_{\Sigma,y'} (B_{31},\A_{31}')$ is a also a banded bridge splitting, which we denote $\BB'$.

Observe that $\BB'$ is not a stabilization of $\BB$ as we have defined stabilization for banded bridge splittings (since $\BB$ and $\BB'$ are splittings for distinct banded links); however, we will show that the bridge trisection $\T(\BB')$ given by Lemma \ref{lemma:morsetri} is a stabilization of $\T(\BB)$.  Suppose that a spine of $\T(\BB')$ is given by $(B_{12},\A_{12}') \cup (B_{23},\A_{23}') \cup (B_{31},\A_{31}')$.  Then by construction, there are arcs $a_{12}' \in \A_{12}$ and $a_{31}' \in \A_{31}$ which have bridge disks with identical shadows, and these shadows intersect the arc $y'$ in a single point, where $y'$ is the shadow of a bridge disk for an arc in $\A_{23}'$.  This implies that $\T(\BB')$ is 1--stabilized, and destabilizing results in canceling these three arcs, yielding $\T(\BB)$.  We conclude that any cup move may be achieved by a sequence of bridge trisection stabilizations.

Next, suppose that $(L,\ups)$ is related to $(L,\ups')$ be a single cap move, which can also be assumed to be performed in a small neighborhood of a point $p \in L$, where $p \notin \pd y^*$.  Possibly after stabilizing $\BB$, we may assume that $p$ is contained in arc $a_{31} \subset \A_{31}$.  We may further stabilize $\BB$ along both boundary points of $a_{31}$, stabilizing toward $B_{12}$ if $\pd a_{31}$ meets $\pd y^*$.  After stabilizing, there are arcs $a_{12}$ and $\widehat{a}_{12}$ in $\A_{12}$ such that a union of shadows $a_{12}^* \cup a_{31}^* \cup \widehat{a}_{12}^*$ is an embedded arc and $\pd a_{31}$ does not meet $\pd y^*$.  Let $y_0 = a_{31}^*$ and let $y' = y^* \cup \{y_0\}$.  Then the surface framed arcs $y'$ induce the bands $\ups'$, and the resulting bridge splitting is a banded bridge splitting $\BB'$.

As above, $\BB'$ is not a stabilization of $\BB$ under our rather narrow definition of stabilization for banded bridge splittings.  However, for the corresponding bridge trisection $\T(\BB')$, there is a triple of arcs in the three components of its spine, where these arcs have shadows $y^*$, $a_{31}^* (= y^*)$, and $a_{12}^*$.  It follows that $\T(\BB')$ is 1--stabilized, and destabilizing cancels this triple of arcs, yielding the original bridge trisection $\T(\BB)$.

Suppose now that $(L,\ups)$ is related to $(L,\ups')$ by a single band swim given by a framed arc $z$ whose endpoints are interior points of surface-framed arcs $y_1^*$ and $y_2^*$ in $y^*$.  Let $g\colon(S^3,L) \rightarrow \R$ be a Morse function such that $g^{-1}(0) = \Sigma$.  We claim that there is an isotopy of $g$ which fixes $y^*$ and yields a stabilization of the banded bridge splitting surface $\Sigma$ such that the flow of $g$ projects $z$ onto a properly embedded surface-framed arc $z^*$.  As mentioned above, and as in the proof of Theorem  \ref{thm:existence}, we will abuse notation and let $g$ and $\Sigma$ denote the result of a specified isotopy.

Let $\pi_g(z)$ denote the projection of $z$ to $\Sigma$ given by the flow of $g$, where $g$ has been suitably isotoped so that this projection avoids the bridge arcs $\A_{12}$ and $\A_{31}$ as in the proof of Theorem \ref{thm:existence} (see Figure \ref{offbridge}).  By stabilizing, we may remove crossings of $\pi_g(z)$ and $y_i^*$ (Figure \ref{cross1}) and self-crossings of $\pi_g(z)$ (Figure \ref{cross2}), while $y _i^*$ remains fixed.

\begin{figure}[h!]
  \centering
    \includegraphics[width=1\textwidth]{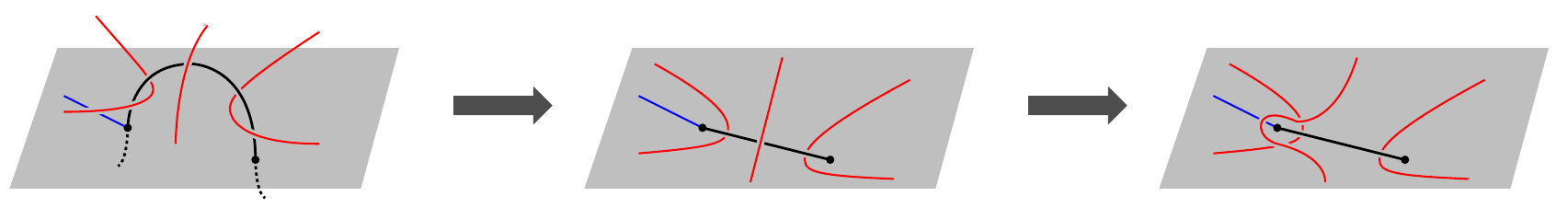}
    \caption{After isotopy, we may project $z$ onto $\pi_g(z) \subset \Sigma$ avoiding the bridge arcs $\A_{12}$ and $\A_{31}$.}
    \label{offbridge}
\end{figure}

\begin{figure}[h!]
  \centering
    \includegraphics[width=.8\textwidth]{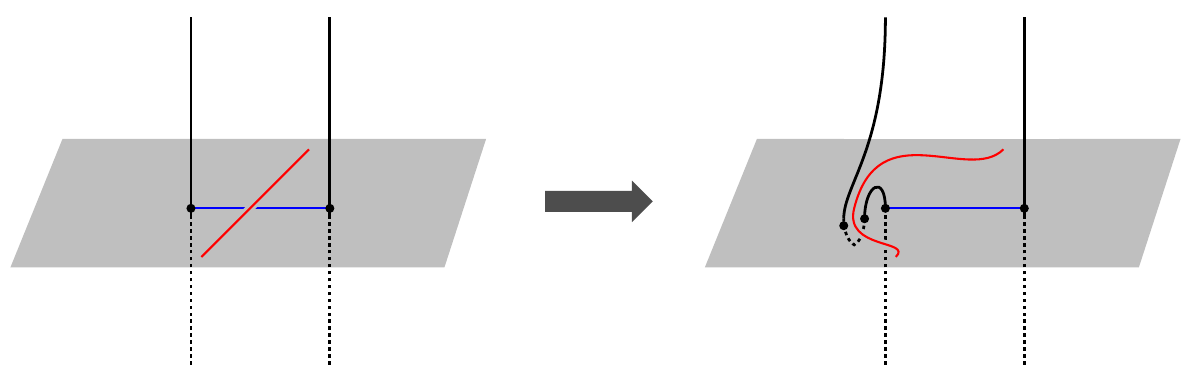}
    \caption{By stabilizing, we may remove crossings of $\pi_g(z)$ with the arcs $y_i^*$.}
    \label{cross1}
\end{figure}

\begin{figure}[h!]
  \centering
    \includegraphics[width=.8\textwidth]{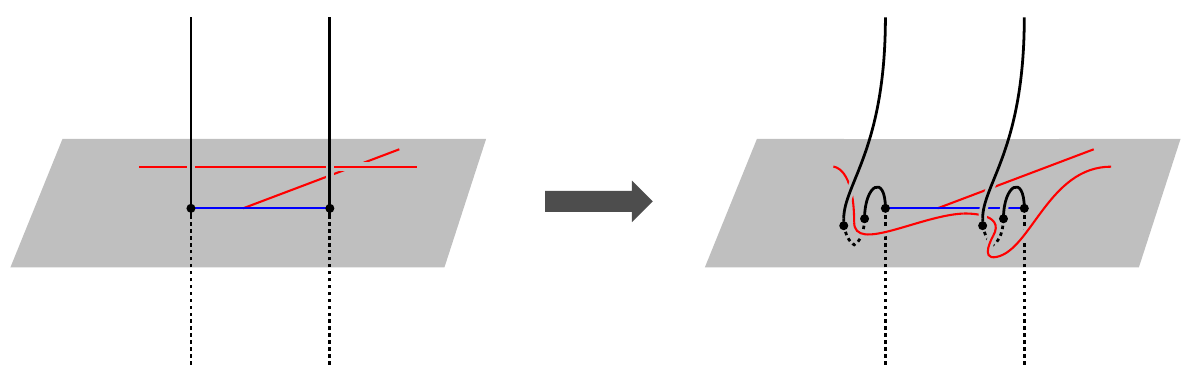}
    \caption{We eliminate self-crossings of $\pi_g(z)$ by stabilizing.}
    \label{cross2}
\end{figure}

\begin{figure}[h!]
  \centering
    \includegraphics[width=1\textwidth]{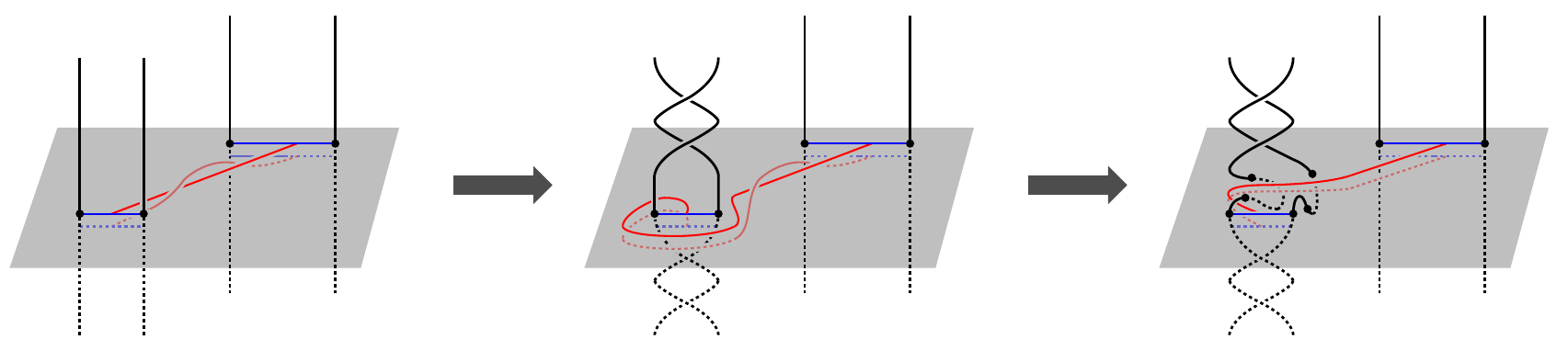}
    \caption{By perturbing $\Sigma$ and isotoping the arc $\pi_g(z)$, we may change the framing of $\pi_g(z)$ induced by $\Sigma$.}
    \label{framing}
\end{figure}

Thus, after performing some number of stabilizations, we may assume that the arc $\pi_g(z)$ is embedded in $\Sigma$.  A priori, the framing of $\pi_g(z)$ may not agree with the surface framing, but after stabilizing further, there is an isotopy which carries $\pi_g(z)$ off of $\Sigma$ and returns it with framing which winds once around the original framing.  See Figure \ref{framing}.  This may be achieved by an isotopy of $g$, and after a finite number of repetitions, we may assume that $z^* = \pi_g(z)$ is embedded in $\Sigma$ with framing given by $\Sigma$, and $\text{int}(z^*) \cap y^* = \emp$.

If necessary, we next stabilize $\Sigma$ near each point of $\pd y_1^*$, so that the arc $y_1^*$ is adjacent to two bridge disks $\Delta$ and $\Delta'$ in $\A_{12}$ which do not intersect any other arcs in $y^*$.  Observe that none of the isotopies performed disturb the arcs $y^*$; hence, the resulting decomposition $\BB'$, which we label $(S^3,L,\ups) = (B_{12}',\A_{12}') \cup_{\Sigma',y^*} (B_{31}',\A_{31}')$, is a stabilization of our original $\BB$.  We may perform the band swim specified by $z$ by replacing $y_1^*$ with a surface-framed arc $y'_1$ as in Figure \ref{swim}. We let $y' = y^* \setminus \{y_1^*\} \cup \{y'_1\}$, and we let $\ups'$ denote the set of bands for $L$ induced by $y'$.

We claim that the decomposition $\BB''$ given by $(S^3,L,\ups') = (B_{12}',\A_{12}') \cup_{\Sigma',y'} (B_{31}',\A_{31}')$ is a banded bridge splitting.  The only property we must verify is that there is a set of shadows for $\A_{12}'$ dual to $y'$, which is accomplished (as in the proof of Theorem \ref{thm:existence}) by a standard cut-and-paste argument involving a set of bridge disks yielding traces dual to $y^*$ and the disk arising from taking the frontier of $\pd \nu(\Delta \cup y'_1 \cup \Delta')$ in $B_{12}'$.

Finally, we claim that the induced bridge trisections $\T(\BB')$ and $\T(\BB'')$ are isotopic.  To see this, note that the band swim represents an isotopy of $\K$ which takes place in a 4--dimensional regular neighborhood of $y_1^* \cup z^* \cup y_2^*$.  After pushing $y_1^* \cup z^* \cup y_2^*$ slightly into the interior of $B_{12}'$ as in the proof of Lemma \ref{lemma:morsetri}, we see that, for the induced trisection $\T(\BB')$, the band swim is realized by an isotopy of $\D_2$ supported in $\text{int}(X_2)$, and thus the result is an isotopic bridge trisection.  We conclude in this case that band swims may be achieved by bridge trisection stabilization.

The proof that band slides can be realized through stabilizations and destabilizations mirrors the proof above for band swims; however, in this case, we require both the stabilization and destabilization operations, and so the argument is more complicated.  Suppose that $(L,\ups)$ is related to $(L,\ups')$ by a single band slide performed along a framed arc $z \subset L$ which connects arcs surface-framed $y_1^*$ and $y_2^*$ in $y^*$.  First, we may arrange so that $z$ meets $y_1^*$ and $y_2^*$ via arcs contained in $B_{12}$ as follows:  Stabilize $\Sigma$ at both endpoints of $\pd y_i^*$ so that $B_{12}$ contains bridge disks $\Delta_1$ and $\Delta_2$ which meet the arcs $y^*$ only in $\pd y_i^*$, as shown in Figure \ref{alphafix}. Then $y_i^*$ may be pushed off of $\Sigma$, over $\Delta_1$ and $\Delta_2$, and back onto $\Sigma$ via an isotopy which reverses the direction of the framing, but which still results in $y_i^* \subset \Sigma$ having the surface framing.  We call this process \emph{flipping} the framed arc $y_i^*$, and we note that flipping may achieved by an isotopy of $g$.

\begin{figure}[h!]
  \centering
    \includegraphics[width=1\textwidth]{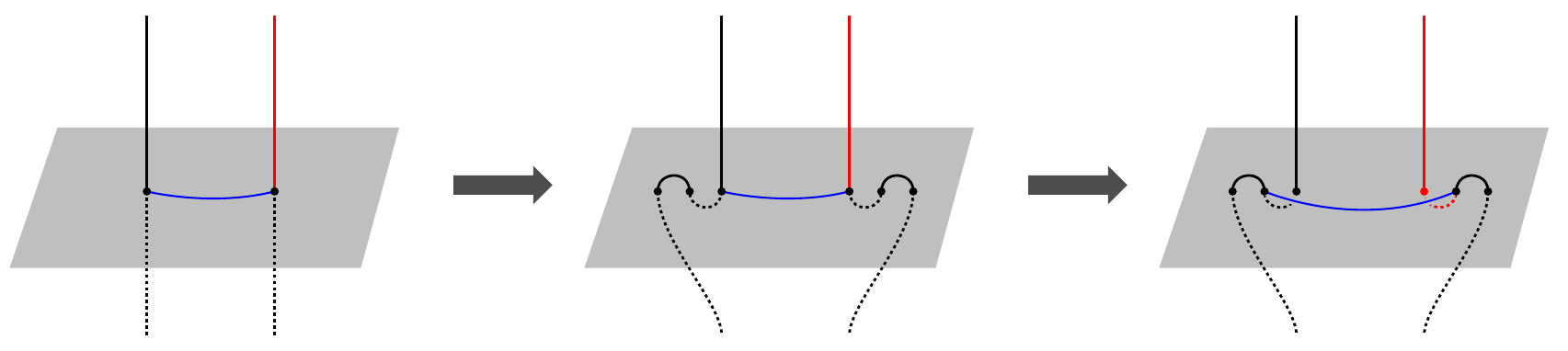}
    \caption{We may perturb $\Sigma$ near $\pd y_i^*$ and push $y_i^*$ off of and back onto $\Sigma$ so that $z$ meets $\pd y_i^*$ via an arc in $B_{12}$.}
    \label{alphafix}
\end{figure}

Since $z$ is a union of arcs in $\A_{12}$ and $\A_{31}$, we use the flow of $g$ to project $z$ onto an immersed arc $\pi_g(z) \subset \Sigma$ away from the other arcs of $\A_{12}$ and $\A_{31}$.  By the definition of a band slide, the interior of $z$ does not contain a boundary point of an arc in $y$, and since we began with a banded bridge splitting, $\pi_g(z)$ meets $y^*$ only in its endpoints.  Thus, the projection has only possible self-intersections.  We may get rid of self-intersections by stabilizations as in Figure \ref{cross2} above, so that after some number of stabilizations, the arc $\pi_g(z)$ is embedded in $\Sigma$.  It is possible that $\pi_g(z)$ has framing which disagrees with the surface framing; however, after stabilizing further as shown in Figure \ref{framing}, we may arrange so that the framed arc $z$ is isotopic to the surface-framed arc $\pi_g(z)$ in $\Sigma$.

If not done already, stabilize $\Sigma$ at the endpoints of $y_1^*$ and $y_2^*$ not contained in $z$ as shown at left in Figure \ref{bslide}. Since $z$ meets $y_1^*$ via arcs in $B_{12}$, we have that $g_z$ has $n$ minima and $n-1$ maxima for some $n > 0$.  Furthermore, there is a collection $\Delta \subset B_{12}$ of $n$ bridge disks and a collection $\Delta' \subset B_{31}$ of $n-1$ bridge disks such that $(\Delta \cup \Delta') \cap \Sigma = \pi_g(z)$.  Thus, we may unperturb $\Sigma$ precisely $n-1$ times until $g_z$ has a single minimum and no maxima; that is to say, $z$ is a single arc in $\A_{12}$ isotopic to a surface-framed arc $z^* = \pi_g(z) \subset \Sigma$.  In addition, our stabilizations of $\Sigma$ at $\pd y_i^*$ guarantee that each unperturbation remains a banded bridge splitting, so that these unperturbations realize $n-1$ destabilizations.

\begin{figure}[h!]
  \centering
    \includegraphics[width=1\textwidth]{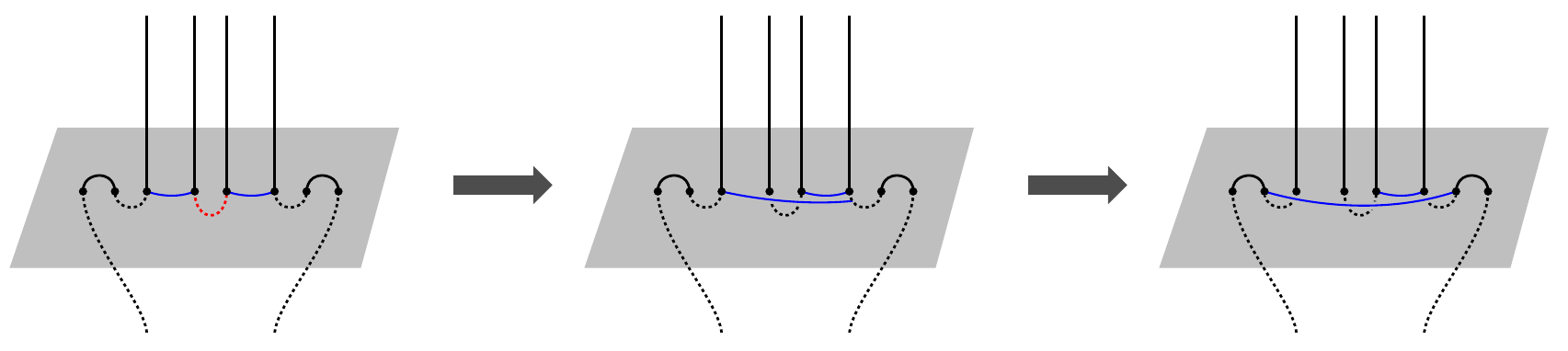}
    \caption{After perturbing $\Sigma$ (left), we may perform the band slide (middle) and flip the resulting band (right) to see that the result of the slide is a banded bridge splitting.}
    \label{bslide}
\end{figure}

As above, none of the isotopies, stabilizations, and destabilizations disturb the arcs $y^*$, and thus the result is a banded bridge splitting $\BB'$ which is stably equivalent to our $\BB$.  Now, we may perform the band slide of $y_1^*$ over $y_2^*$ along $z$ so that the resulting band $y_1'$ is isotopic into $\Sigma$ with the surface framing (ignoring $y_2^*)$. However, this isotopy pushes a point of $\pd y_1'$ onto a point of $\pd y_2^*$ (as in the middle frame of Figure \ref{bslide}).  We resolve this issue by flipping $y_1'$ (as in the right frame of Figure \ref{bslide}), the result of which is a banded bridge splitting $\BB''$ obtained by replacing $y_1^*$ with $y_1'$ in $y^*$.  Finally, as in the case for band swims, we claim that the trisections $\T(\BB')$ and $\T(\BB'')$ are equivalent:  By construction the band slide may be realized by an isotopy of $\K$ supported in the interior of $X_2$, and as such, it preserves the induced bridge trisection.

So far, we have shown that all of the moves of Theorem \ref{bandmoves} may be carried out by stabilizing and destabilizing the banded bridge splitting $\BB$.  Suppose now that the banded links $(L,\ups)$ and $(\widehat L,\widehat \ups)$ for $\BB$ and $\widehat \BB$ are isotopic in $S^3$.  Since the corresponding bridge surfaces $\Sigma$ and $\widehat \Sigma$ may be distinct, there is still work left to do.

Once again, we remedy the situation by stabilizing.  Let $\BB$ be given by $(S^3,L,\ups) = (B_{12},\A_{12}) \cup_{\Sigma,y^*} (B_{31},\A_{31})$ and let $\widehat \BB$ be given by $(S^3,\widehat L,\widehat \ups) = (\widehat B_{12},\widehat \A_{12}) \cup_{\widehat \Sigma,\widehat y} (\widehat B_{31},\widehat \A_{31})$.  We will consider the banded link $(L,\ups)$ to be fixed, so that $(L_,\ups) = (\widehat L,\widehat \ups)$, and let $T$ denote the tangle $T = L \setminus \nu(\ups)$ in $S^3 \setminus \nu(\ups)$.  Stabilize $\Sigma$ enough times so that each arc of $\A_{12}$ is incident to only one arc in $y^*$.  Now, since the arcs $y^*$ are dual to a set of shadows for $\A_{12}$, pushing $\ups$ into the interior of $B_{12}$ and removing $\nu(\ups)$ yields that $\Sigma$ is a bridge surface for $T$ in $S^3 \setminus \nu(\ups)$.  Similarly, after sufficiently many stabilizations, the surface $\widehat \Sigma$ becomes a bridge surface for $T$ in $S^3 \setminus \nu(\ups)$ after pushing $\ups$ into the interior of $\widehat B_{12}$.  By Theorem \ref{common}, there is a bridge surface $\Sigma_*$ for $T$ in $S^3 \setminus \nu(\ups)$ which is a perturbation of both $\Sigma$ and $\widehat \Sigma$.  By gluing the bands $\ups$ back into $T$, we may construct a banded bridge splitting $\BB^*$ with bridge surface $\Sigma^*$ such that $\BB^*$ is a stabilization of both $\BB$ and $\widehat \BB$.

We conclude that any two bridge trisections of $\K$ are related by a sequence of stabilizations and destabilizations, as desired.
\end{proof}

\subsection{Tri-plane moves}\label{subsec:moves}\ 

In this section, we use Theorem \ref{thm:uniqueness} to show that a calculus of moves is sufficient to pass between any two tri-plane diagrams for a fixed knotted surface $\K$ in $S^4$. 

\begin{lemma}\label{lemma:isotopymoves}
Suppose that $\Pau$ and $\widehat\Pau$ are two tri-plane diagrams for equivalent bridge trisections $\T$ and $\widehat\T$ of a knotted surface $\K$ in $S^4$.  Then $\Pau$ and $\widehat\Pau$ are related by a sequence of interior Reidemeister moves and mutual braid transpositions.
\end{lemma}

\begin{proof}
If $\T$ and $\widehat\T$ are equivalent bridge trisections, then their respective spines are isotopic in $S^4$.  Thus, we may suppose that $\Pau$ and $\widehat\Pau$ are tri-plane diagrams for a \emph{fixed} bridge trisection $\T$.  In this case, there is a spine $\Ss = (B_{12},\A_{12}) \cup (B_{23},\A_{23}) \cup (B_{31},\A_{31})$ for $\T$ and isotopic tri-planes $E$ and $\widehat E$ in $\Ss$, projections onto which yield $\Pau$ and $\widehat \Pau$, respectively.  Let $E = E_{12} \cup E_{23} \cup E_{31}$ and $\widehat E = \widehat E_{12} \cup \widehat E_{23} \cup \widehat E_{31}$.  Then for each pair $(i,j)$ of indices, there is an isotopy taking $E_{ij}$ to $\widehat E_{ij}$, and these three isotopies agree on $\Sigma$.

Therefore, we may change our perspective further:  Instead of fixing the tangles and moving the tri-plane, we could just as well fix the tri-plane and move the tangles.  Thus, suppose that $\A_{ij}$ is isotopic to $\widehat \A_{ij}$ in $B_{ij}$, where the three isotopies agree on $\Sigma$ and the three projections of $\widehat \A_{ij}$ onto $E_{ij}$ induce the tri-plane diagram $\widehat\Pau$.  We will let $\Gamma$ denote $\A_{12} \cup \A_{23} \cup \A_{31}$ as an abstract topological 1--complex, and let $V \subset \Gamma$ denote the 0--skeleton of $\Gamma$; that is, the points $\pd \A_{12} = \pd \A_{23} = \pd \A_{31}$.  By the discussion above, there is a continuous family of embeddings $f_t\colon \Gamma \hookrightarrow \Ss$ such that $f_0(\Gamma) = \A_{12} \cup \A_{23} \cup \A_{31}$, $f_1(\Gamma) = \widehat\A_{12} \cup \widehat\A_{23} \cup \widehat\A_{31}$, and $f_t(V) \subset \Sigma$ for all $t$.

We may also assume that $\A_{ij}$ and $\widehat \A_{ij}$ meet $\Sigma$ in the same set of points, which are contained in the equator $e = \pd E_{ij}$ of the tri-plane $E$.  This implies that $f_0(V) = f_1(V)$, so that $f_t|_V$ is a loop in the configuration space of $2b$ points; that is, the restriction of $f_t$ to $V$ is a braid $\sigma$ of $2b$ points in $\Sigma$.  Let $g_t\colon \Gamma \rightarrow \Ss$ be a continuous family of embeddings that agrees with $f_t$ on $\Sigma$ and is the constant map $g_t(\Gamma) = \A_{12} \cup \A_{23} \cup \A_{31}$ away from a small neighborhood of $\Sigma$, so that $g_0(\Gamma) = f_0(\Gamma)$.  In addition, let $\Gamma_{ij} \subset \Gamma$ such that $g_0(\Gamma_{ij}) = \A_{ij}$, and let $\A_{ij}' = g_1(\Gamma_{ij})$.  Then there is a tri-plane diagram $\Pau'$ given by the projection of $\A_{ij}'$ to $E_{ij}$ such that $\Pau$ and $\Pau'$ are related by a sequence of mutual braid transpositions (namely, $\sigma$).

Finally, let $h_t = -g_t + f_t$.  More specifically,
\[ h_t =\begin{cases}
		g_{1-2t}		&\text{if $t \in [0,1/2]$}\\
		f_{2t-1}			&\text{if $t \in [1/2,1]$.}
	\end{cases}\]
The restriction of $h_t$ to $V$ is isotopic to the identity in $\Sigma$, and thus there is an isotopic family of embeddings $h^*_t$ such that $h^*_0 = h_0 = g_1$, $h^*_1 = h_1 = f_1$, and $h^*_t$ restricted to $V$ \emph{is} the identity.  As such it follows that $\A_{ij}'$ is isotopic to $\widehat\A_{ij}$ rel boundary, which implies that $\Pau'_{ij}$ is related to $\widehat\Pau_{ij}$ by a sequence of interior Reidemeister moves.  We conclude that $\Pau$ and $\widehat\Pau$ are related by mutual braid transpositions and interior Reidemeister moves, as desired.
\end{proof}

\begin{reptheorem}{thm:tri-planemoves}
Let $\Pau$ and $\widehat\Pau$  be two tri-plane diagrams for a knotted surface $\K\subset S^4$.  Then $\Pau$ and $\widehat\Pau$ are related by a sequence of tri-plane moves.
\end{reptheorem}

\begin{proof}
Let $\T$ and $\widehat \T$ be the bridge trisections corresponding to $\Pau$ and $\widehat\Pau$.  By Theorem~\ref{thm:uniqueness}, $\T$ and $\widehat \T$ are stably equivalent.  Thus, it suffices to show that if $\T$ and $\widehat \T$ are related by a single stabilization or destabilization, then $\Pau$ and $\widehat \Pau$ are related by a sequence of tri-plane moves. 

Following the discussion in Section \ref{sec:stab}, a stabilization or destabilization of $\T$ may be carried out diagrammatically; that is, there are tri-plane diagrams $\Pau'$ and $\widehat \Pau'$ for $\T$ and $\widehat \T$, respectively, so that $\Pau'$ is related to $\widehat \Pau'$ by a stabilization or destabilization move.  As $\Pau$ and $\Pau'$ are diagrams for the same bridge trisection $\T$, Lemma \ref{lemma:isotopymoves} asserts that they are related by interior Reidemeister moves and mutual braid transpositions.  Similarly, $\widehat\Pau'$ is related to $\widehat\Pau$ by interior Reidemeister moves and mutual braid transpositions.  Therefore, we may pass from $\Pau$ to $\widehat \Pau$ by a sequence of tri-plane moves, as desired.
\end{proof}


\bibliographystyle{acm}
\bibliography{BridgeTrisectionBiblio.bib}

\end{document}